%% file: EGS2022_lessinteraction.tex
\newcommand{\classstyle}{1}

\ifnum\classstyle=0
\documentclass[a4paper,12pt,twoside,english]{article}
\input{packages}

\usepackage{wiaspreprint}

\fi
\ifnum\classstyle=1
\documentclass[12pt]{article}
\input{packages}
\usepackage{arxiv}
\fi

\ifnum\classstyle=2
\documentclass[review,onefignum,onetabnum]{siamonline220329}
\input{packages_siads}
\usepackage{amsfonts}
\usepackage{graphicx}
\usepackage{epstopdf}
\usepackage{algorithmic}
\ifpdf
  \DeclareGraphicsExtensions{.eps,.pdf,.png,.jpg}
\else
  \DeclareGraphicsExtensions{.eps}
\fi
\usepackage{enumitem}
\setlist[enumerate]{leftmargin=.5in}
\setlist[itemize]{leftmargin=.5in}

\usepackage{amsopn}

\ifpdf
\hypersetup{
  pdftitle={Less Interaction with forward models in Langevin Dynamics},
  pdfauthor={M. Eigel, R. Gruhlke, D. Sommer}
}
\fi
\fi

\input{info}

\ifnum\classstyle=0
\input{titlepage_wias}
\fi
\ifnum\classstyle=1
\input{titlepage}
\fi
\ifnum\classstyle=2
\input{titlepage_siads}
\fi

\begin{document}
\maketitle

\begin{abstract}
\input{abstract}
\end{abstract}

\ifnum\classstyle=1
\fi

\ifnum\classstyle=2
\fi


\input{01_introduction.tex}
\input{02_problem.tex}
\input{03_method.tex}
\input{04_error_analysis}
\input{05_numerics.tex}
\input{06_conclusion}
\input{07_acknowledgements.tex}

\ifnum\classstyle=0
\bibliographystyle{abbrv}
\fi
\ifnum\classstyle=1
\bibliographystyle{abbrv}
\fi
\ifnum\classstyle=2
\bibliographystyle{abbrv}
\fi
\bibliography{ref}

\newpage
\appendix

\input{08_appendix_theory.tex}

\end{document}

%% file: packages.tex
\usepackage[utf8]{inputenc}
\usepackage{times}
\usepackage[normalem]{ulem}
\usepackage{graphicx,amsmath,amssymb, mathtools,enumitem}

\usepackage{bm}
\usepackage{tcolorbox}
\usepackage{amsthm}
\usepackage{mathrsfs}  

\usepackage{amsfonts}
\numberwithin{equation}{section}
\newtheorem{theorem}{Theorem}[section]
\newtheorem{definition}{Definition}[section]

\newtheorem{corollary}{Corollary}[section]
\newtheorem{remark}{Remark}[section]
\newtheorem{lemma}{Lemma}[section]

\newtheorem{assumption}{Assumption}[section]

\usepackage[hypertexnames=false]{hyperref}
\hypersetup{
    colorlinks = true,
    linkbordercolor = {white},
    linkcolor = {blue!50!black},
citecolor     = {blue!50!black},
}

\usepackage{caption}
\usepackage{subcaption}

\usepackage{algpseudocode}
\usepackage{tikz}

\usepackage{easy-todo} 

\usepackage{chngcntr}
\counterwithin{figure}{section}

\usepackage{ulem}

\usepackage{yhmath}

\DeclareMathOperator*{\argmax}{arg\,max}
\DeclareMathOperator*{\argmin}{arg\,min}

\newcommand{\apprxd}{\,{\buildrel d \over \approx}\,}

\makeatletter
\@namedef{subjclassname@2020}{%
  \textup{2020} Mathematics Subject Classification}
\makeatother

%% file: info.tex
\newcommand{\ownKeywords}{Langevin dynamics, interacting particle systems, Bayesian inference, computational optimal transport, Wasserstein distance, mean-field Fokker-Planck equation, homotopy}

\newcommand{\ownAMS}{
  65N21, 
  62F15,   	
  65N75,   	
  65C30,   	
  35Q84,   	
  65K10
}

\newcommand{\ownThanks}{We thank Christian Bayer, Paul Hagemann, Matthias Liero and Claudia Schillings for fruitful discussions on the topics of this paper.
ME acknowledges the partial support by the DFG SPP 2998 ``Theoretical Foundations of Deep Learning''. RG acknowledges the support by the DFG SPP 1886 ``Polymorphic Uncertainty Modelling for the Numerical Design of Structures''.
DS acknowledges support by the ProFIT project \textquotedblleft ReLkat – Reinforcement Learning for complex automation engineering\textquotedblright. This study does not have any conflicts to disclose.}

\hypersetup{
	breaklinks,
	colorlinks,
	linkcolor=blue,
	citecolor=blue,
	urlcolor=blue,
    pdftitle={Less Interaction with forward models in Langevin Dynamics},
    pdfsubject={math.NA, cs.LG},
    pdfauthor={Robert Gruhlke},
    pdfkeywords={Langevin dynamics, interacting particle systems, Bayesian inference, computational optimal transport, Wasserstein distance, Mean-field Fokker-Planck equation, homotopy, generative models},
}

%% file: titlepage.tex
\title{Less Interaction with forward models in Langevin Dynamics}

\author{
	\textbf{Martin Eigel} \\
	Weierstrass Institute for \\ Applied Analysis  and Stochastics \\
	Berlin, Germany \\
	\texttt{eigel@wias-berlin.de}
	\And
	\textbf{Robert Gruhlke} \\
	Weierstrass Institute for \\ Applied Analysis  and Stochastics \\
	Berlin, Germany \\
	\texttt{gruhlke@wias-berlin.de}
	\And
	\textbf{David Sommer} \\
	Weierstrass Institute for \\ Applied Analysis  and Stochastics \\
	Berlin, Germany \\
	\texttt{sommer@wias-berlin.de} 
}

\renewcommand{\shorttitle}{Less Interaction with forward models in Langevin Dynamics}

%% file: abstract.tex
Ensemble methods have become ubiquitous for the solution of Bayesian inference problems. State-of-the-art Langevin samplers such as the Ensemble Kalman Sampler (EKS), Affine Invariant Langevin Dynamics (ALDI) or its extension using weighted covariance estimates rely on successive evaluations of the forward model or its gradient. 
A main drawback of these methods hence is their vast number of required forward calls as well as their possible lack of convergence in the case of more involved posterior measures such as multimodal distributions. 
The goal of this paper is to address these challenges to some extend. First, several possible adaptive ensemble enrichment strategies that successively enlarge the number of particles in the underlying Langevin dynamics are discusses that in turn lead to a significant reduction of the total number of forward calls. Second, analytical consistency guarantees of the ensemble enrichment method are provided for linear forward models. Third, to address more involved target distributions, the method is extended by applying adapted Langevin dynamics based on a homotopy formalism for which convergence is proved. 
Finally, numerical investigations of several benchmark problems illustrates the possible gain of the proposed method, comparing it to state-of-the-art Langevin samplers.

%% file: 01_introduction.tex
\section{Introduction}

Consider the inverse problem of finding an unknown $y\in\mathbb{R}^{D}$ from an observation $\Delta\in\mathbb{R}^K$ for $D,K\in\mathbb{N}$, where
\begin{equation}
    \Delta = \mathcal{G}(y) + \eta,
    \label{eq:inverseproblem}
\end{equation}
with a deterministic forward operator $\mathcal{G}\colon\mathbb{R}^{D}\to\mathbb{R}^K$ and centered Gaussian observational noise $\eta\sim\mathcal{N}(0,\Gamma)$ with positive definite covariance matrix $\Gamma\in\mathbb{R}^{K,K}$.
In the Bayesian framework, a prior distribution $\mu_{\mathrm{prior}}$ is associated with the unknown $y$.
Then, for a given measurement $\Delta = \tilde{\Delta}$, the prior is updated via Bayes' rule to yield a posterior distribution $\mu_{\ast}$~\cite{kaipio2006statistical}.  
In case that $\mu_{\mathrm{prior}}$ has a Lebesgue density $\pi_{\mathrm{prior}}$, then under mild assumptions~\cite{stuart2010inverse} there exists a posterior density $\pi_\ast$, which is the Lebesgue density of $\mu_\ast$ given by
\begin{equation}\label{eq:posterior1}
    \pi_\ast(y) \propto \exp(-L(y))\pi_{\mathrm{prior}}(y),
\end{equation}
with the log-likelihood potential $L(y) := 1/2|\tilde{\Delta}-\mathcal{G}(y)|_\Gamma^2$. Here, $|\,.\,|$ denotes the standard Euclidean norm and $|\,.\,|_{M} = | M^{-1/2}\,.\, |$  for any symmetric positive definite $M$. If the prior density is positive on $\mathbb{R}^{D}$, we define the potential 
\begin{equation}\label{eq:potential}
    \Phi(y) := L(y) - \log \pi_{\mathrm{prior}}(y),
\end{equation}
such that the posterior density becomes
\begin{equation}\label{eq:posterior}
    \pi_\ast(y)\propto \exp(-\Phi(y)).
\end{equation}


In this work we are concerned with the problem of sampling from the posterior distribution given by~\eqref{eq:posterior} using interacting particle methods based on Langevin dynamics.
While non-interacting particle systems exhibit slow convergence in time, their extensions to interacting particle system such as ALDI and the EKS have demonstrated a superior convergence speed.
However, such state-of-the-art methods still require a vast number of forward calls to solve the underlying model.
Furthermore, convergence to more involved posteriors such as multimodal distributions may become arbitrarily slow or cannot be guaranteed at all.
To address these drawbacks, we propose to extend existing methods by the following two strategies.

The first concept is \textit{ensemble enrichment}. It allows to work with ensembles of small batch sizes for large parts of the process, utilizing the contained information to build larger ensembles with the desired distribution at a later time.
This substantially reduces the number of necessary calls of the forward model $\mathcal{G}(\cdot)$.

The second concept is based on the notion of \textit{homotopy}.
Here, instead of directly working with a particle system based on the posterior, we utilize intermediate measures obtained from interpolation between a simple auxiliary measure and the posterior.
Such a preconditioning of the particle ensemble potentially increases the convergence speed over time significantly, especially for multimodal distributions.
%
We coin this new methodology LIDL\footnote{An acronym meaning \textit{\textbf{L}ess \textbf{I}nteraction with forward models in \textbf{L}angevin \textbf{D}ynamics}.
It should be noted that the first author has objections against this term but was overruled by majority vote. 
}.



\subsection{Related work}
There is a vast amount of literature on different sampling methods such as Markov Chain Monte Carlo methods (MCMC)~\cite{roberts2004general,brooks2011handbook,robert2011mcmc} and more recently methods based on Langevin dynamics~\cite{roberts1996exponential} or Stein variational gradient descent~\cite{liu2016stein}, to name just a few.
While the idea of ensemble enrichment can in principle be deployed for any ensemble sampler, in this work we focus on the class of Langevin based samplers.
This is due to the fact that Langevin samplers lend themselves quite naturally to ensemble enrichment schemes like \textit{slicing} (adding together ensembles defined by the Langevin process at different time points) and appropriately scaled \textit{random kicks} (adding noisy duplicates of the existing batch, such that the covariance structure is preserved).

In analytical chemistry, the term \textit{sample enrichment} is common to denote certain ways of manipulating samples to give them desirable properties \cite{stalter2016enrichment,sample_enrich2019}.
In order to differentiate from this term and to emphasize the connection to particle ensemble methods, we coin our form of enrichment \textit{ensemble enrichment}.
In this work, we understand ensemble enrichment as the addition of new samples $\{y^i\}_{i=n+1}^N$ to an existing sample batch $\{y^i\}_{i=1}^n$.
In a very general setting, an enrichment can be considered admissible if both sample batches are drawn independently from the same distribution.
Depending on the context however, weaker conditions may be sufficient.
In the setting of Langevin sampling considered here, an enrichment is suitable if the distance from the posterior to the enriched sample batch  can be bounded 
by the respective distance of the batch before enrichment. 

The use of Langevin dynamics in Bayesian inference requires the posterior to be an invariant measure of the chosen dynamics.
Historically, the most basic dynamics with this property is given by the first order overdamped Langevin equation \cite{Pavliotis:2008aa}.
As a state-of-the-art Langevin method, we use the Affine Invariant Langevin Dynamics (ALDI) \cite{garbuno2020affine} in this work.
ALDI is a modification of the Ensemble Kalman Sampler (EKS) \cite{reich2021fokker}, which ensures affine invariance \cite{goodman2010ensemble} as well as convergence in total variation to the posterior, even with a finite number of particles.
Ensemble Langevin methods such as EKS and ALDI have strong links to ensemble Kalman filters (EKF) \cite{evensen2006kalman, law2015data, reich2015probabilistic} and ensemble Kalman inversion (EKI) \cite{iglesias2013ensemble}, which was pointed out in \cite{reich2021fokker}.
A recent survey of ensemble Kalman methods and their application to Bayesian inverse problems can be found in \cite{reich2022kalman}.
The analysis of EKF and EKI for Bayesian inversion has been extensively studied \cite{schillings2017analysis,schillings2018convergence, blomker2019well,ding2021ensemble}.
In \cite{DingLi21} convergence of the EKS to the limiting Fokker-Planck equation in expected 2-Wasserstein distance is shown.
To the best of our knowledge, a comparable result for ALDI has not been shown yet.

\subsection{Contribution}

\begin{itemize}
\item We propose a modification of existing Langevin samplers, coined as LIDL, utilizing successive ensemble enrichment to substantially reduce the number of required forward calls in the Langevin dynamics. The effect of ensemble enrichment is explained in Section \ref{sec:setting}, ensemble enrichment strategies are introduced in Section \ref{sec:enrichment} and we present the analysis of the method in Section \ref{sec:theory}.
For the family of ensemble distributions $(\hat{\mu}_t)_{t\geq 0}$ generated by the Langevin process, we show analytical consistency of the scheme in the following sense: under suitable consistency assumptions on the initial density $\pi_0$ (which need not coincide with the prior), the underlying potential $\Phi$ and the ensemble enrichment scheme, we find for $\delta >0$ a configuration of the enrichment and time $T_{\delta} > 0$ such that the expected 2-Wasserstein distance of the solution to the posterior is smaller than $\delta$, i.e.,
\begin{equation*}
    \mathbb{E}[\mathcal{W}_2(\hat{\mu}_{T_{\delta}},\mu_{\ast})] < \delta.
\end{equation*}
This result is a combination of the corresponding consistency result for ALDI, shown in Theorem \ref{thm:B-T-convergence_ALDI}, and the definition of consistent enrichment schemes given in Definition \ref{def:consistent_enrichment}, yielding the main result in Corollary \ref{cor:B-T-convergence_consistent}.
\item We introduce the concept of \textit{ensemble enrichment}, and provide a rigorous mathematical description of ensemble enrichment schemes $\mathcal{E}$ as stochastic mappings operating on families of empirical measures $(\hat{\mu}_t)_{t\geq 0}$ in Section \ref{sec:enrichment}. For carrying out the enrichment in Langevin processes, we propose several informed choices of enrichments, namely \textit{slicing} (adding up time slices $\hat{\mu}_{t \pm \Delta t}$ of the process at different times, Section \ref{sec:slicing}), \textit{diffusion propagation} (propagating the current ensemble using only the computationally cheap diffusion part of the Langevin process, Section \ref{sec:diffusion_step}) and \textit{generalized transport} (sampling i.i.d. from a learned random variable $Y \sim \hat{\mu}_t$ fitted to the current ensemble, Section \ref{sec:generalized_transport}).


\item In order to accelerate the convergence of Langevin dynamics in the case of multimodal invariant measures, we introduce the terminology of \textit{$2$-Wasserstein stable homotopy maps} $\mathcal{H}$ in Section \ref{sec:homotopy}. 
Such a function maps between an auxiliary initial potential $\mathcal{H}(0)=\Psi$ and the posterior potential $\mathcal{H}(1)=\Phi$. This framework enables to construct a time inhomogeneous drift term $b_{\mathcal{H}}(t,Y_t)=b(\mathcal{H}(s(t),Y_t)$ in the underlying Langevin dynamics that reads
\begin{equation*}
\mathrm{d}y_t^{(i)} = b_{\mathcal{H}}(t, Y_t)\mathrm{d}t + \Gamma(Y_t)\mathrm{d}W_t^{(i)}.
\end{equation*}
Based on Assumption \ref{ass:localconvergence} of local convergence in expected Wasserstein-$2$ distance of the associated particle system, we provide a convergence analysis in the case of piecewise constant inhomogeneity in Theorem \ref{thm:conv_homotopy}.
We show by numerical evidence that different homotopy designs affect the overall convergence speed over time, especially for multimodal distributions.
The suggested homotopy approach can be conceived as a preconditioning tool.

\item 
The final contribution concerns the numerical investigation of our method.
Based on the consistency results of Corollary \ref{cor:B-T-convergence_consistent} and Theorem \ref{thm:conv_homotopy}, we propose in Section \ref{sec:numerics} a principled way of tracking convergence of numerical realizations of the method in debiased Sinkhorn divergence \eqref{eq:sinkhorn_div}.
More precisely, we consider the convergence in distribution
\begin{equation*}
\mathcal{S}_{\epsilon}(\hat{\mu}_t,\mu_{\ast}^{(\overline{b})}) \stackrel{d}{\longrightarrow} \mathcal{S}_{\epsilon}(\tilde{\mu}_{\ast}^{(\overline{b})},\mu_{\ast}^{(\overline{b})}), \qquad t\to \infty, 
\end{equation*}
where $\hat{\mu}_t$ is the ensemble distribution generated by the Langevin process (producing a total of $\overline{b}$ approximate posterior samples), $\mu^{(\overline{b})}_{\ast}$ and $\tilde{\mu}^{(\overline{b})}_{\ast}$ are empirical measures of $\overline{b}\in\mathbb{N}$ particles drawn i.i.d. from the posterior. 
Moreover, initial ideas for adaptive ensemble enrichment schemes and homotopy designs are presented in Sections \ref{sec:num_adaptive} and \ref{sec:num_multi_homotopy}.
We investigate the performance of the proposed method on several numerical test problems in Sections \ref{sec:num:uni_multi} and \ref{sec:num_darcy}.
This includes a challenging multimodal posterior and a high-dimensional Darcy equation often used as benchmark in Uncertainty Quantification. 

\end{itemize}

%% file: 02_problem.tex
\section{Methodology}
\label{sec:setting}

Suppose we are interested to obtain a batch of $\overline{b}\in\mathbb{N}$ samples drawn from the posterior distribution $\mu_\ast$. 
Suppose further that we have a method which, relying on the evaluation of the potential $\Phi$ or its gradient $\nabla \Phi$, maps $\overline{b}$ particles/samples drawn from some initial measure $\mu_0$ 
to particles approximately being distributed according to $\otimes_{i=1}^{\overline{b}} \mu_\ast$ in $N_{\text{iter}}$ iterations. We will refer to this method as the \textit{particle propagator}. 
One such particle propagator is given by Langevin dynamics based particle systems \cite{roberts1996exponential,garbuno2020interacting,garbuno2020affine,reich2021fokker}. 
Here, $N_\text{iter}$ depends on the distance $\mathrm{d}(\mu_{0},\mu_\ast)$ between $\mu_{0}$ and $\mu_\ast$, where $\mathrm{d}$
is a suitable metric on the set of probability measures. Note that we do not require the initial measure $\mu_0$ to be identical to the prior $\mu_{\mathrm{prior}}$ although this is a natural choice in the setup of Bayesian inference.
We assume one iteration step of $\overline{b}$ particles to require $\overline{b}$ forward calls (as is the case e.g. for ALDI \cite{garbuno2020affine,reich2021fokker}), leading to a total workload of
\begin{equation*}
 \#\text{forward calls} = \overline{b}N_\text{iter}.
\end{equation*}
Our goal throughout this manuscript is to reduce the number of forward calls significantly.
%
%
For this we propose to construct a sequence of auxiliary measures from which an increasing number of particles are drawn over time and eventually propagated to the sought posterior distribution.
We may formally divide this approach into two stages.
The first stage deals with the design choice of the underlying particle system in terms of a modification of the potential $\Phi$ through the introduction of surrogates or homotopy maps. 
The second stage concerns the propagation of particles to the true posterior using successive ensemble enrichment.
In practical application, both stages are intertwined in that ensemble enrichment may be performed while using variations of the potential.


\vspace{1ex}
\paragraph{Stage I: Auxiliary measures}
Initially, one chooses an accessible distribution $\mu_{0}$, which not necessarily coincides with the prior distribution.
Next, some auxiliary potential $\Psi$ is introduced and instead of directly computing with $\Phi$, we use the intermediate potentials
\begin{equation}
\label{eq:annealing}
    \mathcal{H}(s):= (1-s)\Psi + s\Phi,\quad s\in[0,1],
\end{equation}
replacing the potential $\Phi$ in the particle propagator. This homotopy approach carries out a linear interpolation from the potential $\Psi=\mathcal{H}(0)$ to $\Phi=\mathcal{H}(1)$ in this example.
By using this potential, we introduce intermediate target measures with density proportional to $\exp(-\mathcal{H}(s))$.
From the perspective of the particle dynamics, the particles are moved towards the final target measure, i.e. the posterior distribution defined by $\mathcal{H}(1)=\Phi$, by passing through the intermediate measures associated with $\mathcal{H}(s)$. 
From the classical optimization point of view, this type of homotopy can be seen as a sequence of Tikhonov regularizations with parameters $(1-s)/s\in[0,\infty]$.
The effect of different designs of interpolations between $\Psi$ and $\Phi$ is discussed in the numerical examples in Section \ref{sec:num_multi_homotopy}.
It turns out that this approach is crucial to enable and accelerate the propagation of particles to multimodal distributions with our approach.
We refer to Figure~\ref{fig:scheme_anneal} for an illustration of the technique and to Section~\ref{sec:homotopy} for a more in depth discussion.

We now discuss potential choices of the start distribution $\mu_0$.
A canonical first choice would be $\mu_0=\mu_\text{prior}$.
However, this might not be ideal if the prior distribution is not a good approximation of the posterior.
A reasonable alternative is to choose $\mu_0$ as a Gaussian approximation of the posterior distribution.

As a third option, assume that a surrogate model of the forward operator $\mathcal{G}$ or of the potential $\Phi$ is available. 
We further assume that it takes $b_\text{surr}$ evaluations of the exact forward map to construct this surrogate and that the cost of evaluating the surrogate potential $\hat{\Phi}$ is negligible compared to the cost of evaluating $\Phi$.
Then, we choose $\mu_0$ with density proportional to $\exp(-\hat{\Phi})$.
Consequently, samples can be drawn with negligible workload through a particle propagator using the surrogate potential $\hat{\Phi}$.
This idea is illustrated in Figure~\ref{fig:stage_1}. 
Depending on the quality of the surrogate, such a $\mu_0$ is a potentially better proxy to $\mu_\ast$ than $\mu_{\text{prior}}$ in the sense that
\begin{equation}
    \label{eq:measureprecond}
    \mathrm{d}(\mu_0, \mu_\ast) < \mathrm{d}(\mu_{\text{prior}},\mu_\ast).
\end{equation}
We refer to evaluations of both the surrogate potential and $\mathcal{H}(0)$ as \textit{free calls}.

In applications the exact solution of the forward model $\mathcal{G}$ is often unknown and only approximations $\mathcal{G}_h\approx\mathcal{G}$ are available.
The influence of this approximation on the convergence to the true posterior has e.g. been considered in~\cite{calvetti2018iterative}.
Note that successive improvement of the approximation during particle propagation potentially leads to significant total workload reduction.
However, this idea and its interaction with the concepts presented in this paper is beyond the scope of this work and we only comment on it in Section~\ref{sec:conclusion}. 



\input{illustration_surr}

\begin{figure}
    \centering
    \begin{tikzpicture}
    \node at (0,0) { \includegraphics[width = 0.9\linewidth]{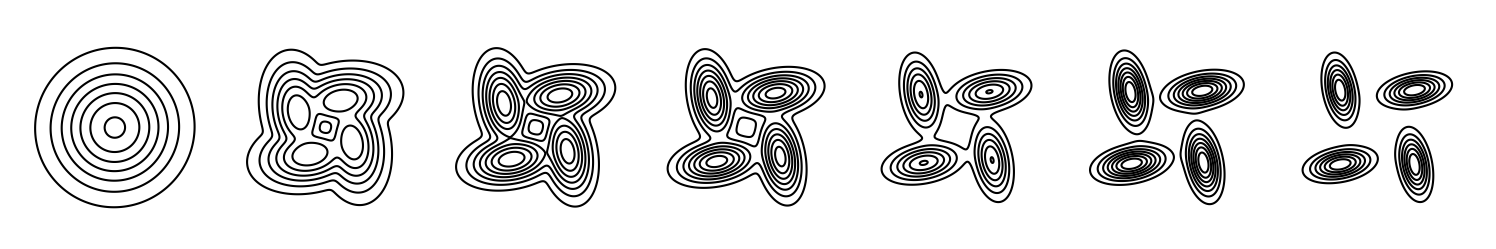}};
    
    \node at (-6.25,-1.2) {$s=0$};
    \node at ( 6.25,-1.2) {$s=1$};
    \draw[->, dashed] (-5.25, -1.2) -- (5.25,-1.2);
    \end{tikzpicture}
   
    \caption{\textit{Illustration of the homotopy approach towards a multimodal distribution with $\Psi$ choosen as the Gaussian approximation for $h=0$ on the left-hand side to the target distribution on the right-hand side for $h=1$.}}
    \label{fig:scheme_anneal}
\end{figure}
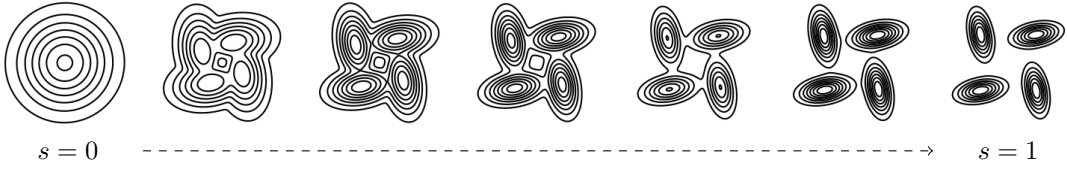

\vspace{1ex}
\paragraph{Stage II: Intermediate ensemble enrichment}
\input{illustration}

In the second stage of our method, the exact forward map is used.
We start by propagating $b_0$ samples drawn from $\mu_0$ obtained by the Stage I procedure and iterate them $n_0\in\mathbb{N}$ times with the exact forward model.
This yields $b_0$ samples distributed according to some intermediate measure $\mu_{1}^{(b_0)}$ defined on $\mathbb{R}^{b_0D}$.
Since we have now used the true potential, it is reasonable to assume that
\begin{equation}
    \mathrm{d}\left(\mu_{1}^{(b_0)},\bigotimes_{i=1}^{b_0}\mu_{\ast}\right) < \mathrm{d}\left(\mu_{0}^{(b_0)},\bigotimes_{i=1}^{b_0}\mu_{\ast}\right), 
\end{equation}
where $\mu_{0}^{(b_0)} = \bigotimes_{i=1}^{b_0}\mu_0$ defines a product measure on $\mathbb{R}^{b_0D}$.
Within the ensemble enrichment, we change the viewpoint from distributions on the product space to ensemble distributions on $\mathbb{R}^D$.
Throughout this manuscript, we shall use the wedge symbol $\hat{\phantom{\mu}}$ to flag a distribution as an empirical measure or an ensemble distribution.
For instance, the intermediate distribution $\mu_1^{(b_0)}$ on the product space $\mathbb{R}^{b_0D}$ gives rise to an empirical (random) measure $\hat{\mu}_1^{(b_0)}$ on $\mathbb{R}^D$. 

\noindent 
Next, we define a \textit{number of enrichments} $L\in\mathbb{N}_0$ and \textit{batch sizes} $b_0,\ldots,b_{L}\in \mathbb{N}$ with $\overline{b} = \sum_{\ell=0}^{L}b_\ell$.
Furthermore, we denote by $\overline{b}_k = \sum_{\ell=0}^k$ the \textit{partial summed batch sizes}.   
Formally, an \textit{ensemble enrichment scheme} $\mathcal{E}$ produces an empirical measure $\hat{\mu}^{(b+a)}_{\mathrm{enriched}}$ from an empirical measure $\hat{\mu}^{(b)}_{\mathrm{original}}$ for some arbitrary $b\in\mathbb{N}$ and a dependent number of additional samples $a\in\mathbb{N}$.
For now, we tacitly assume that both distributions are close enough in the sense that
\begin{equation}
    \label{eq:enriched_dist}
    \mathbb{E}\left[\mathrm{d}\left(\hat{\mu}_{\mathrm{original}}^{(b)}, \hat{\mu}_{\mathrm{enriched}}^{(b+a)} \right)\right]< \epsilon
\end{equation}
for some suitable threshold $\epsilon > 0$.
%
We then apply $\mathcal{E}$ to the current batch of $b_0$ samples at level $\ell =1$ to draw $b_1$ additional samples, yielding a total of $\overline{b}_1$ samples distributed according to some measure $\hat{\mu}_{1}^{(\overline{b}_1)}$, which is close to $\hat{\mu}_{1}^{(b_0)}$ in the sense of \eqref{eq:enriched_dist}.
This enriched batch is propagated through the dynamics $n_1\in\mathbb{N}$ times, yielding $\overline{b}_1$ samples distributed according to an intermediate measure $\hat{\mu}_{2}^{(\overline{b_1})}$.
The process continues iteratively for $\ell=2,\ldots,L $: at level $\ell$ the samples distributed according to the intermediate measure $\mu_{\ell}^{(\overline{b}_{\ell-1})}$ are enriched with $b_{\ell}$ samples by application of $\mathcal{E}$.
The new batch of $\overline{b}_\ell$ samples is distributed according to $\hat{\mu}_{\ell}^{(\overline{b}_\ell)}$.
This batch is then iterated $n_\ell\in\mathbb{N}$ times through the dynamics, leading to $\overline{b}_\ell$ samples drawn from a measure $\hat{\mu}_{\ell+1}^{(\overline{b}_\ell)}$.
Finally, we get an ensemble distributed according to $\hat{\mu}_{L+1}^{(\overline{b})}$.
The described process is sketched for the case of a single enrichment (i.e. $L=1$) in Figure~\ref{fig:stage_2}.
With $n_\text{iter} := \sum_{\ell=0}^L n_\ell$, the number of forward calls in this scheme is given as 
\begin{equation}
\label{eq:callcount}
    \# \text{forward calls} = b_0 n_\text{iter} +  b_1 ( n_\text{iter} - n_0 ) 
    + \ldots + b_{L} ( n_\text{iter} - \sum_{\ell=0}^{L-1} n_\ell ) 
    = 
    \overline{b} n_\text{iter}  -  \sum_{\ell=1}^{L} b_\ell \sum\limits_{k=0}^{\ell-1} n_k .
\end{equation}

To quantify the computational cost reduction, let $b_\ell = p_\ell \overline{b}$ for a percentile $p_\ell \in (0,1)$ such that $\sum_{\ell=0}^{L}p_\ell = 1$ and let $n_\ell = c_\ell n_\text{iter}$ for some $c_\ell \in (0,1)$ with $\sum_{i=0}^{L} c_\ell = 1$.
Then, \eqref{eq:callcount} becomes
\begin{equation}
    \# \text{forward calls} =  \overline{b}n_\text{iter}\underbrace{\left( 1 - \sum_{\ell=1}^{L} p_\ell \sum\limits_{k=0}^{\ell-1} c_k  \right)}_{=:c}.
\end{equation}
Hence, a reduction factor of $c\in [1/n_\text{iter},1)$ is obtained that only depends on the parameters $L$, $p_{\ell}$ and $c_{\ell}$ for $\ell=0,\ldots,L$.
Eventually, the number of forward map evaluations is effectively reduced to 
\begin{equation}
 \#\text{forward calls} = b_\text{surr} +  c\overline{b}n_\text{iter}.
\end{equation}

There are two main design aspects of the method that require a further analysis: first, the choice of the particle propagator and second, how to perform the ensemble enrichment.
The remainder of this section is concerned with answering the first question.
For this, we provide an overview of Langevin sampling methods, which are at the center of our method.
The question of how to generate new samples is discussed subsequently in Section~\ref{sec:enrichment}.

\subsection{Langevin sampling methods}
\label{sec:propagators}


\input{02_4_propagators}

%% file: illustration_surr.tex
\begin{figure}[ht]
\begin{center}
\begin{tikzpicture}

\node at (0, 0)   (A) {};
\node at (7.7,0.35) (C) {};
\node at (8.5+1.5,0.3) (D) {};

\filldraw (A) circle (2.8pt);
\filldraw (C) circle (2pt);
\filldraw (D) circle (2.pt);

\draw   (A) to[out=20,in=200] (C);
\draw[<->,dotted, opacity = 0.5]   (C) -- (D);


\node at (0,0.5) {$\mu_{\mathrm{prior}}$};
\node at (7.7,0.35+0.5) {$\mu_0$};
\node at (8.5+1.5,0.3+0.5) {$\mu_{\ast}$};

\draw[->] (1.4,-1) -- (6.4,-1);
\node at (4,-0.75) {\scriptsize \textit{propagate samples}};
\node at (4,-1.25 ) {\scriptsize \textit{using surrogate potential}};

 \foreach \x in {-0.6,-0.4,-0.2,0,0.2,0.4,0.6}
  {
  \foreach \y in {-1.3,-1.1,-0.9,-0.7,-0.5,-0.3} {
    \filldraw[black]  (\x,\y) circle (0.04);
    }

  }

\foreach \point in
{
(-0.76+7.7, -1.28 *0.8),
(0.13+7.7, -1.2 *0.8),
(0.18+7.7, -1.24 *0.8),
(-0.21+7.7, -0.28*0.8),
(0.46+7.7, -0.6*0.8),
(0.0+7.7, -1.32*0.8),
(0.1+7.7, -1.08*0.8),
(0.24+7.7, -1.54*0.8),
(0.43+7.7, -0.84*0.8),
(0.27+7.7, -1.39*0.8),
(-0.49+7.7, -0.46*0.8),
(0.41+7.7, -0.97*0.8),
(0.41+7.7, -0.7*0.8),
(0.33+7.7, -1.24*0.8),
(-0.5+7.7, -1.51*0.8),
(0.08+7.7, -0.83*0.8),
(-0.72+7.7, -1.2*0.8),
(0.14+7.7, -1.02*0.8),
(-0.19+7.7, -1.03*0.8),
(0.03+7.7, -0.73*0.8),
(0.35+7.7, -1.28*0.8),
(-0.34+7.7, -1.44),
(-0.16+7.7, -1.6),
(-0.01+7.7, -1.26),
(-0.26+7.7, -0.98),
(0.16+7.7, -0.58),
(-0.35+7.7, -0.88),
(-0.25+7.7, -1.59),
(-0.84+7.7, -1.0),
(-0.08+7.7, -1.21),
(0.07+7.7, -0.86),
(-0.4+7.7, -1.33),
(0.14+7.7, -0.77),
(-0.65+7.7, -1.06),
(0.31+7.7, -0.81*0.8),
(-0.48+7.7, -0.75),
(-0.04+7.7, -0.74),
(-0.02+7.7, -0.5),
(0.26+7.7, -0.71),
(0.11+7.7, -0.39),
(0.27+7.7, -0.87),
(-0.32+7.7, -1.06),
(0.1+7.7, -1.37),
(0.14+7.7, -0.91),
(-0.33+7.7, -0.86),
(0.01+7.7, -0.86),
(-0.21+7.7, -0.76),
(-0.25+7.7, -0.75),
(-0.17+7.7, -1.23),
(0.27+7.7, -1.6*0.8)}
{
\filldraw[black]  \point circle (0.04);
}


\end{tikzpicture}
\end{center}
\captionsetup{width=.92\linewidth}
\caption{
\textit{Stage I: Samples drawn from the the prior measure $\mu_{\mathrm{prior}}$ are propagated to $\mu_0$, the approximate posterior measure introduced through the substitution by the chosen surrogate.}
\label{fig:stage_1}
}
\end{figure}

%% file: illustration.tex
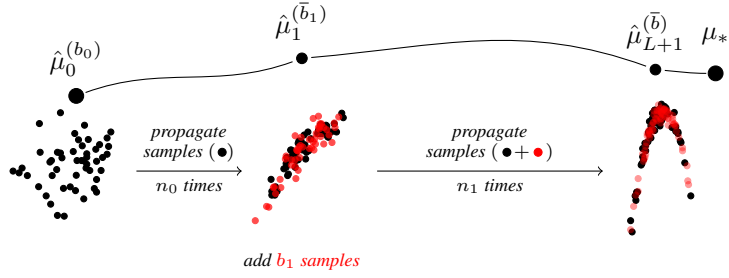
\begin{figure}[ht]
\begin{center}
\begin{tikzpicture}

\node at (0, 0)   (A) {};
\node at (3, 0.5)   (B) {};
\node at (7.7,0.35) (C) {};
\node at (8.5,0.3) (D) {};

\filldraw (A) circle (2.8pt);
\filldraw (B) circle (2pt);
\filldraw (C) circle (2pt);
\filldraw (D) circle (2.8pt);

\draw   (A) to[out=20,in=200] (B);
\draw   (B) to[out=5,in=160] (C);
\draw   (C) to[out=-5,in=180] (D);

\node at (0,0.5) {$\hat{\mu}^{(b_0)}_{0}$};
\node at (3, 0.5+0.5) {$\hat{\mu}^{(\overline{b}_1)}_{1}$};
\node at (7.7,0.35+0.5) {$\hat{\mu}^{(\overline{b})}_{L+1}$};
\node at (8.5,0.3+0.5) {$\mu_{\ast}$};

\node at (3,-2.2) {\scriptsize\textit{add {\color{red}$b_1$ samples}}};


\draw[->] (0.8,-1) -- (2.2,-1);
\node at (1.5,-0.5) {\scriptsize \textit{propagate}};
\node at (1.5,-0.75) {\scriptsize \textit{samples}
$(\phantom{o})$};

\ifnum\classstyle=0 
\filldraw  (1.94 + 0.04,-0.75) circle (1.5pt);
\fi
\ifnum\classstyle=1 
\filldraw  (1.94,-0.75) circle (1.5pt);
\fi
\ifnum\classstyle=2 
\filldraw  (1.94 + 0.13 ,-0.75) circle (1.5pt);
\fi

\node at (1.5,-1.2) {\scriptsize $n_0$~\textit{times}};

\draw[->] (4,-1) -- (7,-1);
\node at (5.5,-0.5) {\scriptsize \textit{propagate}};
\node at (5.5,-0.75) {\scriptsize \textit{samples} $(\phantom{o}+\phantom{o})$};
\ifnum\classstyle=0 
\filldraw  (5.725+0.02,-0.75) circle (1.5pt);
\filldraw[red]  (6.14 +0.07,-0.75) circle (1.5pt);
\fi
\ifnum\classstyle=1 
\filldraw  (5.725,-0.75) circle (1.5pt);
\filldraw[red]  (6.14,-0.75) circle (1.5pt);
\fi
\ifnum\classstyle=2 
\filldraw  (5.725 +0.13,-0.75) circle (1.5pt);
\filldraw[red]  (6.14 + 0.15,-0.75) circle (1.5pt);
\fi

\node at (5.5,-1.2) {\scriptsize $n_1$~\textit{times}};

\foreach \point in
{(-0.76, -1.28 *0.8),
(0.13, -1.2 *0.8),
(0.18, -1.24 *0.8),
(-0.21, -0.28*0.8),
(0.46, -0.6*0.8),
(0.0, -1.32*0.8),
(0.1, -1.08*0.8),
(0.24, -1.54*0.8),
(0.43, -0.84*0.8),
(0.27, -1.39*0.8),
(-0.49, -0.46*0.8),
(0.41, -0.97*0.8),
(0.41, -0.7*0.8),
(0.33, -1.24*0.8),
(-0.5, -1.51*0.8),
(0.08 , -0.83*0.8),
(-0.72 , -1.2*0.8),
(0.14 , -1.02*0.8),
(-0.19 , -1.03*0.8),
(0.03 , -0.73*0.8),
(0.35 , -1.28*0.8*0.8),
(-0.34 , -1.44),
(-0.16 , -1.6),
(-0.01 , -1.26),
(-0.26 , -0.98),
(0.16 , -0.58),
(-0.35 , -0.88),
(-0.25 , -1.59),
(-0.84 , -1.0),
(-0.08 , -1.21),
(0.07 , -0.86),
(-0.4 , -1.33),
(0.14 , -0.77),
(-0.65 , -1.06),
(0.31 , -0.81*0.8),
(-0.48 , -0.75),
(-0.04 , -0.74),
(-0.02 , -0.5),
(0.26 , -0.71),
(0.11 , -0.39),
(0.27 , -0.87),
(-0.32 , -1.06),
(0.1 , -1.37),
(0.14 , -0.91),
(-0.33 , -0.86),
(0.01 , -0.86),
(-0.21 , -0.76),
(-0.25 , -0.75),
(-0.17 , -1.23),
(0.27 , -1.6*0.8)}
{
    \filldraw[black]  \point circle (0.04);
}

\foreach \point in
{(2.81, -1.06),
(3.33, -0.51),
(3.25, -0.38),
(3.39, -0.48),
(2.95, -0.85),
(2.64, -1.32),
(2.92, -0.76),
(2.91, -0.95),
(2.79, -0.9),
(2.88, -0.76),
(3.08, -0.48),
(3.05, -0.84),
(2.96, -0.6),
(2.9, -0.8),
(3.27, -0.5),
(3.25, -0.5),
(3.34, -0.54),
(2.67, -1.27),
(2.86, -0.81),
(2.6, -1.27),
(3.46, -0.48),
(2.95, -0.63),
(3.03, -0.52),
(3.32, -0.39),
(2.8, -0.82),
(3.55, -0.29),
(2.65, -1.21),
(3.01, -0.72),
(2.82, -1.04),
(2.81, -0.84),
(3.04, -0.65),
(3.21, -0.5),
(3.33, -0.43),
(3.28, -0.43),
(3.43, -0.37),
(2.79, -0.95),
(3.33, -0.4),
(3.06, -0.75),
(2.6, -1.32),
(2.92, -0.86),
(2.86, -0.72),
(2.59, -1.11),
(2.76, -0.97),
(3.53, -0.23),
(2.85, -0.72),
(2.83, -0.85),
(3.12, -0.57),
(3.22, -0.46),
(3.09, -0.73),
(3.35, -0.55),
(2.99, -0.6)}
{
    \filldraw[black]  \point circle (0.04);
}

\foreach \point in
{
(2.86, -0.85),
(3.24, -0.37),
(2.63, -1.21),
(3.01, -0.68),
(3.18, -0.64),
(3.4, -0.62),
(3.19, -0.5),
(3.29, -0.42),
(3.15, -0.61),
(2.91, -0.53),
(2.92, -1.1),
(3.1, -0.58),
(3.23, -0.44),
(2.91, -0.85),
(3.49, -0.28),
(3.43, -0.36),
(3.33, -0.45),
(3.29, -0.64),
(2.53, -1.47),
(2.78, -0.84),
(2.93, -0.72),
(2.72, -1.22),
(3.13, -0.69),
(3.16, -0.72),
(3.26, -0.37),
(3.25, -0.54),
(2.97, -0.82),
(2.91, -0.94),
(2.65, -1.3),
(3.44, -0.42),
(2.86, -0.92),
(2.38, -1.66),
(2.95, -0.73),
(2.99, -0.55),
(2.95, -0.99),
(3.22, -0.42),
(3.1, -0.7),
(3.17, -0.57),
(2.56, -1.48),
(3.36, -0.52),
(2.79, -0.8),
(2.96, -0.68),
(2.71, -1.11),
(3.22, -0.27),
(2.94, -0.73),
(2.62, -1.22),
(3.24, -0.41),
(3.15, -0.3),
(2.98, -0.66)}
{
    \filldraw[red, opacity = 0.7]  \point circle (0.04);
}

\foreach \point in
{(7.75, -0.18),
(7.43, -1.33),
(7.76, -0.24),
(7.66, -0.3),
(7.86, -0.24),
(7.67, -0.36),
(7.46, -1.28),
(7.72, -0.3),
(7.75, -0.27),
(8.04, -0.62),
(7.79, -0.2),
(7.78, -0.11),
(7.96, -0.42),
(8.0, -0.48),
(7.7, -0.28),
(7.8, -0.29),
(7.79, -0.19),
(7.8, -0.19),
(7.56, -0.63),
(7.73, -0.24),
(7.72, -0.36),
(7.69, -0.19),
(7.45, -1.28),
(8.03, -0.6),
(7.75, -0.2),
(7.69, -0.46),
(7.76, -0.4),
(7.77, -0.24),
(7.57, -0.64),
(7.54, -0.85),
(7.84, -0.16),
(7.66, -0.34),
(7.37, -1.66),
(7.68, -0.33),
(7.52, -0.9),
(7.61, -0.5),
(7.7, -0.33),
(7.97, -0.47),
(7.83, -0.19),
(8.17, -1.33),
(7.81, -0.22),
(7.84, -0.18),
(7.7, -0.35),
(7.63, -0.46),
(8.09, -0.89),
(7.57, -0.73),
(7.53, -0.7),
(7.77, -0.25),
(7.76, -0.24),
(7.82, -0.14),
(7.69, -0.19),
(7.6, -0.61),
(7.53, -0.81),
(7.95, -0.22),
(7.74, -0.28),
(7.61, -0.48),
(7.62, -0.34),
(7.66, -0.23),
(7.78, -0.14),
(8.0, -0.5),
(7.79, -0.27),
(7.78, -0.33),
(7.69, -0.36),
(7.73, -0.28),
(7.64, -0.4),
(7.39, -1.81),
(7.82, -0.23),
(7.51, -0.92),
(7.78, -0.27),
(7.64, -0.56),
(7.44, -1.39),
(7.64, -0.34),
(7.97, -0.41),
(7.92, -0.42),
(7.57, -0.66),
(7.63, -0.55)}
{
    \filldraw[black]  \point circle (0.04);
}

\foreach \point in
{(7.62, -0.46),
(7.77, -0.36),
(7.58, -0.62),
(7.97, -0.39),
(7.44, -1.32),
(7.47, -1.27),
(7.78, -0.23),
(7.67, -0.36),
(7.96, -0.4),
(7.87, -0.3),
(7.78, -0.21),
(7.76, -0.17),
(7.81, -0.19),
(7.74, -0.12),
(7.59, -0.5),
(8.13, -1.13),
(7.77, -0.22),
(7.78, -0.3),
(8.01, -0.55),
(7.47, -1.1),
(7.56, -0.73),
(7.87, -0.39),
(7.68, -0.2),
(7.89, -0.24),
(7.61, -0.51),
(7.66, -0.44),
(7.85, -0.19),
(7.71, -0.31),
(7.67, -0.46),
(7.55, -0.84),
(7.72, -0.24),
(8.06, -0.69),
(7.65, -0.39),
(7.73, -0.2),
(7.57, -0.67),
(7.7, -0.3),
(7.68, -0.25),
(7.59, -0.53),
(7.94, -0.38),
(7.75, -0.2),
(7.62, -0.48),
(7.82, -0.17),
(7.62, -0.56),
(7.57, -0.78),
(7.81, -0.18),
(7.79, -0.25),
(7.54, -0.87),
(7.76, -0.28),
(7.83, -0.28),
(7.72, -0.19),
(7.93, -0.26),
(7.85, -0.21),
(7.86, -0.21),
(7.64, -0.51),
(7.93, -0.35),
(7.78, -0.16),
(7.95, -0.37),
(8.12, -0.93),
(7.77, -0.31),
(7.8, -0.24),
(7.61, -0.46),
(7.93, -0.31),
(7.56, -0.82),
(7.37, -1.78),
(7.67, -0.35),
(8.14, -1.18),
(7.53, -0.77),
(7.98, -0.47),
(7.86, -0.3),
(7.82, -0.18),
(7.72, -0.33),
(7.41, -1.47),
(8.06, -0.73)}
{
    \filldraw[red, opacity = 0.4]  \point circle (0.04);
}
\end{tikzpicture}
\end{center}
\captionsetup{width=\linewidth}
\caption{
\textit{Stage II:
Samples are propagated and enriched a single time, i.e. $L=1$, eventually leading to $\overline{b} = \overline{b}_1= b_0+b_1$ samples defining the final empirical measure $\hat{\mu}_{L+1}^{(\overline{b})}$.
}
%
%
\label{fig:stage_2}
}
\end{figure}

%% file: 02_4_propagators.tex

The class of particle propagators considered in this work is based on interacting particle systems. The starting point is a \textit{single-particle} first order overdamped Langevin process of the form
\begin{equation}
    \mathrm{d}y_t = -\nabla\Phi(y_t)\mathrm{d}t + \sqrt{2}\mathrm{d}W_t, \label{eq:basic_SDE}
\end{equation}
where $W_t$ is $D$-dimensional Brownian motion and the potential $\Phi$ is defined by the Bayesian inference problem via~\eqref{eq:potential}.
The probability density function $\pi_t$ of the particle $y_t$ at time $t\geq0$ satisfies the Fokker-Planck equation \cite{oksendal2003stochastic}
\begin{equation}\label{eq:basis_Fokker}
    \partial_t\pi_t = \nabla\cdot \left( \pi_t \nabla (\log\pi_t - \log\pi_\ast) \right). 
\end{equation}
It can be seen immediately that $\pi_\ast$ is a stationary solution, in particular the right-hand side of \eqref{eq:basis_Fokker} equals zero for $\pi_t=\pi_\ast$.
Hence, the posterior $\mu_\ast$ is an invariant measure of the process \eqref{eq:basic_SDE}.
Furthermore, a desirable property of the process is ergodicity. This means that $\mu_\ast$ is the only invariant measure and all initial measures converge to $\mu_\ast$ in a suitable sense as $t\rightarrow \infty$.
Consequently, $y_t$ defined by \eqref{eq:basic_SDE} is distributed according to the posterior in the limit $t\rightarrow\infty$.
Ergodicity is ensured, for example, under certain growth conditions on the potential $\Phi$ \cite{garbuno2020affine}. \\



\noindent
There exists a vast literature on extensions to dynamics of the form \eqref{eq:basic_SDE} usually with the goal of accelerating convergence to the posterior by introducing interaction between multiple particles \cite{garbuno2020interacting, garbuno2020affine,reich2021fokker}. Collecting $B\in\mathbb{N}$ particles $\{y_t^{(i)}\}_{i=1}^B$ at time $t$ into a vector
\begin{equation*}
    Y_t = \operatorname{vec}(y_t^{(1)},y_t^{(2)},\hdots, y_t^{(B)}) \in \mathbb{R}^{DB},\quad t\geq 0,
\end{equation*}
many interacting particle approaches admit the general form
\begin{equation}\label{eq:general_SDE}
    \mathrm{d}y^{(i)}_t = -A(Y_t)\nabla_{y^{(i)}_t}\mathcal{V}(Y_t)\mathrm{d}t + \Gamma(Y_t)\mathrm{d}W^{(i)}_t\qquad\text{for } i=1,\hdots,B.
\end{equation}
Here, $A(Y_t)\in \mathbb{R}^{D\times D}$, $\Gamma(Y_t)\in\mathbb{R}^{D\times J}$ for some $J\in\mathbb{N}$, $W^{(i)}_t$ are independent $J$-dimensional Brownian motions for $i=1,\hdots,J$ and $\mathcal{V}\colon\mathbb{R}^{D B}\rightarrow \mathbb{R}$ usually depends on the potential $\Phi$.
For convenience, we introduce a block notation for \eqref{eq:general_SDE}, which becomes
\begin{equation}
    \mathrm{d}Y_t = -\bm{A}(Y_t) \nabla_{Y_t} \mathcal{V}(Y_t) \mathrm{d}t + \bm{\Gamma}(Y_t)\mathrm{d}W_t,
\end{equation}
where $\bm{A}(Y_t) \in \mathbb{R}^{D B\times D B}$ and $\bm{\Gamma}(Y_t) \in \mathbb{R}^{D B\times JB}$ are block matrices with all $B$ blocks equal to $A(Y_t)$ and $\Gamma(Y_t)$ respectively.
Moreover,
\begin{equation}
    \nabla_{Y_t} \mathcal{V}(Y_t) = \operatorname{vec}\left(\nabla_{y^{(i)}_t}\mathcal{V}(Y_t), ~i = 1,\ldots, B\right)
\end{equation}
and $W_t$ is $JB$-dimensional Brownian motion. 

Choosing $A \equiv I_d$, $\Gamma \equiv \sqrt{2}I_d$ and $\mathcal{V}(Y_t) = \sum_i^{B} \Phi(y_t^{(i)})$ in \eqref{eq:general_SDE} leads to a particle system where each particle follows the process \eqref{eq:basic_SDE} independently.
In this case, there is no interaction between the particles.
Moreover, the system lacks \textit{affine invariance} \cite{goodman2010ensemble}, meaning that it does not retain its convergence properties under linear affine transformations of the state variables.
Both properties can be important to enable accelerated convergence \cite{garbuno2020affine}, e.g. in the case of multimodal or highly skewed posteriors.
We hence review some more involved methods of the form \eqref{eq:general_SDE} that address these issues.
For this, we define the time-dependent ensemble mean 
\begin{equation}
    \overline{y}_t = \dfrac{1}{B}\sum_{i=1}^B y_t^{(i)} \in \mathbb{R}^{D}, 
\end{equation} 
and 
\begin{equation}
    \overline{Y}_t = \mathrm{vec}(\overline{y}_t,\ldots,\overline{y}_t) \in \mathbb{R}^{DB},
\end{equation}
as well as the empirical covariance
\begin{equation}
    \label{eq:emp_covariance}
    C(Y_t) = \frac{1}{B} \sum\limits_{i=1}^B (y_t^{(i)} - \overline{y}_t)(y^{(i)} - \overline{y}_t)^T\in\mathbb{R}^{D\times D}.
\end{equation}  
Furthermore, we define the product posterior density 
\begin{equation}
    \label{eq:product_post}
    \tilde{\pi}_{\ast}(Y) = \prod_{i=1}^B \pi_{\ast}(y^{(i)}), \qquad \text{for}\quad  Y=\operatorname{vec}(y^{(1)},\ldots,y^{(B)}) \in \mathbb{R}^{DB},
\end{equation}
with $y^{(i)}\in\mathbb{R}^D$ for $i=1,\ldots,B$.
In what follows, we briefly review some extensions of the classic Langevin dynamics from \eqref{eq:basic_SDE} in historical order.

\vspace{1ex}
\paragraph{Scaled first order overdamped Langevin dynamics}
As a first step towards affine invariance, \eqref{eq:basic_SDE} is extended by the introduction of a positive definite scaling matrix $C\in\mathbb{R}^{D\times D}$.
The choice $A \equiv C$, $\Gamma \equiv \sqrt{2}C^{1/2}$ and $\mathcal{V}(Y) = \sum_i^{B} \Phi(y^{(i)})$ in \eqref{eq:general_SDE} leads to the dynamics
\begin{equation}\label{eq:scaled_SDE}
    \mathrm{d}y^{(i)}_t = -C\nabla\Phi(y^{(i)}_t)\mathrm{d}t + \sqrt{2}C^{1/2}\mathrm{d}W_t\qquad\text{for $i=1,\hdots,B$}.
\end{equation}
Ideally, $C$ should be close to the (unknown) posterior covariance matrix.
The Fokker-Planck equation for the PDF $\pi_t^{(i)}$ of the $i$-th particle at time $t$ now reads
\begin{equation}
    \partial_t\pi_t = \nabla\cdot \left( \pi_t C \nabla (\log\pi_t - \log\pi_\ast) \right),
\end{equation}
retaining the posterior $\mu_\ast$ as invariant measure.
This system has no interaction between particles and is not affine invariant for time-homogeneous $C$.
However, affine invariance can be achieved for $C=C(t)$ equal to the process covariance matrix \cite{garbuno2020affine}.
This is a key observation, laying the groundwork for the following sections.

\vspace{1ex}
\paragraph{Ensemble Kalman Sampler (EKS, cf. \cite{garbuno2020interacting})}
    
The EKS employs a time dependent scaling $C = C(Y_t)$ given by the empirical covariance. In particular, this scaling depends nonlinearly on the current ensemble.
Setting $A(Y_t) = C(Y_t)$, $\Gamma(Y_t)=\sqrt{2}C(Y_t)^{1/2}$ and $\mathcal{V}(Y_t) = \sum_i^{B} \Phi(y_t^{(i)})$ in \eqref{eq:general_SDE}, the process equations become 
\begin{equation}
\mathrm{d}y^{(i)}_t = -C(Y_t)\nabla\Phi(y_t^{(i)})\mathrm{d}t + \sqrt{2}C(Y_t)^{1/2}\mathrm{d}W_t^{(i)},\quad i=1,\ldots,B.
\label{eq:EKSparticle}
\end{equation}
This choice of scaling is motivated by the large particle limit $B\rightarrow\infty$, which formally leads to the mean field equation
\begin{equation}\label{eq:EKS_mean_field_limit}
    \mathrm{d}y_t = -C(\pi_t)\nabla\Phi(y_t)\mathrm{d}t + \sqrt{2}C(\pi_t)^{1/2}\mathrm{d}W_t
\end{equation}
with the true process covariance $C(\pi_t) = \mathbb{E}_{y\sim\pi_t}[(y-\mathbb{E}_{\pi_t})(y-\mathbb{E}_{\pi_t})^{\intercal}]$. Here,  $\mathbb{E}_{\rho} = \int_{\mathbb{R}^D} y \rho(y) \mathrm{d}y$ denotes the first moment of a probability density $\rho$ on $\mathbb{R}^D$.
The corresponding Fokker-Planck equation 
\begin{equation}\label{eq:eks_fokker_planck}
    \partial_t\pi_t = \nabla\cdot \left( \pi_t C(\pi_t) \nabla (\log\pi_t - \log\pi_\ast) \right)
\end{equation}
and its associated process \eqref{eq:EKS_mean_field_limit}
are shown to be affine invariant in \cite{garbuno2020affine}, while retaining the  invariant measure $\mu_{\ast}$.
    
\vspace{1ex}
\paragraph{Affine Invariant Langevin Dynamics (ALDI, cf. \cite{garbuno2020affine})}
It is shown in \cite{reich2019note} and further elaborated on in \cite{garbuno2020affine} that the posterior $\mu_\ast$ is actually not invariant under the particle system \eqref{eq:EKSparticle}, i.e. in the case of finitely many particles.
Instead, the finite ensemble version \eqref{eq:EKSparticle} requires an additional correction term in the potential. 
To see this, note that in block notation \eqref{eq:EKSparticle} becomes
\begin{equation}\label{eq:EKS_block}
    \mathrm{d}Y_t = \bm{C}(Y_t)\nabla \ln \tilde{\pi}_{\ast}(Y_t)\mathrm{d}t + \sqrt{2}\bm{C}(Y_t)^{1/2}\mathrm{d}W,
\end{equation}
where $\bm{C}(Y)\in\mathbb{R}^{D B\times D B}$ is a block diagonal matrix with $B$ block entries $C(Y)$.
The associated Fokker-Planck equation for the product density $\tilde{\pi}_t$ on $\mathbb{R}^{D B}$ now reads
\begin{equation}\label{eq:fokker_bad}
    \partial_t\tilde{\pi}_t = \nabla\cdot \left( \tilde{\pi}_t \bm{C} \nabla (\log\tilde{\pi}_t - \log\tilde{\pi}_\ast) + \tilde{\pi}_t\nabla\bm{C} \right),
\end{equation}
with the vector valued divergence $(\nabla\bm{C})_i = \sum_{j=1}^{NB} \partial_j \bm{C}_{ij}$.
Due to this divergence term, $\tilde{\pi}_\ast$ is not invariant under \eqref{eq:fokker_bad} for a finite number of particles.
However, a straightforward calculation (see \cite{reich2019note} for details) yields
\begin{equation}
    \nabla \bm{C}(Y) = \frac{D + 1}{B}(Y-\overline{Y}).
\end{equation}
Hence, the undesired term in the Fokker-Planck equation can be eliminated by replacing the drift in \eqref{eq:EKS_block} accordingly with
\begin{equation*}
    \bm{C}(Y_t)\nabla \ln \tilde{\pi}_{\ast}(Y_t) \longrightarrow \bm{C}(Y_t)\nabla \ln \tilde{\pi}_{\ast}(Y_t) + \frac{D +1}{B}(Y_t-\overline{Y}_t).
\end{equation*}
ALDI adds the correction term to the drift term of the EKS.
Hence, $\mathcal{V}$ from \eqref{eq:general_SDE} takes the form $\mathcal{V}(Y) = \sum_{i=1}^{B} \Phi(y^{(i)}) - \frac{D +1}{2}\log|C(Y)|$.
Using the identity $C(Y)\nabla_{y^{(i)}}\log|C(Y)| = \frac{2}{B}(y^{(i)}-\overline{y})$, this leads to
\begin{equation}\label{eq:aldi}
    \mathrm{d}y^{(i)}_t = -C(Y_t)\nabla\Phi(y_t^{(i)})\mathrm{d}t + \frac{D +1}{B}(y_t^{(i)}-\overline{y}_t)\mathrm{d}t + \sqrt{2}C(Y_t)^{1/2}\mathrm{d}W_t^{(i)}.
\end{equation}
Under strong growth bound conditions on $\Phi$, $\nabla \Phi$ and $\operatorname{Hess}\Phi$, and given $B>D+1$, ALDI is ergodic, i.e. $\tilde{\pi}_t$ converges to $\otimes_{i=1}^B\mu_\ast$ as $t\to\infty$ in total variation distance.
The correction term not only makes $\tilde{\pi}_{\ast}$ invariant under the process, it also retains affine invariance for the finite particle system and its gradient-free version \cite{garbuno2020affine}. \\

\noindent
In practice, ALDI is used with a non-symmetric generalization of the square root
\begin{equation}\label{eq:nonsym_sqrt}
    C(Y_t)^{1/2} = \dfrac{1}{\sqrt{B}}\left( y_t^{(1)} - \overline{y}_t,\ldots, y_t^{(B)} - \overline{y}_t \right) \in \mathbb{R}^{D \times B}
\end{equation}
such that $C(Y) = C(Y)^{1/2}(C(Y)^{1/2})^{\intercal}$.
In contrast to the proper symmetric square root, \eqref{eq:nonsym_sqrt} can be obtained without additional computational cost.
Recently, the formalism of ALDI was extended to other forms of time-dependent scaling matrices such as \textit{localized empirical covariances} \cite{reich2021fokker} of the form
  \begin{equation}
    C^{(i)}(Y_t) = \sum_{j=1}^B \omega^{ij}_t (y^{(j)}_t - \overline{y}^{(i)}_t)(y^{(j)}_t - \overline{y}^{(i)}_t)^{\intercal},
\end{equation}
for weights $\omega_t^{ij}\in\mathbb{R}$. 
This approach has been demonstrated to be effective for multimodal distributions \cite{reich2021fokker}, but it requires the actual computation of the square root $C(Y_t^{(i)})^{1/2}$ at each time step to build the localized empirical covariances.
The respective cubic scaling in the number of samples might becomes computationally costly compared to the generalized non-symmetric square root \eqref{eq:nonsym_sqrt} used in ALDI.
Hence, to handle multimodalities in this work, we instead apply homotopy techniques to alter the potential, as in \ref{eq:annealing}, retaining standard empirical covariances.



%% file: 03_method.tex
\section{Ensemble enrichment}
\label{sec:enrichment}
The goal of this section is to introduce various enrichment strategies that lead to a reduction of the interaction with the forward model within the Langevin dynamics.
To that end, we will distinguish between three types of discrete measures. First, we define a general \textit{discrete measure} $\hat{\mu}^{(B)}$ for $B\in\mathbb{N}$ through a set of points $\{y^{(i)}\}_{i=1}^B$
in $\mathbb{R}^D$ by
\begin{equation}\label{eq:discrete_measure}
    \hat{\mu}^{(B)} = \dfrac{1}{B}\sum_{i=1}^B \delta_{y^{(i)}},
\end{equation}
where $\delta_y$ denotes the Dirac measure concentrated in $y\in\mathbb{R}^D$. 
If the points $\{y^{(i)}\}_{i=1}^B$ in \eqref{eq:discrete_measure} are i.i.d. samples from some distribution $\mu$ on $\mathbb{R}^D$, we will call $\mu^{(B)}=\hat{\mu}^{(B)}$ (a realization of) an \textit{empirical measure} associated with $\mu$. Third, if  $\{y^{(i)}\}_{i=1}^B = \{y_t^{(i)}\}_{i=1}^B$ corresponds to the solution of a $B$-particle Langevin process at time $t$, we will call $\hat{\mu}_t = \hat{\mu}^{(B)}$ an \textit{ensemble distribution}. In all of these cases, we suppress the dependence of the measure on the (random) set of points. In the last case, we additionally suppress the dependence on the batch size $B$, as it will be clear by context. Note that the $\{y_t^{(i)}\}$ are not necessarily i.i.d. in this case (in fact they are not independent for the solution of \eqref{eq:aldi}). 
 We will sometimes use the particle ensemble $\{y_t^{(i)}\}_{i=1}^B$ and the corresponding ensemble distribution $\hat{\mu}_t$ interchangeably.
The set of discrete measures on $\mathbb{R}^{D}$ of the form \eqref{eq:discrete_measure} for arbitrary $B\in\mathbb{N}$ and an arbitrary set of points $\{y^{(i)}\}_{i=1}^B$ in $\mathbb{R}^D$ will be denoted by $\mathcal{M}(\mathbb{R}^{D})$. The set of all families $(\hat{\mu}_{t})_{t\geq 0}$ with $\hat{\mu}_{t}\in \mathcal{M}(\mathbb{R}^{D})$ we denote by $\mathcal{M}_{t\geq 0}(\mathbb{R}^{D})$. 
Subsequently, the notion of ensemble enrichment on such families of discrete measures in combination with an enrichment time $t_a$ and a number $a\in\mathbb{N}$ of requested additional particles is defined.

\begin{definition}
\label{def:enrichment_scheme}
Let $(\Omega,\sigma, \mathbb{P})$ be an abstract probability space.
Then, a map $\mathcal{E}\colon \mathcal{M}_{t\geq 0}(\mathbb{R}^{D})\times \Omega \times \mathbb{R}_{+} \times \mathbb{N} \longrightarrow \mathcal{M}(\mathbb{R}^{D})$ is called an \textnormal{ensemble enrichment scheme} if the following condition $\mathrm{(C)}$ is satisfied.
\begin{itemize}
    \item[$\mathrm{(C)}$] If $\hat{\mu}_{t}$ is determined by an ensemble of $b$ 
    particles at $t=t_a$, then for $a \in \mathbb{N}$, $\mathcal{E}[((\hat{\mu}_t)_{t\geq 0},\omega,t_a,a)]$ is a discrete measure determined by $b+a$ particles.
\end{itemize}
\end{definition}
\noindent
The dependence on $\omega\in\Omega$ models random effects, e.g. introduced through random perturbations or random selection of particles. In practice, such dependence is given as 
\begin{equation}\label{eq:random_enrichment}
\mathcal{E}[((\hat{\mu}_t)_{t\geq 0},\omega,t_a,a)] 
= 
\mathcal{E}[((\hat{\mu}_t)_{t\geq 0}, \zeta(\omega),t_a,a)] 
\end{equation}
with some random variable $\zeta$ independent of the $(\hat{\mu}_t)_{t\geq 0}$. If (C) is satisfied, we call $\hat{\mu}_{t_a}$ and $\{y_{t_a}^{(i)}\}$ the \textit{original measure} and \textit{original batch}, respectively.
Moreover, $\mathcal{E}[((\hat{\mu}_t)_{t\geq 0},\omega,t_a,a)]$ is the \textit{enriched measure} and its corresponding sample batch the \textit{enriched batch}.
When the associated family of measures is clear from the context, we call this the application of an enrichment scheme to a batch $\{y_{t_a}^{(i)}\}$.
This can be considered the more natural viewpoint and appeals to our intuitive understanding of enrichment as adding samples to an existing batch.
The reason we formally define an enrichment scheme as a map acting on a family of measures is two-fold.
First, we want an enrichment to be able to rely on past and future values of the process $\{y_t^{(i)}\}$.
Second, the operation on measures instead of particle ensembles allows to view the solution provided by our LIDL method as a time-continuous process.
To make this point clear, note that we cannot define the result of the process described in Section \ref{sec:setting} as a time-dependent ensemble $\{y_t^{(i)}\}_{i=1}^B$ since the batch size $B$ changes over time.
This however can be achieved in terms of a family of ensemble distributions $(\hat{\mu}_t)_{t\geq 0}$.
This formulation is agnostic with regard to the current batch size, which is made precise in Definition \ref{def:lidl}.

Having formally defined what an ensemble enrichment scheme is, the question what constitutes a \textit{good} enrichment strategy arises immediately.
Since in the end we are concerned with convergence of the sampling scheme, the distance of the enriched measure to the posterior has to be controlled.
In the remainder of this paper we use the \textit{Kantorovich–Rubinstein metric} (or \textit{p-Wasserstein distance}) of measures for $p=2$.
Let $\mathcal{D}_p(\mathcal{V})$ be the space of measures on a metric space $\mathcal{V}$ with finite second moments.
The 2-Wasserstein distance of measures $\mu,\nu\in\mathcal{D}_2(\mathbb{R}^{D})$ is then defined by
\begin{equation}
    \begin{aligned}
    \label{eq:wasserstein_metric}
\mathcal{W}_{2}(\mu,\nu) = \left[ \min\limits_{\pi\in\mathcal{D}(\mathbb{R}^{D}\times\mathbb{R}^{D})}\{ \langle \pi, c_2\rangle \colon \pi_1 =\mu , \pi_2 = \nu \} \right]^{1/2},
    \end{aligned}
\end{equation}
where $\pi_1 = \int_{\mathbb{R}^{D}} \mathrm{d}\pi(\cdot,y)$ and $\pi_1 = \int_{\mathbb{R}^{D}}  \mathrm{d}\pi(x,\cdot)$ are the marginals of the transport plan and $c_2(x,y) = \frac{1}{2}| x-y |^2$. The $2$-Wasserstein metric space is denoted by $\mathcal{W}_2(\mathbb{R}^D):=(\mathcal{D}_2(\mathbb{R}^D),\mathcal{W}_2)$.
The following enrichment strategies aim to approximately follow the flow towards the posterior measure, as initiated by the original measure.

\subsection{Enrichment schemes}
\label{sec:enrichment schemes}

Let $a,b\in\mathbb{N}$ with $a\leq b$ and let the measure $\hat{\mu}\in\mathcal{M}(\mathbb{R}^{D})$ be determined by $b\in\mathbb{N}$ particles $y^{(i)}$. Furthermore, let the random variable $\zeta_{a}$ encode the uniformly at random selection of $a$
particles out of $\{y^{(i)}\}_{i=1}^{B}$ with realizations denoted by $\zeta_{a}(\omega)[\hat{\mu}]$ being a discrete measure determined by $a$ particles. For instance, if the first $i=1,\ldots,a$ particles have been selected by $\zeta_{a}$, the resulting measure is $\frac{1}{a}\sum_{i=1}^{a} \delta_{y^{(i)}}$.



\subsubsection{Slicing}
\label{sec:slicing}
The idea of (time)-\textit{slicing} relies on the idea to enrich the batch of size $b$ at time $t_a$ by adding particles from other batches associated to a selection of finitely many ensemble distributions $\hat{\mu}_t$ from the family $(\hat{\mu}_t)_{t\geq 0}$.  A canonical slicing is defined at time points around $t_a$. For this let, $0 < \Delta t < t_a$ and assume the batch size of the family of measures to be equal to $b \geq a$ in the neighborhood $[t_a - \Delta t,t_a + \Delta t]$. \\

\noindent
We define the \textit{forward slicing} enrichment $\mathcal{E}_{+\Delta t}$ via
\begin{equation}\label{eq:forward_slice}
    \mathcal{E}_{+\Delta t}((\hat{\mu}_t)_{t\geq 0}, \omega, t_a, a) = \dfrac{1}{b+a}\left( b\hat{\mu}_{t_a} + a\zeta_{a}(\omega)\left[\hat{\mu}_{t_a + \Delta t}\right] \right). 
\end{equation}

In the idealized setting where $b=a$ and $\mathcal{W}_2(\hat{\mu}_t,\mu_{\ast})$ is monotonically decreasing in $t$, we get that the 2-Wasserstein distance of \eqref{eq:forward_slice} to the posterior is bounded by $\mathcal{W}_2(\hat{\mu}_{t_a},\mu_{\ast})$ and hence, \eqref{eq:forward_slice} should at least preserve the distance to the posterior at time $t_a$.
This property is obviously desirable but it comes at the cost of carrying out the forward model evaluations necessary to compute the measure $\hat{\mu}_{t_a+\Delta t}$ from $\hat{\mu}_{t_a}$. Note that in practice the measure family is only available for $t\leq t_a$ without additional computational burden.\\

\noindent
To avoid the additional forward model calls, one may instead use a \textit{backward slicing} scheme $\mathcal{E}_{-\Delta t}$ defined by
\begin{equation}\label{eq:backward_slice}
    \mathcal{E}_{-\Delta t}((\hat{\mu}_t)_{t\geq 0},\omega, t_a, a) = \dfrac{1}{b+a}\left( b\hat{\mu}_{t_a} + a\zeta_a(\omega)\left[\hat{\mu}_{t_a - \Delta t}\right] \right). 
\end{equation}
Here, no extra forward model evaluations are needed since $\hat{\mu}_{t_a-\Delta t}$ is readily available at time $t_a$.
However, compared to the forward slicing scheme, we only get a 2-Wasserstein distance bound proportional to $\mathcal{W}_2(\hat{\mu}_{t_a-\Delta t},\mu_{\ast})$ when assuming monotonic decrease of the Wasserstein distance in $t$.

\subsubsection{Diffusion propagation}
\label{sec:diffusion_step}

Recall that computing the measure $\hat{\mu}_{t_a+\Delta t}$ in the forward slicing scheme \eqref{eq:forward_slice} involves propagating $\hat{\mu}_{t_a}$ through the underlying Langevin process \eqref{eq:general_SDE} from time $t_a$ to $t_a+\Delta t$.
For small $\Delta t$, we can define an  approximation to $\hat{\mu}_{t_a+\Delta t}$ by propagating the particle ensemble $\{y_{t_a}^{(i)}\}$ associated with $\hat{\mu}_{t_a}$ using only the diffusion part of \eqref{eq:general_SDE},
\begin{equation}\label{eq:diffusion_propagation_samples}
    \mathrm{d}y^{(i)}_t =  \Gamma(Y_t)\mathrm{d}W^{(i)}_t.
\end{equation}
We encode the effect of this \textit{diffusion propagation} on $\hat{\mu}_{t_a}$ by a random variable $\zeta^{\Delta t}_W$ with realizations $\zeta^{\Delta t}_{W}(\omega)[\hat{\mu}_{t_a}] \approx \hat{\mu}_{t_a+\Delta t}$. By $\zeta_{a} \circ \zeta^{\Delta t}_{W}$ with realizations $(\zeta_{a} \circ \zeta^{\Delta t}_{W})(\omega)[\hat{\mu}_{t_a}] = (\zeta_{a}(\omega) \circ \zeta^{\Delta t}_{W}(\omega)) [\hat{\mu}_{t_a}]$,
we denote the concatenation with the previously defined random selection. Hereby, we choose $\zeta_a$ and $\zeta_W^{\Delta t}$ to be independent.\\

\noindent
The diffusion propagation scheme is then defined by
\begin{equation}\label{eq:diffusion_propagation}
    \mathcal{E}^{\mathrm{diff}}_{+\Delta t}((\hat{\mu}_t)_{t\geq 0}, \omega, t_a, a) = \dfrac{1}{b+a}\left( b\hat{\mu}_{t_a} + a(\zeta_a \circ \zeta^{\Delta t}_{W})(\omega)\left[\hat{\mu}_{t_a}\right] \right). 
\end{equation}
This scheme entirely avoids any forward model calls, while also not relying on the history of the process as in the backward slicing scheme. This approximation is also motivated by the numerical time discretization of \eqref{eq:general_SDE}. To that end, consider an Euler-Maruyama discretization $\tilde{y}^{(i)}_t$ of the process $y_t^{(i)}$ with time step $\Delta t$ defined through
\begin{equation}\label{eq:euler_step}
    \tilde{y}^{(i)}_{t+\Delta t} = \tilde{y}^{(i)}_{t} -\Delta t A(\tilde{Y}_{t})\nabla_{\tilde{y}^{(i)}_{t}}\mathcal{V}(\tilde{Y}_{t}) + \sqrt{\Delta t} \Gamma(\tilde{Y}_{t}) \xi^{(i)}_{t}.
\end{equation}
Here, $\xi_t \sim \mathcal{N}(0,I_d) $ are i.i.d. increments and $\tilde{Y}_t = \operatorname{vec}(\tilde{y}_t^{(1)},\ldots,\tilde{y}_t^{(b)})$.
Vice-versa a discretization of \eqref{eq:diffusion_propagation_samples} leads to the relation
\begin{equation}\label{eq:diffusion_step}
    \tilde{y}^{(i)}_{t+\Delta t} = \tilde{y}^{(i)}_{t} + \sqrt{\Delta t} \Gamma(\tilde{Y}_{t}) \xi^{(i)}_{t},
\end{equation}
where we used the same notation to underline similarities to \eqref{eq:euler_step}.
While in \eqref{eq:euler_step} the drift part scales with $\Delta_t$, the diffusion part scales with $\sqrt{\Delta t}$, and hence an approximation via \eqref{eq:diffusion_step} is justified provided $\Delta t \ll 1$ and $A(Y_{t})\nabla_{y^{(i)}_{t}}\mathcal{V}(Y_{t})$ being bounded.
Consequently, the diffusion propagation approximates a forward slice, without the computational burden of evaluating $\mathcal{V}$.
Note that, in this time-discrete setting, the randomness of the propagation, formerly represented by $\zeta_{W}^{\Delta t}$ is now encoded in the increments $\xi_t^{(i)}$ for $t=t_a$.





In this time-discrete setting, the randomness of the propagation, which was formerly encoded in $\zeta_{W}^{\Delta t}$ is now encoded in the increments $\xi_k^{(i)}$.

\subsubsection{Random kicks}
\label{sec:random_kicks}

The discrete viewpoint taken in \eqref{eq:euler_step} and \eqref{eq:diffusion_step} allows for the interpretation of the diffusion step as a special case of a \textit{random kicks} method
\begin{equation}\label{eq:random_kick}
    \mathcal{E}_{\eta}((\hat{\mu}_t)_{t\geq 0}, \omega, t_a, a) = \dfrac{1}{b+a}\left( b\hat{\mu}_{t_a} + a(\zeta_a \circ \zeta_{\eta})(\omega)\left[\hat{\mu}_{t_a }\right] \right),
\end{equation}
where the random variable $\zeta_{\eta}$ encodes adding perturbation noise $\eta \sim \bigotimes_{i=1}^b\mathcal{D}(\mathbb{R}^{D})$ to the particles $\{y^{(i)}_{t_a}\}$ determining $\hat{\mu}_{t_a}$.
Note that \eqref{eq:diffusion_step} can be seen as a highly informed choice of noise, utilizing the underlying process.
In the case of the ALDI method \eqref{eq:aldi}, it preserves the covariance structure of the ensemble.

The question arises why one would consider other less informed random kicks at all. We note that in the low batch size regime one important strength of diffusion propagation, namely preserving the covariance rank as in the case of ALDI, turns out to be detrimental.
To see this, recall \eqref{eq:aldi} with covariance matrix $C(Y_t)$ given by \eqref{eq:emp_covariance} and generalized non-symmetric square root $C(Y_t)^{1/2}$ given by \eqref{eq:nonsym_sqrt}.
Consider now a batch size $b<D$ such that the covariance matrix has at most rank $b$ and is hence not positive definite.
While any proper perturbation noise $\eta$ independent of the particle ensemble (with sufficiently large enrichment size $a$) leads to a positive definite covariance matrix with probability 1, \eqref{eq:euler_step} and \eqref{eq:diffusion_step} produce additional samples in the range of the covariance matrix, not increasing its rank at all.
This means that the covariance matrix is not positive definite at any point in the future (up to numerical instability) and particles remain in the linear subspace associated to the range of the covariance.
Hence, preserving the rank of the covariance via methods like slicing and diffusion propagation only makes sense for sufficiently large batch sizes $b\geq D$, leading to full rank with probability $1$. As a simple and purely heuristic choice of a single random kick for the low batch size regime, we propose using scaled Gaussian noise
\begin{equation}
    \tilde{y}^{(i)}_{t + \Delta t} = \tilde{y}^{(i)}_{t} + \sqrt{\Delta t} \xi^{(i)}_{k}
\end{equation}
for i.i.d. increments $\xi^{(i)}_k \sim \mathcal{N}(0,I_d)$ and a step size $\Delta t $ depending on $C(\tilde{Y}_t)$. We leave the discussion of other informed choices of noise as a topic for future work.

Finally, we note that all of these methods can be extended to the case $a > b$, e.g., by adding together multiple forward/backward slices at different times $t_a \pm \Delta t_i$. The corresponding modifications of \eqref{eq:forward_slice}, \eqref{eq:backward_slice}, \eqref{eq:diffusion_propagation} and \eqref{eq:random_kick} are straightforward and left for the reader.

\subsubsection{Generalized transport approach}
\label{sec:generalized_transport}
A conceptually different approach relies on the intermediate learning of random variables, that approximately follow the distribution of $y_t\sim\pi_t$ given by \eqref{eq:EKS_mean_field_limit} and \eqref{eq:eks_fokker_planck} at time point $t=t_a$. This allows for a fast generation of samples due to a functional representation, which takes the form
\begin{equation}
    \label{eq:modelclass}
      \mathcal{M}_t(x_t) \apprxd  y_t 
\end{equation}
with a suitable model class $\mathcal{M}_t$ and some auxiliary random variable $x_t$.
Here, the approximation quality 
should be controlled in the same metric used to analyze the convergence of the particle propagator.
Once such representation is found for a fixed $t=t_a$, samples can be drawn by sampling from $x_t$ and propagating through $\mathcal{M}_t$.

A setup of particular importance arises for $t=T\gg 1$. If $\mu_T$ is close to $\mu_\ast$, then the left-hand side of \eqref{eq:modelclass} provides an approximate functional access to the posterior distribution. This concept is similar to the case of transport maps, where $x_T$ is distributed with respect to the prior distribution and $\mathcal{M}_T$ is a diffeomorphism \cite{villani2009optimal,rezende2015variational,marzouk2016introduction,brennan2020greedy}.
However, here we relax the assumption of bijectivity or continuity properties of $\mathcal{M}$ in order to enable accurate approximations of more involved distributions such as multimodal ones even when $x_T$ is a unimodal distribution.
Moreover, $x_T$ is not necessarily distributed 
with respect to the prior distribution but rather defined on some latent space.

Another approach based on \eqref{eq:modelclass} is denoted as \textit{sequential learning of generalized transport} and described in the following.
Let $0=t_0<t_1<\ldots<t_L<\infty$ and 
$$
x_{t_\ell} = y_{t_{\ell-1}},\quad \ell = 1,\ldots, L.
$$
This design leads to the form
$$
y_{t_L} \apprxd M_{t_L}\circ \ldots \circ M_{t_1}(y_0)
$$
for a suitable model class $\mathcal{M}_{t_\ell}$.
Hence, the compositional structure imitates the associated particle propagator flow.
This concept appears to be closely related to \textit{stochastic normalizing flows} \cite{wu2020stochastic,hagemann2022stochastic}.
The discussion of such a compositional approach is subject to future research.

Another special case of \eqref{eq:modelclass} is of the form 
\begin{equation}
    \label{eq:constant_aux}
    y_t \apprxd \mathcal{M}_t(x).
\end{equation}
In particular, the auxiliary random variable remains the same over the time horizon and only the model class is updated. 
A particular design for \eqref{eq:constant_aux} is realized by \textit{generative adversarial networks} (GANs), where typically $x$ corresponds to a standard normal multivariate Gaussian distribution.
As an alternative, the use of a \textit{compressed Wasserstein polynomial chaos expansion} (WPCE) was proposed in \cite{gruhlke2022low}.
Here, $x$ determines the family of orthonormal stochastic polynomials used for the approximation. 

In order to realize the representation \eqref{eq:constant_aux}, one can utilize techniques from computational optimal transport in the framework of unsupervised learning.
We define the debiased Sinkhorn divergence \cite{feydy2019interpolating} based on \eqref{eq:wasserstein_metric} 
for $\epsilon >0$ by
\begin{equation}
    \begin{aligned}\label{eq:sinkhorn_div}
    \mathcal{S}_{\epsilon}(\mu,\nu) = \mathcal{W}_{c,\epsilon}(\mu,\nu) - \dfrac{1}{2}\left( \mathcal{W}_{c,\epsilon}(\mu,\mu) + \mathcal{W}_{c,\epsilon}(\nu,\nu) \right),
    \end{aligned}
\end{equation}
for measures $\mu,\nu \in\mathcal{D}_2(\mathbb{R}^{D})$, where
\begin{equation}
\begin{aligned}
    \mathcal{W}_{c,\epsilon}(\mu,\nu) = &\min_{\pi\in\mathcal{D}(\mathbb{R}^{D}\times\mathbb{R}^{D})} \langle \pi, c \rangle + \epsilon \mathrm{KL}(\pi,\mu\otimes\nu),\\
    &\textnormal{subject to }\quad \pi \geq 0, \quad \pi_1 = \mu, \quad \pi_2 = \nu.
    \end{aligned}
\end{equation}
Here, KL denotes the Kullback-Leibler divergence.
Then, samples $y_t$ define an discrete measure $\hat{\mu}_t$ and samples from $x$ propagated through $\mathcal{M}_t=\mathcal{M}[\theta_t]$ define a discrete measure $\hat{\nu}[\theta_t]$ depending on the parameter $\theta_t$.
The desired coefficient $\theta_t$ is defined as the minimizer of
\begin{equation}
    \label{eq:sinkhornloss}
    \min\limits_{\theta}\mathcal{S}_\epsilon(\hat{\mu}_t, \hat{\nu}[\theta]).
\end{equation}

%% file: 04_error_analysis.tex
\section{Theoretical Foundations}\label{sec:theory}

We henceforth use ALDI~\eqref{eq:aldi} as our particle propagator of choice. Together with a chosen enrichment scheme from Section \ref{sec:enrichment}, this allows to define our method in a rigorous way.

\begin{definition}[LIDL]\label{def:lidl}
Set $L\in\mathbb{N}_0$ and define the index set $I_L = \{0,1,2,\ldots,L\}$.
Moreover, let
\begin{enumerate}[label=(\roman*)]
    \item $(t_\ell)_{\ell\in I_L}$, $t_{\ell}\geq t_0 = 0$ for all $\ell$ be a strictly monotonically increasing sequence of time points with $t_{L+1} = \infty$,
    \item $(b_\ell)_{\ell\in I_L}$, $b_\ell\in\mathbb{N}$ be a sequence of batch sizes,
    \item $\{y_0^i\}_{i=1}^{b_0}$ be an initial sample batch,
    \item $\mathcal{E}$ be an enrichment strategy.
\end{enumerate}
Then, an instance of a LIDL run with parameters $((t_\ell)_{\ell}, (b_\ell)_{\ell}, \{y_0^i\}, \mathcal{E} )$ applied to the inverse problem defined by \eqref{eq:posterior} produces a family of ensemble distributions $(\hat{\mu}_{t})_{t\geq0}$ on $\mathbb{R}^{D}$ by repeating the following two steps for $\ell=0,\ldots,L$:
\begin{itemize}
    \item[(Step 1)] Solve \eqref{eq:aldi} with initial conditions $\{y_\ell^i\}$ for $0\leq t \leq t_{\ell+1}-t_\ell$.
    Denote the resulting solution by $\{y_{\ell,t}^i\}$ and the corresponding ensemble distribution by $\hat{\mu}_{\{y_{\ell,t}^i\}}$. Set
    \begin{equation}\label{eq:lidl_measure}
       \hat{\mu}_{t} = \hat{\mu}_{\{y_{\ell,t-t_{\ell}}^i\}} \qquad \text{for} \quad t\in [t_\ell,t_{\ell+1}).
    \end{equation}
    \item[(Step 2)] If $\ell\leq L-1$, get the next initial sample batch $\{y_{\ell+1}^i\}$ by applying the ensemble enrichment $\mathcal{E}$ with $b_{\ell+1}$ new samples to $\{y_{\ell,t_{\ell+1}}^i\}$. 
\end{itemize}
The random process $(\hat{\mu}_t)_{t\geq0}$ is called the solution of LIDL.
A visualization of steps 1 and 2 are depicted in Figure~\ref{fig:scheme_lidl}.
\end{definition}

\input{illustration_lidl}

\begin{remark}\label{rem:aldi_is_lidl}
Some comments on the definition are in order.
\begin{enumerate}
\item Setting $t_{L+1}=\infty$ is a formality, guaranteeing that LIDL returns a measure $\hat{\mu}_t$ for all $t\geq 0$.
The last stage is equivalent to solving \eqref{eq:aldi} with initial conditions $\{y_L^i\}$.
\item When $L=0$, the LIDL solution becomes identical to the solution of \eqref{eq:aldi} for all $t$.
Hence, ALDI can be seen as a special case of LIDL.
\item We could define the method more generally by admitting a sample propagator like EKS, ALDI, etc. as an additional parameter.
For the sake of simplicity, we only work with ALDI and neglect this dependence in the rest of the paper.
\end{enumerate}
\end{remark}

\subsection{Convergence analysis in the linear case}

Throughout this section, we consider the Bayesian inverse problem \eqref{eq:inverseproblem} in the special case of a Gaussian prior and a linear forward map.
Hence, let $\mathcal{G}(\cdot) = A\cdot$ for some $A\in\mathbb{R}^{K\times D}$ and
\begin{equation}\label{eq:gaussian_prior}
    \pi_{\mathrm{prior}}(y) \propto \exp\left( -\dfrac{1}{2}(y-y_0)^{\intercal} \Gamma_0^{-1} (y-y_0) \right) = \exp\left( -\dfrac{1}{2}|y-y_0|^2_{\Gamma_0} \right),
\end{equation}
 with prior mean $y_0\in\mathbb{R}^{D}$ and prior covariance matrix $\Gamma_0\in\mathbb{R}^{D\times D}$. In this special case, the posterior is again Gaussian with density
\begin{equation}
    \pi_{\ast}(y) \propto \exp\left( -\dfrac{1}{2}|y-y^\ast|^2_{P^{-1}} \right),
\end{equation}
with the posterior precision matrix 
\begin{equation}
    P = C(\mu_\ast)^{-1}  = A^{\intercal}\Gamma^{-1}A + \Gamma_0^{-1}
\end{equation}
and the posterior mean
\begin{equation}\label{eq:posterior_precision}
    y^\ast = \mathbb{E}_{\mu_\ast} = P^{-1}(A^{\intercal}\Gamma^{-1}\delta + \Gamma_0^{-1}y_0).
\end{equation}
In the following, we denote by 
$\lambda_{\min}(M)$ the smallest eigenvalue of a symmetric positive definite matrix $M\in\mathbb{R}^{D,D}$.

As a first step in the convergence analysis, we aim for a consistency result in the expected 2-Wasserstein distance $\mathcal{W}_2$ between the measures $\hat{\mu}_t$ generated by LIDL and the posterior measure $\mu_{\ast}$.
By consistency we mean that for any $\delta > 0$ there is a configuration  $((t_\ell)_{\ell}, (b_\ell)_{\ell}, \{y_0^i\}, \mathcal{E} )$ of LIDL and a time $T_{\delta} > 0$ such that
\begin{equation}\label{eq:W2_consistency}
    \mathbb{E}[\mathcal{W}_2(\hat{\mu}_{T_{\delta}},\mu_{\ast})] \leq \delta.
\end{equation} 
We have established in Remark \ref{rem:aldi_is_lidl} that ALDI can be seen as a special case of LIDL with $L=0$ enrichment steps.
Hence, consistency as defined above can be achieved in a trivial manner if ALDI can be shown to be B-T-consistent, meaning that for $\delta >0$ there are $T_{\delta}, B_{T_{\delta}}$ such that the ensemble distribution $\hat{\mu}_t$ generated by ALDI satisfies \eqref{eq:W2_consistency}.
The following theorem provides sufficient conditions for consistency in this sense.

\begin{theorem}[B-T-consistency of ALDI]
\label{thm:B-T-convergence_ALDI}
        Suppose $\mathcal{G}(\cdot) = A\cdot$ is linear and $\pi_{\mathrm{prior}}$ is given by \eqref{eq:gaussian_prior}.
        Furthermore, let $\pi_0 \in \mathcal{C}^2$ be a density with bounded higher moments and for $B\in\mathbb{N}$ let $\{y_t^i\}^B_{i=1}$ be the solution of \eqref{eq:aldi} with initial condition $\{y_0^i\}_{i=1}^B$ drawn i.i.d from $\pi_0$.
        Let $\delta>0$ and $T>0$ be such that $\mathcal{W}_2(\pi(T),\pi_{\ast}) < \delta$ where $\pi(T)$ is the solution of \eqref{eq:eks_fokker_planck} with initial condition $\pi_0$. Furthermore assume that 
        \begin{equation}\label{eq:ev_condition}
            \frac{\lambda_{\text{min}}(P)\lambda_{0}(t)}{2}\ \geq 1, \quad ~\text{for all }~ 0 \leq t \leq T,
        \end{equation}
        where $\lambda_0(t) = (\lambda_{\min}(C(Y_t))^{1/2} + \lambda_{\min}(C(\pi(t)))^{1/2})^{2}$. 
        Then, there exists $T_{\delta} \leq T$ and $B_{T_{\delta}} > 0$ such that the solution $\{ y_t^i \}_{i=1}^B$ with $B=B_{T_{\delta}}$ and its corresponding ensemble distribution $\hat{\mu}_{T_{\delta}}^{B_{T_{\delta}}}$ satisfy
         \begin{equation}\label{eq:B-T-bound}
            \mathbb{E}[\mathcal{W}_2(\hat{\mu}^{B_{T_{\delta}}}_{T_{\delta}},\mu_{\ast})] \leq \delta.
        \end{equation}
    \end{theorem}
    
    \begin{proof}
    The proof can be found in Appendix~\ref{sec:proof_btconv}.
    \end{proof}
    
    \begin{remark}
    Since $\pi(t)$ converges to $\pi_{\ast}$ exponentially fast in 2-Wasserstein distance (see Theorem \ref{thm:carrillo} in the appendix), we can always find a $T > 0$ such that $\mathcal{W}_2(\pi(T),\pi_{\ast}) < \delta$.
    The second assumption \eqref{eq:ev_condition} on $T$ is a technical one and comes from an application of the Ando-Hemmen inequality (see the proof of Lemma 5.4 in \cite{DingLi21}).
    \end{remark}
\begin{remark}
The proof uses a triangle argument.
For $0\leq t \leq T$ this leads to a bound of the form
\begin{equation}
    \begin{aligned}\label{eq:B-T-bound-full}
    \mathbb{E}[\mathcal{W}_2(\hat{\mu}^{B}_{t},\mu_{\ast})] \leq c e^{-t}\mathcal{W}_2(\pi_0,\pi_{\ast}) + c(t,D,\epsilon) \begin{cases}
    B^{-1/2+\epsilon},& \quad D\leq 4,\\
    B^{-2/D},& \quad D > 4,
    \end{cases}
    \end{aligned}
\end{equation}
where $0<\epsilon<1/2$, $c>0$ is a constant depending only on the initial density $\pi_0$ and the posterior density $\pi_{\ast}$ and  $c(t,D,\epsilon)>0$ is a constant depending on $t, D, \epsilon$.
The first term comes from the Fokker-Planck solution and decays exponentially fast for $t\rightarrow T$.
The constant $c(t,D,\epsilon)$ for the remaining term however grows exponentially with $T$ and we hence obtain no monotonicity for $t\rightarrow T$.
Nevertheless, the bound \eqref{eq:B-T-bound-full} yields monotonicity in the following sense: if \eqref{eq:B-T-bound} is satisfied for some $T_{\delta}$, $B_{T_{\delta}}$, then for any $t \in [T_{\delta},T]$ we find a $B_t \geq B_{T_{\delta}}$ such that \eqref{eq:B-T-bound} is also satisfied for $t,B_t$.
Note that this already provides a theoretical motivation for an ensemble enrichment: to achieve the same error bound $\delta$ for times $t > T_{\delta}$, we may have to increase the batch size.
\end{remark}
    
\begin{corollary}[B-T-consistency of LIDL]
\label{cor:B-T-convergence_LIDL}
Under the same conditions as in Theorem \ref{thm:B-T-convergence_ALDI}, for every $\delta >0$ there exists a configuration $((t_\ell)_{\ell}, (b_\ell)_{\ell}, \{y_0^i\}, \mathcal{E} )$ of LIDL and a time $T_{\delta} > 0$ such that the solution $\hat{\mu}_t$ defined by \eqref{eq:lidl_measure} satisfies
\begin{equation}
    \mathbb{E}[\mathcal{W}_2(\hat{\mu}_{T_{\delta}},\mu_{\ast})] \leq \delta.
\end{equation}
\end{corollary}
\begin{proof}
Choosing $L=0$ and $T_{\delta}, b_0=B_{T_{\delta}}$ as in Theorem \ref{thm:B-T-convergence_ALDI} yields one such configuration.
\end{proof}

A full ALDI run is not the only configuration with B-T-consistency.
In particular, we require the last sample batch $\{y^i_{L}\}$ to satisfy the conditions of Theorem \ref{thm:B-T-convergence_ALDI}.
One of the conditions is that $\{y^i_{L}\}$ is drawn i.i.d. from some $\mathcal{C}^2$-density with bounded higher moments.
This condition comes from the smoothness required of a strong solution of the Fokker-Planck equation.
This can be formulated as an assumption on the enrichment scheme.

\begin{definition}[Consistent Enrichment]
\label{def:consistent_enrichment}
We call an enrichment strategy $\mathcal{E}$ \textrm{consistent} if it creates enriched ensembles drawn i.i.d. from some $\mathcal{C}^2$ density with finite higher moments.
\end{definition}
As the following remark shows, this is a purely theoretical condition.
In practice, any enrichment scheme can be seen as an arbitrarily close approximation of a consistent one.

\begin{remark}
\label{rem:mollifying} 
The subset of measures in $\mathcal{D}_2(\mathbb{R}^D)$ with $\mathcal{C}^2(\mathbb{R}^D)$--Lebesque density is dense in $\mathcal{W}_2(\mathbb{R}^D)$, see Lemma \ref{lemma:dense_F_p_k_W_p} for $p=2$ and $k=2$.
Now let $\mathcal{E}$ be an arbitrary ensemble enrichment scheme taking the original batch $\{y_1^i\}_{i=1}^{b_1}$ and generating an enriched batch $\{y_2^i\}_{i=1}^{b_2}$ for $b_2 > b_1$ with ensemble distribution $ \hat{\mu}$.
Then, for any $\epsilon>0$ there exists a regular measure $\mu_\epsilon\in\mathcal{W}_2(\mathbb{R}^D)$ with Lebesque density in $\mathcal{C}^2(\mathbb{R}^D)$ such that $\mathcal{W}_2(\hat{\mu},\mu_\epsilon)<\epsilon$.
Hence, the enriched ensemble approximately can be seen as being drawn i.i.d. from a distribution with $C^2$ density.
\end{remark}

\vspace{1em}

\noindent
With this notion of consistency in place, the following result immediately follows.

\begin{corollary}[B-T-consistency of LIDL]
\label{cor:B-T-convergence_consistent}
Assume the conditions of Theorem \ref{thm:B-T-convergence_ALDI} to be satisfied.
Let $\mathcal{E}$ be a consistent ensemble enrichment scheme in the sense of Definition \ref{def:consistent_enrichment}.
Then, for every configuration $((t_{\ell}, (b_\ell)_{\ell}, \{y_0^i\}, \mathcal{E} )$ of LIDL with $b_0,\ldots,b_L$ sufficiently large and $\delta > 0$, there exists a time $T_{\delta} > 0$ such that the solution $\hat{\mu}_t$ defined by \eqref{eq:lidl_measure} satisfies
\begin{equation}\label{eq:B-T-convergence_consistent}
    \mathbb{E}[\mathcal{W}_2(\hat{\mu}_{T_{\delta}},\mu_{\ast})] \leq \delta.
\end{equation}
\end{corollary}

\begin{proof}
Since $\mathcal{E}$ is consistent, the last batch $\{y_L^i\}$ is drawn i.i.d. from a $\mathcal{C}^2$-density with bounded higher moments.
Moreover, the conditions of Theorem \ref{thm:B-T-convergence_ALDI} apply to the starting ensemble $\{y_L^i\}$ by assumption.
Hence, there exists a $T_{\delta}$ and $B_{T_\delta}$ 
such that \eqref{eq:B-T-convergence_consistent} holds true,
provided $\overline{b}_L=b_0+\ldots +b_L\geq B_\delta$.
\end{proof}

\subsection{Convergence of homotopy approach}
\label{sec:homotopy}
Motivated by the particular homotopy approach from Section \ref{sec:setting}, define \textit{switch points} $0=s_0<s_1<\ldots < s_K =1$ for some $K\in\mathbb{N}$.
Let $L_\mathrm{log}^1\left(\mathbb{R}^{D}\right) := \left\{ \phi\colon\mathbb{R}^{D}\to \mathbb{R} : \exp(-\phi)\in L^1(\mathbb{R}^{D}) \right\}$.
In what follows, the constant $Z_{s}^{-1}$ is a generic $L^1$ normalization constant enumerated by $s\in[0,1]$ and associated with a measure $\mu_s$.
First, we define a class of feasible homotopy functions, which we call $\mathcal{W}_2$-\textit{stable homotopies}.

\begin{definition}[$\mathcal{W}_2$-stable homotopy]\label{def:stable_homotopy}
Let $\Phi_i\colon\mathbb{R}^{D}\to\mathbb{R}$ such that $\mu_i:=Z_{i}^{-1}\exp(-\Phi_i)\in\mathcal{W}_2(\mathbb{R}^{D})$ for $i=0,1$.
A mapping $\mathcal{H}\colon [0,1]\to L_\mathrm{log}^1\left(\mathbb{R}^{D}\right)$ is denoted a $\mathcal{W}_2$-\textit{stable homotopy} between $\Phi_0$ and $\Phi_1$ if
\begin{enumerate}[label=(\roman*)]
    \item $\mathcal{H}(0) = \Phi_0,\quad \mathcal{H}(1) = \Phi_1, $\hfill (interpolation)
    \item $\mu_s:=Z_s^{-1}\exp(-\mathcal{H}(s))\in \mathcal{W}_2\left(\mathbb{R}^{D}\right), \quad s\in(0,1)$,\hfill (consistency)
    \item $    \mathcal{W}_2(\mu_{s_1}, \mu_{s_2}) \leq \phi(|s_1-s_2|)$, \quad $s_1,s_2\in[0,1],$ 
    \hfill (stability)
\end{enumerate}
for continuous $\phi\colon[0,1]\to\mathbb{R}$ with $\phi(s)\to 0 $ as $s\searrow 0$. 
\end{definition}

With this preparation, let $\mathcal{H}$ be a $\mathcal{W}_2$-stable homotopy between some auxillary potential $\mathcal{H}(0) = \Psi$ and the posterior potential $\mathcal{H}(1) = \Phi$. In order to still be able to work with It\^o diffusion processes as in \eqref{eq:general_SDE}, we aim to define an inhomogeneous drift term $f_{\mathcal{H}} = f_\mathcal{H}(t, \cdot)$ that is piecewise constant in $t$. For this, we define a time horizon partition $\mathcal{T}\colon\{0,\ldots,K\}\to \mathbb{R}_+$ with
\begin{equation}
    \mathcal{T}(0) = 0,\quad \mathcal{T}(k)<\mathcal{T}(k+1), \quad k=0,\ldots, K-1.
\end{equation}

Then, for $Y=\mathrm{vec}\left(y^{(i)}\right)_i\in\mathbb{R}^{DB}$ with $B\in\mathbb{N}$ and $\mathcal{T}(K+1)=\infty$, let
\begin{equation}
    \label{eq:pcdrift}
    f_\mathcal{H}(t, Y) := 
    f( \mathcal{H}(s_k), Y), \quad t \in [\mathcal{T}(k),\mathcal{T}(k+1)), \quad k = 0,\ldots, {K}.
\end{equation}
Then, for for $k=0,\ldots, K$ consider the time-partitioned dynamics
\begin{equation}
    \label{eq:auxillary_propagator}
    \mathrm{d}y_t^{(i)} = f(\mathcal{H}(s_k), Y_t)\mathrm{d}t + \Gamma(Y_t)\mathrm{d}W_t^{(i)}, \qquad t \in [\mathcal{T}(k),\mathcal{T}(k+1)),
\end{equation}
yielding the full process
\begin{equation}
    \label{eq:full_propagator}
    \mathrm{d}y_t^{(i)} = b_{\mathcal{H}}(t, Y_t)\mathrm{d}t + \Gamma(Y_t)\mathrm{d}W_t^{(i)},\qquad t\geq 0.
\end{equation}
For the intended application, the drift term $f_{\mathcal{H}}$ corresponds to the drift term in \eqref{eq:aldi}, replacing $\Phi$ with $\mathcal{H}(s_k)$, i.e. 
$$
f(\mathcal{H}(s_k),Y_t) = - C(Y_t)\nabla_{y_t^{(i)}} 
\mathcal{H}(s_k)\left( y_t^{(i)}\right) +  \frac{D+1}{B} \left(y^{(i)}_t-\overline{y}_t\right).
$$
The partitioning of the dynamics up to time $\mathcal{T}(K)$ via \eqref{eq:auxillary_propagator} can be seen as a preconditioner with auxiliary potentials prior to starting the ALDI run with the posterior potential $\Phi$ at $t=\mathcal{T}(K)$.


\begin{assumption}[Local convergence]
\label{ass:localconvergence} 
Let $\delta\colon[0,1]\to[\underline{\delta},\infty]$ for $\underline{\delta}>0$ be a parameter dependent convergence radius.
For $s\in[0,1]$, let $\hat{\mu}_{0,s}$ be an arbitrary random probability measure with 
\begin{equation}
    \label{eq:convergenceradius}
    \mathbb{E}[\mathcal{W}_2(\mu_s, \hat{\mu}_{0,s})]<\delta(s).
\end{equation}
Then, the propagator SDE \eqref{eq:auxillary_propagator} implies expected convergence to $\mu_s$ in the following sense:

For $0<\epsilon_s < \delta(s)$ there exists $T_s := T_s(\epsilon_s)>0$ and $B_s=B_s(\epsilon_s)>0$ such that for all $B\geq B_s$,
\begin{equation}
    \mathbb{E}\left[\mathcal{W}_2(\hat{\mu}_{T_s,s}^{B}, \mu_s)\right] \leq \epsilon_s.
\end{equation}
Here, $\hat{\mu}_{t,s}^{B}$ denotes the random ensemble distribution associated to $(y_t^{(i)})_{i=1}^B$ as the solution of \eqref{eq:auxillary_propagator} with $y_0^{(i)}\sim \hat{\mu}_{0,s}$.
\end{assumption}

Note that Assumption \ref{ass:localconvergence} refers to local convergence in terms of the starting distribution only.

\begin{theorem}[Convergence of homotopy approach]
\label{thm:conv_homotopy}
Let $\mathcal{H}$ be a $\mathcal{W}_2$ stable homotopy between $\mathcal{H}(0) = \Psi\in L_{\text{log}}^1\left(\mathbb{R}^{D}\right)$ and the posterior potential $\mathcal{H}(1) = \Phi$ and let Assumption \ref{ass:localconvergence} be satisfied.
Let $\epsilon >0$.
Then there exists $K\in\mathbb{N}$, switch points $s_0,\ldots,s_K$ and a corresponding horizon partition $\mathcal{T}$, $T_\epsilon>0$ and $B_\epsilon>0$ such that 
$$
\mathbb{E}\left[\mathcal{W}_2(\hat{\mu}_{T_\epsilon}^{B_\epsilon}, \mu_\ast)\right] < \epsilon,
$$
where $\hat{\mu}_t^B$ denotes the ensemble distribution of $(y_t^{(i)})_{i=1}^B$.
\end{theorem}
\begin{proof}
For $B\in\mathbb{N}$ let $\hat{\mu}^B_{t,s}$ denote the $B$-particle ensemble distribution at time $t$ using the propagation \eqref{eq:general_SDE}
with 
$$-A(Y_t)\nabla_{y_t^{(i)}}\mathcal{V}(Y_t) = -C(Y_t)\nabla_{y_t^{(i)}}\mathcal{H}(s)\left(y_t^{(i)}\right) + \frac{D+1}{B}\left(y_t^{(i)}-\overline{y}_t\right)$$
with initial condition $\hat{\mu}_{0,s}^B$. Since $\mathcal{H}$ is $\mathcal{W}_2-$stable, it follows from the stability assumption in Definition \ref{def:stable_homotopy} that for all $s\in[0,1]$
\begin{equation*}
    \mathcal{W}_{2}(\mu_{0}, \mu_s) \leq \phi(s).
\end{equation*}
Since $\phi$ is continuous and $\delta(s)$ is uniformly bounded from below by $\underline{\delta}$, for arbitrary but fixed $0<\epsilon_0<\underline{\delta}$ there is a maximal $s_1\in(0,1]$ such that
$$
 \phi(s_1) + \epsilon_0 < \delta(s_1).
$$
Now let $\hat{\mu}_{0,0}$ be an initial ensemble distribution with $\mathbb{E}[\mathcal{W}_2(\mu_0, \hat{\mu}_{0,0})] < \delta(0)$. 
Then by the triangle inequality and using Assumption \ref{ass:localconvergence} we find $T_0=T_0(\epsilon_0)$, $B_0=B_0(\epsilon_0)$ such that
\begin{equation*}
\mathbb{E}\left[ \mathcal{W}_2(\mu_{s_1}, \hat{\mu}_{T_0,0}^{B_0}) \right]
\leq 
\mathcal{W}_2(\mu_{s_1},\mu_{s_0}) +
\mathbb{E}\left[ \mathcal{W}_2(\mu_{s_0}, \hat{\mu}_{T_0,s_0}^{B_0}) \right] 
\leq 
\phi(s_1) + \epsilon_0
\leq \delta(s_1)
\end{equation*}
with initial ensemble distribution $\hat{\mu}_{0,0}^{B_0}$ with particles drawn from $\hat{\mu}_{0,0}$.
Hence, $\hat{\mu}_{T_0,0}^{B_0}$ is a random measure that is within the convergence radius of the expected $2$-Wasserstein distance to $\mu_{s_1}$ and the setting of Assumption \ref{ass:localconvergence} holds again. We then define the first part of the horizon partition as $$\mathcal{T}(1) = T_0.$$

With similar arguments for arbitrary but fixed $\epsilon_k<\underline{\delta}$, we define maximum $s_{k+1}$ with $s_{k}\leq s_{k+1} \leq 1$ inductively for $k = 1,\ldots$, satisfying
$$
  \phi(s_{k+1}-s_{k}) + \epsilon_k < \delta(s_{k+1}).
$$
Then, the for the chosen $\epsilon_k$ we again find a time horizon $T_{k}$ and sample size $B_{k}$ such that 
\begin{equation*}
\mathbb{E}\left[ \mathcal{W}_2(\mu_{s_{k+1}}, \hat{\mu}_{T_{k},s_k}^{B_{k}}) \right]
\leq \delta(s_{k+1}).
\end{equation*}
This yields the time horizon update
$$
 \mathcal{T}(k+1) := \mathcal{T}(k) + T_{k}.
$$
Since $\underline{\delta}>0$ and $\phi$ is continuous, this procedure stops in finite time.
Concretely, there exists $K\in\mathbb{N}$ with $s_K = 1$.
\end{proof}

\begin{remark}
Provided that Theorem \ref{thm:B-T-convergence_ALDI} holds true for $\mathcal{H}(s)$ instead of $\phi$ and $\exp(-\mathcal{H}(s))$ is proportional to a Gaussian for some $s\in[0,1]$, we can then set $\delta(s)=\infty$.
This becomes especially relevant when choosing $\Psi=\mathcal{H}(0)$ such that $\exp(-\Psi)$ is proportional to a Gaussian approximation of $\mu_\ast$ in the non-linear setup.
\end{remark}

\begin{remark}
\label{remark:homotopy_design}
The design of the piecewise constant drift term is motivated by the numerical realisation of the homotopy approach. 
In a more general framework, one may want to define an inhomogenous drift term $f(t,Y)$ using some proper time scaling between the SDE time $t$ and the homotopy switch design $s$ of the form $s = s(t)$.
This would lead to a particle system of the form
\begin{equation*}
    \mathrm{d}y_t^{(i)} = f(\mathcal{H}(s(t)), Y_t)\mathrm{d}t + \Gamma(Y_t)\mathrm{d}W_t^{(i)}.
\end{equation*}
\end{remark}
The analysis of the requirements on general $s(t)$ as in remark \ref{remark:homotopy_design} are out of the scope of this work, but we discuss several designs in the following numerics section.

%% file: illustration_lidl.tex
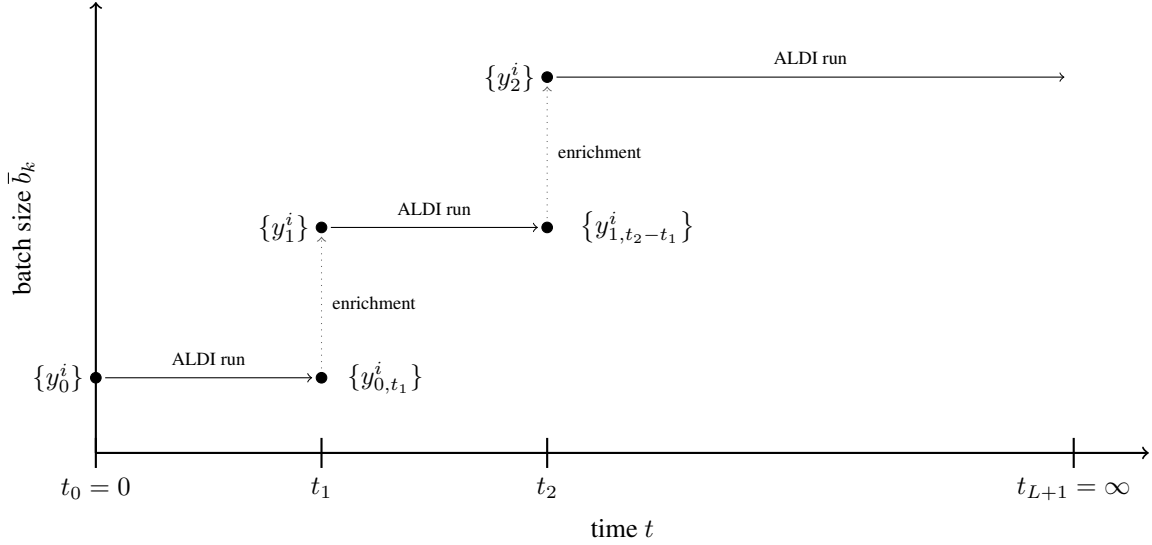
\begin{figure}[ht]
\begin{center}
\begin{tikzpicture}

\node at (0, 1)   (a1) {};
\node at (3,1) (a2) {};
\node at (3,3) (b1) {};
\node at (6,3) (b2) {};
\node at (6,5) (c1) {};
\node at (13,5) (c2) {};

\filldraw (a1) circle (2pt);
\filldraw (a2) circle (2pt);
\filldraw (b1) circle (2.pt);
\filldraw (b2) circle (2.pt);
\filldraw (c1) circle (2.pt);

\draw[->]   (a1) -- (a2);
\node at (1.5,1+0.25) {\scriptsize ALDI run};
\draw[->,dotted, opacity = 0.6]   (a2) -- (b1);
\ifnum\classstyle=0 
\node at (3.7,2) {\scriptsize enrichment};
\fi
\ifnum\classstyle=1 
\node at (3.7,2) {\scriptsize enrichment};
\fi
\ifnum\classstyle=2 
\node at (3.9,2) {\scriptsize enrichment};
\fi
\draw[->]   (b1) -- (b2);
\node at (4.5,3+0.25) {\scriptsize ALDI run};
\draw[->,dotted, opacity = 0.6]   (b2) -- (c1);
\ifnum\classstyle=0 
\node at (6.7,4) {\scriptsize enrichment};
\fi
\ifnum\classstyle=1 
\node at (6.7,4) {\scriptsize enrichment};
\fi
\ifnum\classstyle=2 
\node at (6.9,4) {\scriptsize enrichment};
\fi
\draw[->]   (c1) -- (c2);
\node at (9.5,5+0.25) {\scriptsize ALDI run};

\node at (-0.5,1) {$\{y_0^i\}$};
\node at (3+0.85,1) {$\{y_{0,t_1}^i\}$};
\node at (3-0.5,3) {$\{y_1^i\}$};
\ifnum\classstyle=0 
\node at (3-0.5+4.7+0.07,3) {$\left\{y_{1,t_2-t_1}^i\right\}$};
\fi
\ifnum\classstyle=1 
\node at (3-0.5+4.7,3) {$\left\{y_{1,t_2-t_1}^i\right\}$};
\fi
\ifnum\classstyle=2 
\node at (3-0.5+4.7+0.02,3) {$\left\{y_{1,t_2-t_1}^i\right\}$};
\fi
\node at (3-0.5+3,3+2) {$\{y_2^i\}$};

\draw[thick,->] (0,0) -- (14,0);
\draw[thick] (0,0.2) -- ++ (0,-0.4) node[below] {$t_0=0$};
\draw[thick] (3,0.2) -- ++ (0,-0.4) node[below] {$t_1$};
\draw[thick] (6,0.2) -- ++ (0,-0.4) node[below] {$t_2$};
\draw[thick] (13,0.2) -- ++ (0,-0.4) node[below] {$t_{L+1}=\infty$};
\draw[thick,->] (0,0) -- (0,6);
\node at (7,-1) {time $t$};
\node [rotate=90] at (-1,3) {batch size $\overline{b}_k$};


\end{tikzpicture}
\end{center}
\captionsetup{width=.92\linewidth}
\caption{
\textit{Schematic of a run of LIDL with $L=2$ enrichment steps. An initial batch of samples $\{y_0^i\}$ is propagated through the ALDI dynamics \eqref{eq:aldi} for $t\in[0,t_1]$. The resulting samples $\{y_{0,t_1}^i\}$ are enriched via some enrichment strategy $\mathcal{E}$ to receive a larger sample batch $\{y_1^i\}$, which is again propagated and then enriched at $t=t_2$. Finally, a full run of ALDI is performed on the final batch $\{y_2^i\}$.
\label{fig:scheme_lidl}
}}
\end{figure}

%% file: 05_numerics.tex
\section{Numerical examples}
\label{sec:numerics}

This section is devoted to the numerical investigation of the ideas presented in this manuscript.
Since our error analysis is carried out in the $2$-Wasserstein metric, numerical errors are discussed to some extend in this error discrepancy.
In particular, since we are concerned with the setup of ensembles, only discrete measures are examined.
It is well-known that the computation of the $2$-Wasserstein distance of two discrete measures is equivalent to a constrained assignment problem.
Its computation can easily become cumbersome with a growing number of samples. Consequently, we perform the numerical error analysis in terms of an approximation of the Wasserstein error, namely the debiased Sinkhorn metric $\mathcal{S}_\epsilon$ introduced in Section \ref{sec:generalized_transport}.
For the approximation error introduced by the Sinkhorn metric in terms of the regularizing parameter $\epsilon>0$ we refer to \cite{genevay2019sample}.
Throughout this section, we choose $\epsilon=0.1$.
The computation of the debiased Sinkhorn metric is realised with the python package
\texttt{GeomLoss} by Jean Feydy~\cite{feydy2019interpolating}.

\vspace{2ex}
\paragraph{Random error variables}
Since we want to track the convergence in the case of finite batch size $\overline{b}$ as reliably as possible, we define the following random error variables
\begin{align}
    \label{eq:EP_t}
    &&\mathrm{EP_t}\colon \omega &\mapsto \mathcal{S}_\epsilon(\hat{\mu}_t,\mu^{(\overline{b})}_{\ast}), &(t\textit{-ensemble-posterior}) &\\
    \label{eq:PP}
    &&\mathrm{PP}\colon\omega &\mapsto \mathcal{S}_{\epsilon}(\mu^{(\overline{b})}_{\ast},\tilde{\mu}^{(\overline{b})}_{\ast}).& (\textit{posterior-posterior})&
\end{align}
Here, the ensemble distribution $\hat{\mu}_t$ is defined by \eqref{eq:lidl_measure} and 
$\mu^{(\overline{b})}_{\ast}, \tilde{\mu}^{(\overline{b})}_{\ast}$ are independent empirical measures of $\overline{b}$ samples drawn from the true posterior distribution, respectively.
As in the previous sections, $\overline{b}$ corresponds to the total number of posterior samples that are produced with our method.
For the sake of readability, we suppress the dependence of $\mathrm{EP}_t$ on the particular used batch sizes for $\hat{\mu}_t$ and the total batch size $\overline{b}$ as this becomes clear from the context.
Furthermore, since we are dealing with particular instances of ensemble trajectories, the corresponding ensemble distributions are random measures and thus it is natural to investigate the expected Wasserstein (Sinkhorn) error
\begin{equation}
    \begin{aligned}\label{eq:num_metric}
    \mathbb{E}[\mathcal{W}_2(\hat{\mu}_t,\mu^{(\overline{b})}_{\ast})] \approx \mathbb{E}[\mathcal{S}_\epsilon(\hat{\mu}_t,\mu^{(\overline{b})}_{\ast})] = \mathbb{E}[\mathrm{EP}_t].
    \end{aligned}
\end{equation}

Under the assumption that \eqref{eq:ev_condition} is satisfied for sufficiently large $T$, this expectation can be controlled for suitable configurations of ALDI/LIDL since by Theorem \ref{thm:conv_emp_measure} and Corollary \ref{cor:B-T-convergence_consistent} for any $\delta > 0$ there exist a time $T_{\delta}$ and a batch size $\overline{b}_{T_\delta}$ such that
\begin{equation}\label{eq:wasserstein_bound}
    \mathbb{E}[\mathcal{W}_2(\hat{\mu}_{t},\mu^{(\overline{b}_{T_{\delta}})}_{\ast})] \leq \mathbb{E}[\mathcal{W}_2(\hat{\mu}_t,\mu_{\ast})] + \mathbb{E}[\mathcal{W}_2(\mu_{\ast},\mu^{(\overline{b}_{T_{\delta}})}_{\ast})] \leq \delta.
\end{equation}
However, for fixed batch sizes $\overline{b}$, \eqref{eq:num_metric} cannot be expected to be close to $0$ for any $t>0$ unless $\overline{b}$ tends to infinity.
In order to still track some type of convergence based on \eqref{eq:num_metric} for a finite number of particles, we consider the so-called expected \textit{posterior-posterior} error for finite batch size given by
 \begin{equation}\label{eq:num_benchmark}
    \mathbb{E}[\mathcal{W}_2(\mu^{(\overline{b})}_{\ast},\tilde{\mu}^{(\overline{b})})] \approx \mathbb{E}[\mathcal{S}_{\epsilon}(\mu^{(\overline{b})}_{\ast},\tilde{\mu}^{(\overline{b})}_{\ast})] = \mathbb{E}[\mathrm{PP}].
\end{equation}
The random variable $\mathrm{PP}$ and its expectation $\mathbb{E}[\mathrm{PP}]$ are constant in $t$.
Moreover, since $\mathrm{PP}$ is defined upon different realisations of finite posterior samples it is a non-negative random variable. 
However, its expectation $\mathbb{E}[\mathrm{PP}]$ goes to 0 in the limit $\overline{b}\rightarrow \infty$.
In the numerical error analysis we then examine the convergence 
\begin{align}
    \mathrm{EP}_t &\stackrel{d}{\longrightarrow} \mathrm{PP},\label{eq:num_conv_dist}\\
    \mathbb{E}[\mathrm{EP}_t] 
    &\longrightarrow \mathbb{E}[\mathrm{PP}],\label{eq:num_conv_ep}
\end{align}
as $t\to \infty$.
We remark that the investigation of the expected $2$-Wasserstein error is also motivated by the so-called \textit{concentration around expectation} property.
In particular, let $\mu\in\mathcal{W}_2(\mathbb{R}^{D})$ and $\mu^{(B)}$ be an empirical measure determined by $B$ i.i.d. samples drawn from $\mu$.
Then, McDiarmid's inequality yields~\cite{weed2019sharp}
$$
\mathbb{P}\left( \mathcal{W}_2^2(\mu,\mu^{(B)}) \geq 
\mathbb{E}[\mathcal{W}_2^2(\mu,\mu^{(B)})] + t\right) \leq \exp\left(-2B t^2\right).
$$

\vspace{1ex}
\paragraph{Double Sinkhorn error}
As discussed above, the aforementioned quantities do not converge to $0$ for a finite batch size.
The random variables $\mathrm{EP}_t$ and $\mathrm{PP}$ are non-negative scalar-valued.
Due to the ergodic property of the ALDI propagation~\cite{garbuno2020affine}, we expect that $\mathrm{EP}_t$ converges to $\mathrm{PP}$ in distribution \eqref{eq:num_conv_dist}.
Let $\mu_{\mathrm{EP}_t}$ and $\mu_{\mathrm{PP}}$ be the distribution of $\mathrm{EP}_t$ and $\mathrm{PP}$, respectively.
Then, we define the so-called \textit{double Sinkhorn error} as the mapping
\begin{equation}
    \label{eq:cont_doublesinkhorn}
    t\to \mathcal{S}_{\epsilon}( \mu_{\mathrm{EP}_t} , \mu_{\mathrm{PP}} ).
\end{equation}
Note that $\mathcal{S}_{\epsilon}( \mu_{\mathrm{EP}_t} , \mu_{\mathrm{PP}} ) \to 0$ implies \eqref{eq:num_conv_dist}.
The numerical realisation of \eqref{eq:cont_doublesinkhorn} is examined as follows.
For a \textit{number of runs} $R\in\mathbb{N}$, let $(\hat{\mu}^{r}_t)_{t}$, $r=1,\ldots,R$, be the family of measures generated by the $r$-th ALDI/LIDL run.
Moreover, let
$$
\mu_{\ast}^{r,\overline{b}}, \tilde{\mu}_{\ast}^{r,\overline{b}}, \tilde{\tilde{\mu}}_{\ast}^{r,\overline{b}}
$$
be independent random empirical measures of $\overline{b}$ particles sampled from the posterior for every $r$.
Then, the sets
\begin{equation}
    \label{eq:double_sinkhorn_estimate}
    \{\mathcal{S}_{\epsilon}(\hat{\mu}^{r}_t,\mu_{\ast}^{r,\overline{b}})\}_{r=1}^R,\quad
\{\mathcal{S}_{\epsilon}(\tilde{\mu}_{\ast}^{r,\overline{b}}, \tilde{\tilde{\mu}}_{\ast}^{r,\overline{b}})\}_{r=1}^R
\end{equation}
are interpreted in this instance as an empirical measure approximating $\mu_{\mathrm{EP}_t}$ and $\mu_{\mathrm{PP}}$, respectively.
Finally, these empirical measures are used to approximate \eqref{eq:cont_doublesinkhorn} with error converging to zero as $t,r\to \infty$.

\vspace{1ex}
\paragraph{Expected trajectory difference and slope}

To adaptively choose the enrichment times 
$t_{\ell}$, we are interested in tracking the convergence speed of $\mathbb{E}[\mathrm{EP}_t]$ to $\mathbb{E}[\mathrm{PP}]$.
In the case of convergence, for any $k\in\mathbb{N}$, we have
\begin{equation}\label{eq:diff}
    \operatorname{DIFF}(t)\coloneqq 
\mathbb{E}[\mathrm{EP}_t]-\mathbb{E}[\mathrm{EP}_{t- k\Delta t}] \longrightarrow 0,
\end{equation}
which we will use for a first heuristic for adaptivity. Towards a second heuristic,
suppose that we have an idealized setting with a smooth and strictly monotonically decreasing $\mathbb{E}[\mathrm{EP}_t]$.
Then, a suitable measure for the convergence speed would be the \textit{expected trajectory slope}
\begin{equation}
\label{eq:slope}
    \operatorname{SLOPE}(t) \coloneqq \dfrac{\mathbb{E}[\mathrm{EP}_{t}] - \mathbb{E}[\mathrm{EP}_{t-\Delta t}]}{\Delta t} = \dfrac{\mathrm{d}}{\mathrm{d}t}\mathbb{E}[\mathrm{EP}_t]  + \mathcal{O}(\Delta t), \quad t\geq \Delta_t,
\end{equation}
with time step $\Delta t > 0$.
However, $\mathbb{E}[\mathrm{EP}_t]$ and $\mathbb{E}[\mathrm{EP}_{t-k\Delta t}]$ depend on the unknown posterior, which we cannot access during computation.
Furthermore, a solution of ALDI/LIDL produces only a \textit{single} realization of a family of ensemble distributions $(\hat{\mu})_t$ and hence does not provide access to the expectation.
Therefore, we aim to approximate the above quantities based on information available from such a single realization $(\hat{\mu}_t)_t$.
For this, we devise two simple heuristic approaches.
Let $t_a\geq 0$ be a possible time point for ensemble enrichment. 
Moreover, let $N\in\mathbb{N}$ be a history depth with $t_a - N\Delta t \geq 0$ and $0 \leq t_1,\ldots,t_N=t_a$ with $t_{i} = t_{i+1}-\Delta t$ for $i=1\ldots,N-1$.
The heuristics based on the difference \eqref{eq:diff} and the slope \eqref{eq:slope} respectively are given as follows.

\begin{itemize}
    \item Let $N_1,N_2\in\mathbb{N}$ be local averaging parameters with $N=N_1+N_2$.
    We then define the \textit{difference heuristic} by
    \begin{equation}\label{eq:num_diff_approx}
        \widehat{\operatorname{DIFF}}(t_a)
        =\left|
        \frac{1}{N_1}\sum\limits_{i=1}^{N_1}\mathcal{S}_\epsilon(\hat{\mu}_{t_i},\hat{\mu}_{t_a})
        -
         \frac{1}{N_2}\sum\limits_{i=N_1+1}^{N}\mathcal{S}_\epsilon(\hat{\mu}_{t_i},\hat{\mu}_{t_a})\right|.
    \end{equation}
    Given some threshold $0 < \mathrm{tol} < 1$, we encourage an ensemble enrichment at time $t_a$ if $\widehat{\operatorname{DIFF}}(t_a) < \mathrm{tol} * \widehat{\operatorname{DIFF}}(t_\mathrm{last})$, where $t_\mathrm{last}$ is the time point of the last enrichment, if at least one enrichment has already taken place.
    If no enrichment has occurred yet, we choose some small fixed positive value $ \mathrm{ref}>0 $ and set $\widehat{\operatorname{DIFF}}(t_\mathrm{last}) \equiv \mathrm{ref}$.
    
    \item For the \textit{slope heuristic}, let
        \begin{equation}
        \begin{aligned}
        \label{eq:slope_approx_average}
    \widehat{\operatorname{SLOPE}}(t_a) &\coloneqq \dfrac{1}{N-2}\sum_{i=2}^{N-1} \dfrac{\mathcal{S}_\epsilon(\hat{\mu}_{t_i},\hat{\mu}_{t_a}) - \mathcal{S}_\epsilon(\hat{\mu}_{t_{i-1}},\hat{\mu}_{t_a})}{\Delta t} \\
    &= \dfrac{1}{(N-2)\Delta t}[ \mathcal{S}_\epsilon(\hat{\mu}_{t_{N-1}},\hat{\mu}_{t_a}) - \mathcal{S}_\epsilon(\hat{\mu}_{t_{1}},\hat{\mu}_{t_a})].
\end{aligned}
\end{equation}
Given some threshold $0 < \mathrm{tol} < 1$, we encourage ensemble enrichment at time $t_a$ if $ \widehat{\operatorname{SLOPE}}(t_a) > -\mathrm{tol} * \left| \widehat{\operatorname{SLOPE}}(t_{\mathrm{last}})\right|$, where $t_{\mathrm{last}}$ is defined analogously to the difference heuristic.
\end{itemize}

This choice of heuristics may additionally be motivated as follows.
First, instead of the expectations in \eqref{eq:diff} and \eqref{eq:slope} we use history based averaging to alleviate stochastic fluctuation.
Second, we substitute $\hat{\mu}_{t_a}$ for $\mu_{\ast}^{(\overline{b})}$, since, in the ideal setup of ergodicity of ALDI with $b$ particles, and for $t_a\to\infty$, the measure $\hat{\mu}_{t_a}$ is an instance of an empirical measure $\mu_{\ast}^{(b)}$ defined through $b$ i.i.d. samples from the posterior $\mu_\ast$.

\vspace{1ex}
\paragraph{Error with respect to forward calls}
To carry out the numerical investigation and numerical error analysis, we make use of the error development with respect to both the time $t$ and the number of \textit{forward calls}.
Let $\Delta t$ denote the uniform time discretisation step and let $T=n_{\mathrm{iter}}\Delta t<\infty$ for some $n_{\mathrm{iter}}\in\mathbb{N}$ denote the maximum time point for which we compute the solution of the underlying particle system.
Moreover, let an instance of ALDI be realised with $\overline{b}\in\mathbb{N}$ particles and of LIDL with enrichment sample sizes $(b_0,\ldots, b_L)\in\mathbb{N}_0^{L+1}$, where $L\in\mathbb{N}_0$ denotes the number of enrichment stages, such that $\overline{b}=\sum_{\ell=0}^L b_\ell$.
Furthermore, for $L>0$ let
\begin{equation*}
    t_\ell = k_\ell \Delta t,
\end{equation*}
with $0=t_0 < t_1 < \ldots < t_L < T$ for a suitable number of local iterations $k_\ell\in\mathbb{N}$ denote the time points at which the ensemble enrichment with $b_\ell$ additional particles takes place.
Then, for $k_0=k_{-1}\coloneqq 0$ we define the number of forward calls for an instance of ALDI or LIDL for $t\leq T$ as follows:
\begin{align}
    \label{eq:FC_aldi}
    \mathrm{FC}_{\mathrm{ALDI}}(t) &= 
    \overline{b}_{\phantom{\ell}}\cdot \argmax\limits_{k\in\mathbb{N}}\{k\Delta t \leq t\},\\  
     \label{eq:FC_lidl}
    \mathrm{FC}_{\mathrm{LIDL}}(t) &= 
    \overline{b}_{\ell} \cdot \argmax\limits_{k\in\mathbb{N}}\{k\Delta t \leq t-t_\ell : t_\ell \leq t \leq t_{\ell+1}\} +
    \sum\limits_{\ell\,:\,t_\ell < t}
    \overline{b}_{\ell-1}\cdot (k_{\ell} - k_{\ell-1}),
    %
      %
\end{align}
where $\overline{b}_{\ell} \coloneqq  \sum_{m=0}^{\ell} b_m$ for $\ell = 0,\ldots,L$ denotes the partial sum of batch sizes.
When using the homotopy techniques, we implicitly make use of some function $s(t)$ relating the homotopy time scale $s$ to the time scale $t$ of the particle system.
As discussed in Section \ref{sec:homotopy}, this $s(t)$ is chosen to be piecewise constant in $t$, corresponding to the piecewise constant definition of the potentials from \eqref{eq:auxillary_propagator}.
When required, we may interpret $t\mapsto s(t)$ in terms of the number of function calls, formally defining
\begin{align}
  \label{eq:switch_aldi_fc}
    s_\mathrm{ALDI}(\mathrm{fc}) &= s\left( \mathrm{FC}_{\mathrm{ALDI}}^{-1}(\mathrm{fc})\right),
    \\
    \label{eq:switch_lidl_fc}
        s_\mathrm{LIDL}(\mathrm{fc}) &= s\left( \mathrm{FC}_{\mathrm{LIDL}}^{-1}(\mathrm{fc})\right),
    \end{align}
for a number of function calls $\mathrm{fc}\in\mathbb{N}$. %

\vspace{1ex}
\paragraph{Outline}
In the following, several problems are investigated and discussed.
First, some validation experiments for Gaussian posteriors are performed in Sections \ref{section:num_unimodal_1} and \ref{sec:num_adaptive}.
Second, the application and efficacy of the homotopy approach in the case of multimodal posteriors is examined in Section \ref{sec:num_multi_homotopy}.
Third, we investigate a high-dimensional linear second order PDE problem in Section \ref{sec:num_darcy}. In most examples, we use diffusion propagation \eqref{eq:diffusion_propagation_samples} as our enrichment scheme of choice, as we did not discover substantial differences in performance to other methods.


\subsection{Gaussian mixtures: Unimodal and multimodal posterior}
\label{sec:num:uni_multi}

For a number of mixtures $K\in \{1,4\}$, define the potential 
\begin{equation}\label{eq:num_potential}
    \Phi : \mathbb{R}^2 \longrightarrow \mathbb{R}, \qquad \Phi(x) = -\ln\left( \frac{1}{2\pi K\sqrt{ |\Sigma|}}\sum_{i=1}^K \exp\left( -\frac{1}{2}\| x-x_i \|^2_{\Sigma} \right) \right),
\end{equation}
with $\Sigma=I_2$ being the identity matrix in $\mathbb{R}^2$ and $x_i = (\cos(i\frac{\pi}{2}),\sin(i\frac{\pi}{2}))$ for $i=1,\hdots, 4$.
Then, the posterior density $\pi_\ast$ is given by $\pi_\ast = \exp(-\Phi)$.
As initial distribution for $\{y_0^{i}\}$ in the particle system, we choose a Gaussian with density
\begin{equation}\label{eq:num_initial}
    \pi_0(x) = \frac{1}{2\pi \sqrt{ |\Sigma|}} \exp\left( -\frac{1}{2}\| x-x_3 \|^2_{\Sigma} \right).
\end{equation}
This initial density $\pi_0$ and the posterior $\pi_\ast$ corresponding to $K=4$ are depicted in Figure~\ref{fig:prior_and_posterior}.

\begin{figure}
    \centering
    \includegraphics[width=\linewidth]{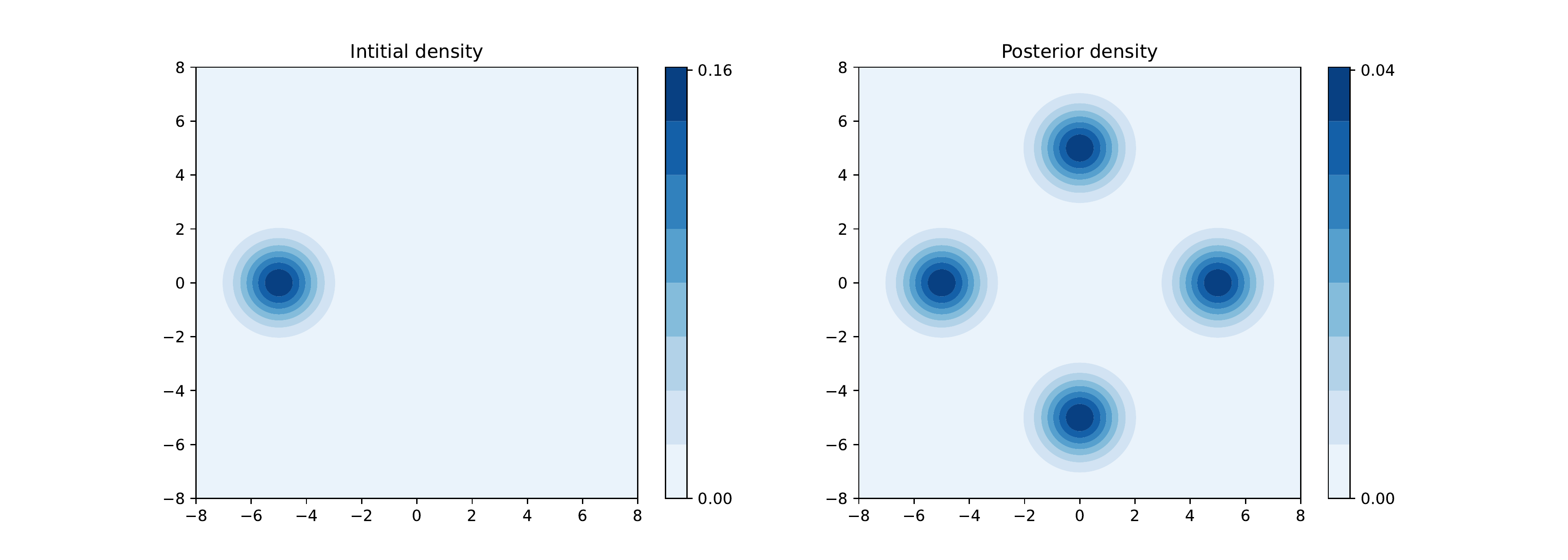}
    \caption{\textit{Initial distribution \eqref{eq:num_initial} and posterior density $\exp(-\Phi)$ with $\Phi$ from \eqref{eq:num_potential} for $K=4$.}}
    \label{fig:prior_and_posterior}
\end{figure}

Hence, for $K=1$ the task of a given particle propagator consists in translating the initial density from $(-5,0)$ to $(5,0)$.
Moreover, for $K=4$ the initial condition is close to a local minimum of the potential $\Phi$.

\subsubsection{The unimodal case: proof of concept}
\label{section:num_unimodal_1}

First, we compare the performance of EKS \cite{garbuno2020interacting}, ALDI \cite{garbuno2020affine} and LIDL on the translation problem, i.e. $K=1$.
The dynamics \eqref{eq:aldi} are discretized using the Euler-Maruyama method with step size $\Delta t = 0.05$ for a finite time horizon with $T=10$ leading to $200$ time steps.
We choose a total batch size of $\overline{b}=400$ for both methods and $L=3$ enrichment stages with $b_0=\ldots=b_3 = 100$ for LIDL.
Ensemble enrichment is done via diffusion steps \eqref{eq:diffusion_step} at fixed time steps $(t_1,t_2,t_3)=(1,2,3)$.
In between ensemble enrichment stages, LIDL uses ALDI as a particle propagator.

As shown in Figure~\ref{fig:iters_fcs}, all methods achieve similar convergence speed with regards to required time steps, but LIDL converges faster in terms of forward calls using this simple enrichment strategy.
Note that ALDI and EKS perform similarly on this problem due to the large batch size $\overline{b}=400$.
The smaller $\overline{b}$ gets the more we expect ALDI to outperform the EKS since the correction term entering \eqref{eq:aldi} is inversely proportional to the batch size (this is demonstrated for the Darcy problem in Section \ref{sec:num_darcy}).
ALDI's ergodicity even for small batch sizes $b > D + 1$ is crucial for LIDL's ensemble enrichment strategy, which is why we use ALDI instead of the EKS as a particle mover.
Due to its superior performance, we mainly use ALDI as a benchmark henceforth.

\begin{figure}
        \centering
        \includegraphics[width = \linewidth]{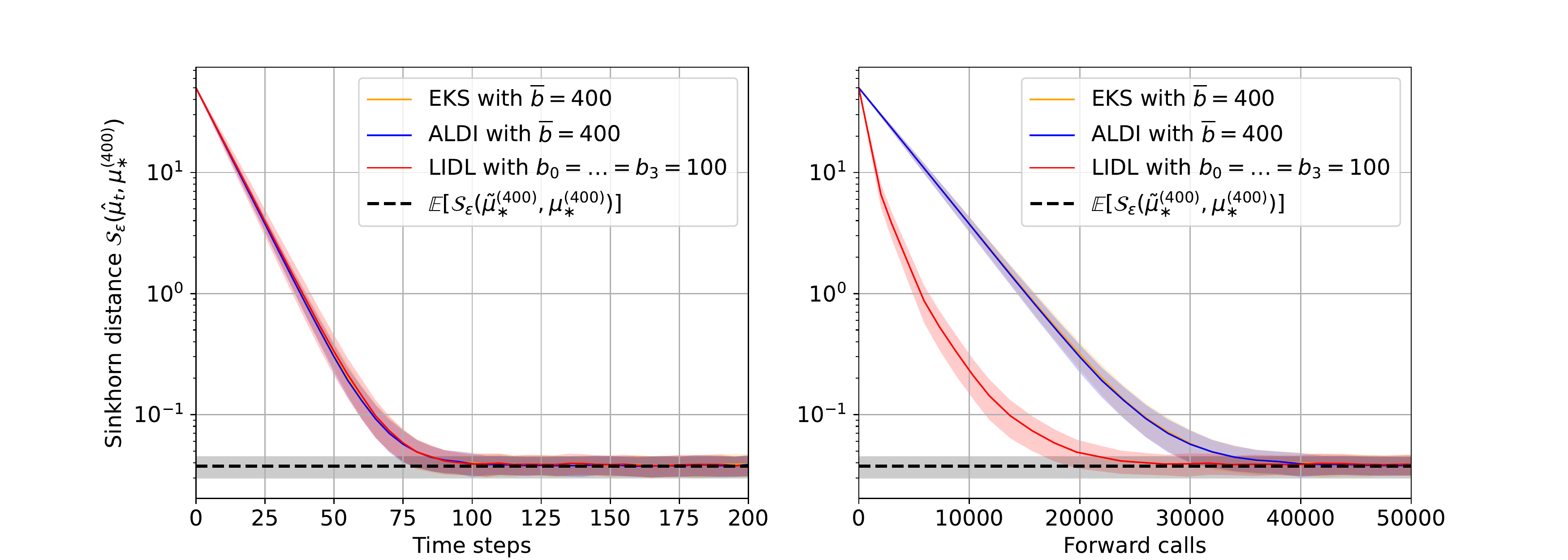}
        \caption{ \textit{
        Convergence history of the $\mathrm{EKS}$, $\mathrm{ALDI}$ and $\mathrm{LIDL}$ for the case of translation problem with potential $\Phi$ from \eqref{eq:num_potential} with $K=1$.
        The shaded areas and the bold interior lines represent standard deviation and mean of the error random variables $\mathrm{PP}$ ({\color{black}black}) from \eqref{eq:PP} and $\mathrm{EP}_t$ from \eqref{eq:EP_t} generated by the $\mathrm{EKS}$ (\textcolor{orange}{orange}), $\mathrm{ALDI}$ (\textcolor{blue}{blue}) and $\mathrm{LIDL}$ (\textcolor{red}{red}). 
        While LIDL leads to similar convergence speed with respect to the time scale (left) it outperforms the classical particle systems in terms of required forward calls (right).
       }
        }
     \label{fig:iters_fcs}
    \end{figure}

\subsubsection{Motivating adaptive Enrichment}
\label{sec:num_adaptive}

A simple equidistant enrichment scheme as in Section \ref{section:num_unimodal_1} is by no means optimal.
In particular, the LIDL method used in Figure~\ref{fig:iters_fcs} is configured to trace ALDI in terms of convergence over the time steps (left plot in Figure~\ref{fig:iters_fcs}).
However, this is not necessarily desirable as Figure~\ref{fig:plateaus} shows.
Here, two ALDI runs with a total particle number of $\overline{b}=50$ and $\overline{b}=400$ each are compared with a LIDL run with $b_0 = 50, b_1 =350$ and the equidistant enrichment $b_i= 100$ for $i=0,\dots,3$ from Section \ref{section:num_unimodal_1}. 
The ensemble enrichment to obtain $350$ samples from $50$ samples is realised with diffusion propagation \eqref{eq:diffusion_step} using $\delta=k\Delta t$ with $k=1,\ldots, 7$.
The non-equidistant enrichment approach clearly lacks behind the ALDI scheme with $\overline{b}=400$ particles with respect to the time steps (left plot).
But examining the convergence with respect to forward calls (right plot), we see that this behaviour is actually beneficial as long the solver is in the region of decay for that particular particle number $\overline{b}=50$.
Roughly speaking, LIDL follows the trajectory of a $50$-particle ALDI run as long as it has not converged.
Once convergence with that particle number is reached, LIDL jumps (by means of the ensemble enrichment) onto the trajectory corresponding to a 400-particle ALDI run.
As can be seen in the right plot in Figure~\ref{fig:plateaus}, hardly any convergence speed is lost due to the enrichment. 
Here, the gray shaded areas can be seen as the maximum possible accuracy for the specific choice of finite batch size $B = 50,400$ up to the desired amount of particles $\overline{b}$.

In Figure~\ref{fig:double_sink}, all four methods are tracked with respect to the double Sinkhorn error from \eqref{eq:cont_doublesinkhorn} approximated empirically using \eqref{eq:double_sinkhorn_estimate} with $R=210$ independent runs for each of the four instances.
We observe that at a double Sinkhorn error of around $10^{-7}$ the curves start to fluctuate, which is explained by the empirical approximation with $R<\infty$.
When comparing both enrichment strategies for LIDL, we see that the $(50,350)$ approach outperforms the $(100,100,100,100)$ setup in terms of forward calls.
This motivates to find adaptive schemes in terms of time points $t_{\ell}$ and batch sizes $b_{\ell}$ of ensemble enrichment.

\begin{figure}
    \centering
    \includegraphics[width = \linewidth]{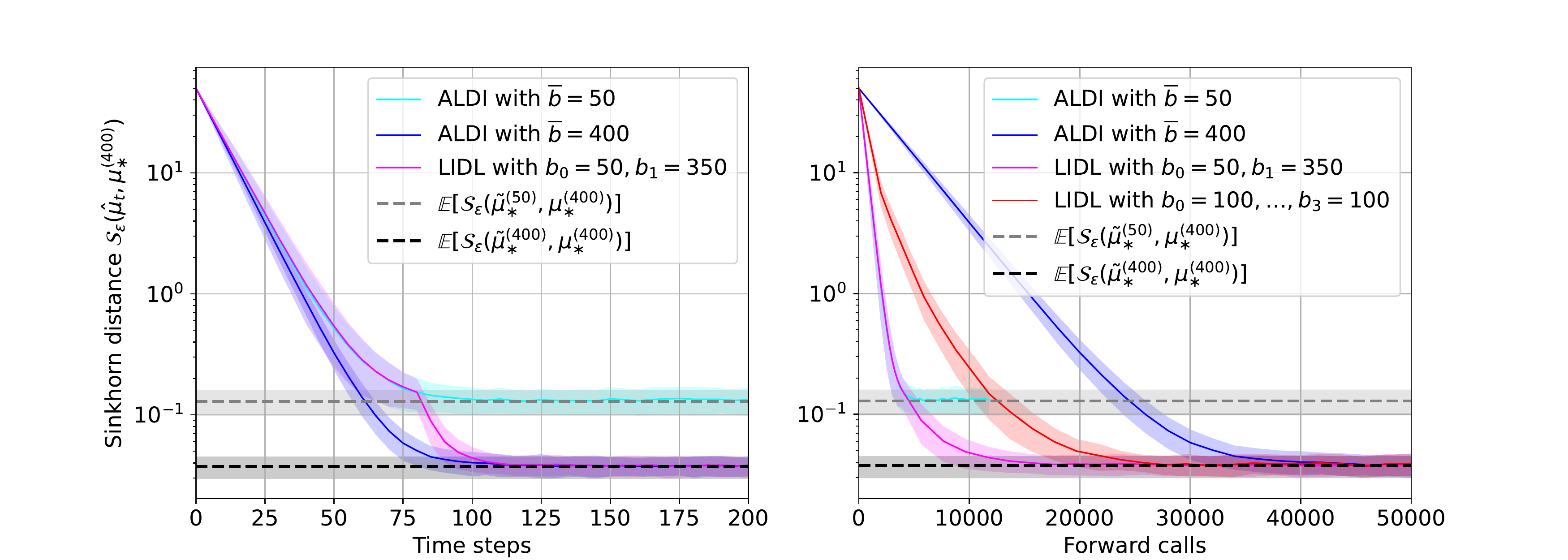}
    \caption{
    \textit{ Performance of $\mathrm{LIDL}$ and $\mathrm{ALDI}$ with different choice of batch sizes and enrichment times.
    Instead of tracing the convergence of the 400-particle $\mathrm{ALDI}$ run (\textcolor{blue}{blue}) over time steps, $\mathrm{LIDL}$ (\textcolor{magenta}{magenta}) follows the 50-particle $\mathrm{ALDI}$ run (\textcolor{cyan}{cyan}) until it is close to its accuracy limit (light gray plateau).
    $\mathrm{LIDL}$ then jumps 
    to the 400-particle trajectory while saving a substantial number of forward calls compared to the full 400-particle $\mathrm{ALDI}$ run.
    For comparison, the suboptimal $\mathrm{LIDL}$ configuration from Figure~\ref{fig:iters_fcs} is added (\textcolor{red}{red}).}}
    \label{fig:plateaus}
\end{figure}

\begin{figure}
        \centering
        \includegraphics[scale=0.5]{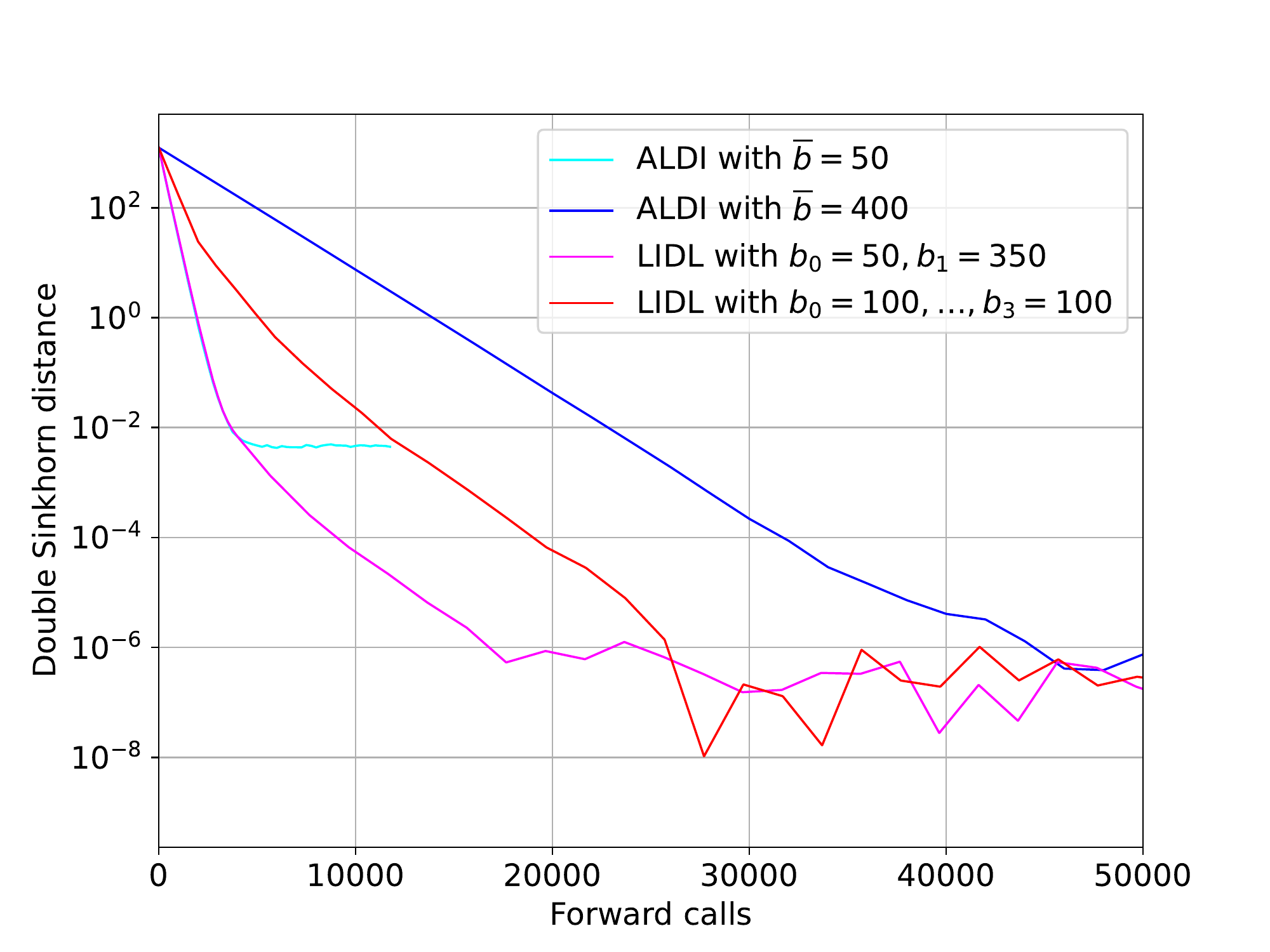}
        \caption{\textit{Double-Sinkhorn distance \eqref{eq:cont_doublesinkhorn} for the four methods from Figure~\ref{fig:plateaus}.
        While the full $\mathrm{ALDI}$ run with $400$ particles (\textcolor{blue}{blue}) has still not converged for $50000$ forward calls, $\mathrm{LIDL}$ (\textcolor{magenta}{magenta}) seems to converge after $\sim 30000$ forward calls with a distance fluctuating stochastically around $10^{-7}$.}}
        \label{fig:double_sink}
    \end{figure}

\paragraph{Fixed batch size enrichment at heuristically determined time points}
While adaptivity in the batch size is a topic for future work, we already present preliminary results for our time point heuristics.
The difference and slope heuristics \eqref{eq:num_diff_approx} and \eqref{eq:slope_approx_average} are designed with the goal to recognize the occurrence of plateaus as seen in the dashed grey line in Figure \ref{fig:iters_fcs}.
In the emergence of such a plateau, \eqref{eq:num_diff_approx} should be positive and close to 0 while \eqref{eq:slope_approx_average} should be negative and close to 0.
We deploy both heuristics under the exact same conditions as in Section \ref{section:num_unimodal_1} for a LIDL run with 3 enrichments and $b_0=\ldots=b_3=100$.
In practice, we check the heuristics only every $5$ time steps to save computation time.
We choose $\mathrm{tol}=0.5$ and $\mathrm{ref}=1$ for both heuristics.
In our case this leads to average enrichment times $(\overline{t}_1,\overline{t}_2,\overline{t}_3)=(2.6,3.65,4.75)$ for the difference heuristic and $(\overline{t}_1,\overline{t}_2,\overline{t}_3)=(3.2,4.3,5.95)$ for the slope heuristic.
Note that the required forward calls \eqref{eq:FC_lidl} at time $t$ now vary from run to run.
For a fixed $t$, we denote the average of \eqref{eq:FC_lidl} over all LIDL runs by $\overline{\mathrm{FC}_{\mathrm{LIDL}}}(t)$.
For an infinite number of runs $R\to\infty$ this becomes the expected number of forward calls at time $t$ with a given heuristic.
In Figure \ref{fig:adaptive_schemes}, we consider this average when plotting both $\mathrm{EP}_t$ and the double Sinkhorn distance with respect to the number of forward calls.
It can be seen that for this problem both heuristics perform better than the naive LIDL configuration with fixed equidistant enrichment stages.
In particular, both heuristics deploy ensemble enrichment before a suboptimal convergence plateau is reached.
At the same time, the methods utilize smaller batch sizes long enough to benefit from the fast convergence speed.
While the unimodal Gaussian is only a very simple test case, these promising results encourage future work on adaptive heuristics.

\begin{figure}
    \centering
    \includegraphics[width=\linewidth]{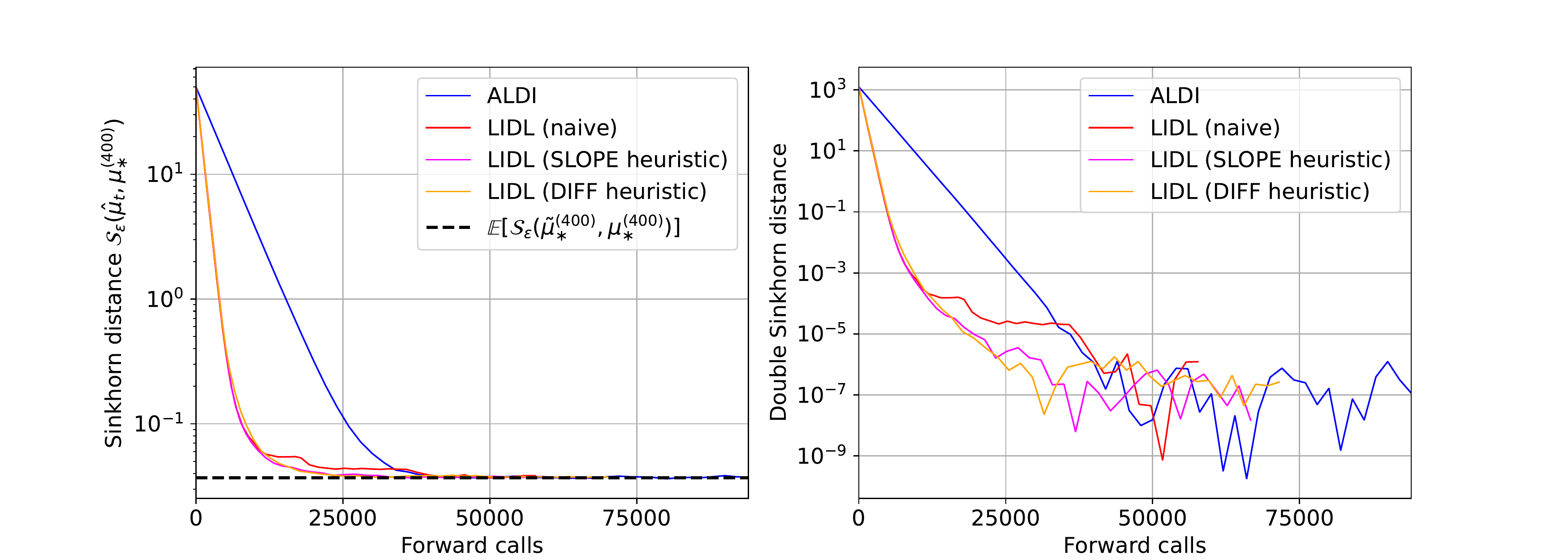}
    \caption{\textit{Difference and slope heuristics applied to the unimodal Gaussian posterior from Section \ref{section:num_unimodal_1}.
    Both heuristics clearly outperform $\mathrm{ALDI}$ as well as the naive $\mathrm{LIDL}$ implementation.}}
    \label{fig:adaptive_schemes}
\end{figure}

\subsubsection{The multimodal case: homotopy approach}
\label{sec:num_multi_homotopy}

For $K=4$ the task is more challenging since the center of the initial density \eqref{eq:num_initial} is now close to a local minimum of the potential \eqref{eq:num_potential} (see Figure~\ref{fig:prior_and_posterior}).
Figure~\ref{fig:num_stage2} showcases the particle movement for $K=4$ with ALDI-based LIDL.
In Stage 1, an auxiliary potential $\Psi$ -- in this case a zero mean Gaussian with covariance matrix $5\Sigma$ -- is used to precondition the data.
In this particular case, the auxiliary potential is used to get the particles out of the vicinity of the local minimum near $(-5,0)$.
In stage 2, LIDL is deployed with successive ensemble enrichments in $L=3$ stages and a homotopy based potential $\mathcal{H}$ interpolating between $\mathcal{H}(0)=\Psi$ and $\mathcal{H}(1)=\Phi$.

\begin{figure}
\begin{subfigure}{.23\textwidth}
  \centering
  \includegraphics[width=.8\linewidth]{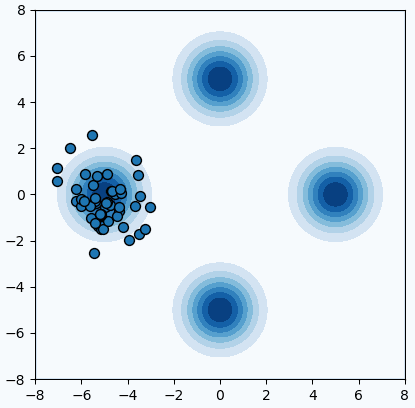}
  \caption{}
  \label{fig:sfig1}
\end{subfigure}%
\begin{subfigure}{.23\textwidth}
  \centering
  \includegraphics[width=.8\linewidth]{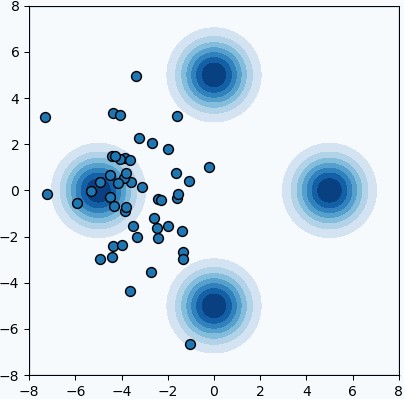}
  \caption{}
  \label{fig:sfig2}
\end{subfigure}
\begin{subfigure}{.23\textwidth}
  \centering
  \includegraphics[width=.8\linewidth]{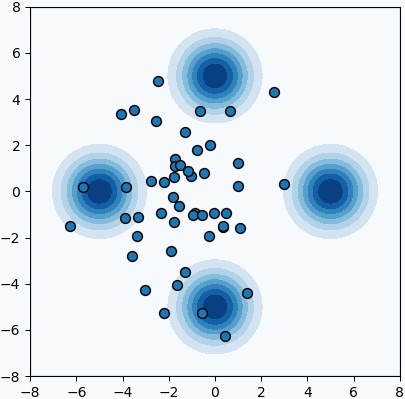}
  \caption{}
  \label{fig:sfig3}
\end{subfigure}
\begin{subfigure}{.23\textwidth}
  \centering
  \includegraphics[width=.8\linewidth]{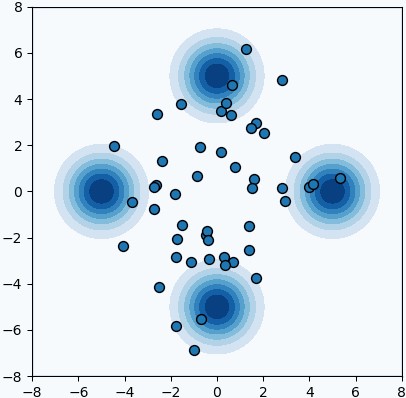}
  \caption{}
  \label{fig:sfig4}
\end{subfigure}%

\begin{subfigure}{.23\textwidth}
  \centering
  \includegraphics[width=.8\linewidth]{slides/stage14.jpg}
  \caption{}
  \label{fig:sfig5}
\end{subfigure}%
\begin{subfigure}{.23\textwidth}
  \centering
  \includegraphics[width=.8\linewidth]{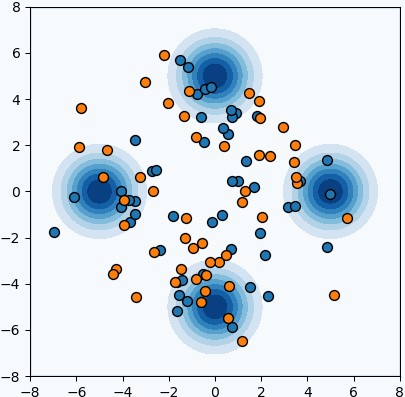}
  \caption{}
  \label{fig:sfig6}
\end{subfigure}
\begin{subfigure}{.23\textwidth}
  \centering
  \includegraphics[width=.8\linewidth]{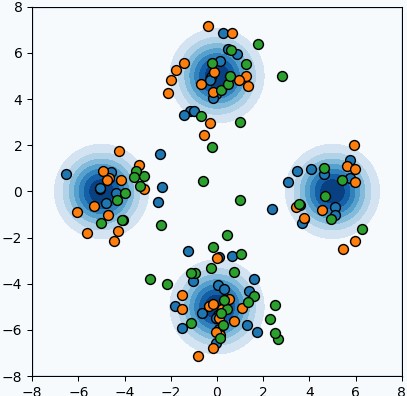}
  \caption{}
  \label{fig:sfig7}
\end{subfigure}
\begin{subfigure}{.23\textwidth}
  \centering
  \includegraphics[width=.8\linewidth]{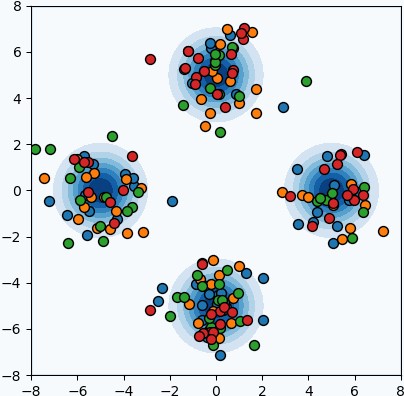}
  \caption{}
  \label{fig:sfig8}
\end{subfigure}
\caption{\textit{\textbf{Stage I (top row)}: An auxiliary potential corresponding to a centered Gaussian distribution with covariance $5I_2$ is used to precondition a batch of $b_0 = 50$ initial samples.
The particles are propagated with time moving forward from Figure (a) to (d) using $\mathrm{ALDI}$.
In (d), the samples follow the distribution associated to the auxiliary potential.
In particular, they have been moved out of the vicinity of the local minimum near $(-5,0)$.\\
\textbf{Stage II (bottom row)}: The potential used in stage I is successively replaced via linear homotopy by the true multimodal potential (depicted as a contour plot behind the particles).
The original batch (\textcolor{blue}{blue}) of $50$ particles is enriched $3$ times by $b_1 = b_2 = b_3 = 50$ particles each (\textcolor{orange}{orange}, \textcolor{green}{green}, \textcolor{red}{red}) during $\mathrm{ALDI}$ propagation, yielding a total batch of $\overline{b} = 200$ approximate posterior samples in (h).}}
\label{fig:num_stage2}
\end{figure}

\paragraph{Comparison of homotopy switch designs $s(t)$}
We now shift the focus to homotopy based approaches and compare these with a standard ALDI, which we denote as \textit{plain} ALDI.
In the experiment, $\overline{b}=200$ particles are propagated for a time horizon $T=40$ using step size $\Delta t = 0.01$ leading to a total of $800$ iterations.
We deploy $3$ different interpolation variations $s_{\mathrm{ALDI}}$ for the homotopy based approach, starting with an auxiliary zero-mean Gaussian potential $\mathcal{H}(0)$ with covariance $8I_2$ for $t\in[0,2]$ (we have chosen $5 I_2$ in Figure~\ref{fig:num_stage2} for the purpose of presentation.
In fact, a more spread out auxiliary distribution is more advantageous).
Then, for $t\in[a,b]:=[2,18]$ we use a linear ($\alpha_1 t+\alpha_2$), convex ($\beta t^4$) and concave ($1 - \gamma (t-b)^4$) speed of change for $\alpha_1,\alpha_2,\beta,\gamma\in\mathbb{R}$ such that the corresponding interpolation maps are continuous with values $s(t)=0$ at $t\leq 2$ and $s(t)=1$ at $t\geq 18$.
The corresponding $s_{\mathrm{ALDI}}$ are illustrated in dashed lines in Figure~\ref{fig:homotopy_comparison}.
Note since we are using ALDI without ensemble enrichment, the representation of $s$ with respect to time $t$ and forward calls $\mathrm{fc}$ is a linear rescaling only.
Since the posterior distribution is more spread out due the multimodality, we choose $R=2000$ runs to approximate the expectations.
The results of this experiment are illustrated in Figure~\ref{fig:homotopy_comparison}. 
As a first observation, the strategy of plain ALDI leads to very slow convergence.
The three homotopy based approaches achieve fast initial convergence due to the auxiliary potential $\mathcal{H}(0)$.
The subsequent performance depends on the corresponding homotopy switches $s_{\mathrm{ALDI}}(t)$.
We observe that the concave version performs the best, almost carrying on with the initial convergence speed for times $t > 2$.
This numerical result clearly motivates the further investigation of potential optimal choices of homotopy switch designs.

\begin{figure}
    \centering
    \includegraphics[width=0.85\linewidth]{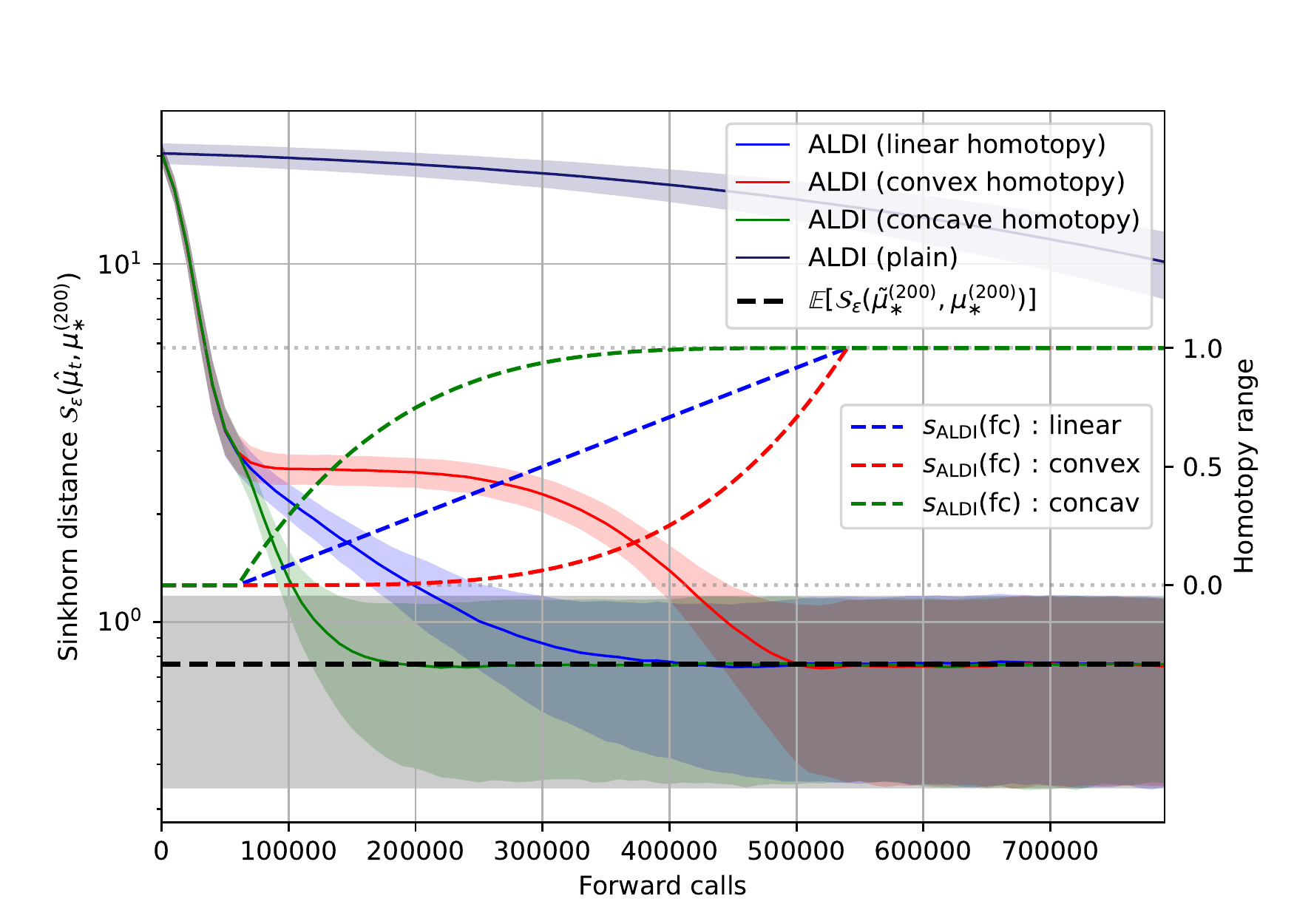}
    \caption{\textit{Comparison of different homotopy switch designs with the plain $\mathrm{ALDI}$ strategy based on $\overline{b}=200$ particles.
    While the slow convergence of plain $\mathrm{ALDI}$ underlines the potential of the homotopy approach, possible improvements of the actual switch point design becomes clear when comparing a linear, convex and concave homotopy variation.}}
    \label{fig:homotopy_comparison}
\end{figure}

From Figure~\ref{fig:homotopy_comparison} we conclude that the design of the homotopy has a significant impact on the convergence speed.
In particular, the concave case is a faster switch design compared to classical linear interpolation within the desired time horizon.
We note that in the numerical investigation this switch design can be critical if the switches towards the potential $\Phi$ are too fast.
Then, no convergence or at least a very slow convergence is observed asymptotically.
The slow convergence then matches the convergence speed of the plain non-homotopy based SDE scheme.

\paragraph{Combination of homotopy and ensemble enrichment}
As a next experiment, we investigate the interaction and the possible advantages of both, homotopy and ensemble enrichment.
Here, we consider the case of linear homotopy and concave switch design combined with $L=3$ enrichment stages.
In particular for the linear homotopy we use
\begin{equation}
\label{eq:numeric_linear_homotopy}
 s(t) = \begin{cases}
        0, & t <0.2T, \\
        \frac{10}{7T}t-\frac{2}{7}, & t\in[0.2T,0.9T],\\
        1, & t > 0.9 T,
        \end{cases}
\end{equation}
with a time horizon $T=40$.
For the concave homotopy, we use the switch design
\begin{equation}
\label{eq:numeric_concave_homotopy}
 s(t) = \begin{cases}
        0, & t <0.1T, \\
        1- \frac{1}{(0.8T)^4}(0.9T - t)^4, & t\in[0.1T,0.9T],\\
        1, & t > 0.9 T,
        \end{cases}
\end{equation}
with $T = 40 $ for ALDI and $T=60$ for LIDL.
The results of the experiments are depicted in Figures~\ref{fig:homotopy_lin_enrich} and~\ref{fig:homo_concav_enrich}.
Here, the time interval on which $s(t)\equiv 0$ is considered as free forward calls since no evaluation of the posterior potential $\Phi$ is performed.
Consequently, in both experiments the respective ALDI and LIDL schemes reach plateaus, indicating convergence to the auxiliary measure associated to $\exp(-\Psi)$. 
As an observation we note that the ensemble enrichment has an impact on the actual switch design $s_{\mathrm{LIDL}}$ as a function of function calls as defined in \eqref{eq:switch_lidl_fc}.
While $s_{\mathrm{ALDI}}(\mathrm{fc})$ is a linear rescaling of its $t$-dependent version $s_{\mathrm{ALDI}(t)}$, the effect is non-linear in the case of LIDL. 
An interesting observation can be made in the linear homotopy case.
In Figure~\ref{fig:homotopy_lin_enrich}, the enrichment scheme yields a concave (piecewise linear) switch design (red line) when interpreted as a function of function calls.
The fact that LIDL with linear homotopy turns out to require less forward calls may be explained by the faster switch speed. 
In particular, the use of smaller batch sizes for a fixed time horizon requires less forward calls spent on the auxiliary intermediate measures.
Once the maximum batch size $\overline{b}$ is reached, LIDL recovers the convergence speed of ALDI. 
We again use time step $\Delta t = 0.01$ and time horizon $T=40$ with $R=2000$ runs with a maximum number of $\overline{b}=200$ particles.
A total of $L=3$ ensemble enrichment stages $(b_0,b_1,b_2,b_3) = (20,40,60,80)$ are realized at $(t_1,t_2,t_3) = (12,15,18)$, based on forward slicing. 

The ensemble enrichment strategy is more involved in the concave homotopy setup shown in Figure~\ref{fig:homo_concav_enrich}.
Here, the construction of suitable time steps for the enrichment is not straightforward.
In particular, in the numerical investigation not every enrichment concept allowed for increased convergence speed.
For example, using the equidistant enrichment scheme from the linear homotopy case leads to a significant slow-down of convergence.
This phenomena may be explained by observing that such equidistant enrichment leads to a switch design $s_{\mathrm{LIDL}}(\mathrm{fc})$ that changes too quickly towards the posterior potential.

Consequently, the homotopy and the sample enrichment design in general need to interact.
In our experiment we construct an a priori enrichment scheme as follows.
Let $0=s_0<s_1 <s_2<s_3 < s_4=1$ be an equidistant partition of the switch range $[0,1]$.
We then define the time point for an enrichment and some $\gamma>0$ by
\begin{equation}
\label{eq:num_t_i_interact}
    t_i := \argmin  |s(t)-s_i^\gamma|,\quad i=1,2,3.
\end{equation}
In the experiment we use $\gamma=1.0$ and the batch size enrichment $b_\ell=100$ for $\ell=0,\ldots,L$.

The chosen concave switch design of order $4$ seems to be a reasonable choice.
It can be observed that significantly increasing the interpolation speed leads to a severe slow-down of the convergence speed.
This effect may be controlled by choosing smaller time steps.
While an increased switch speed may enable a good approximation of the posterior for a smaller time horizon $T$, the introduced stiffness in the SDE system may be computationally prohibitive due to the necessity of a smaller time step $\Delta t$.
Consequently, a proper automatic design of the homotopy speed and the enrichment stages should also involve an adaptive time stepping scheme.
In Figure~\ref{fig:homo_concav_enrich}, we additionally plot the current partial batch sizes.
The initial lower batch sizes lead to a faster switch (red dashed line) and increase the convergence speed locally.
However, due to the choice of $T=60$ for LIDL the switch $s_{\mathrm{LIDL}}(\mathrm{fc})$ becomes slower than $s_{\mathrm{ALDI}}(\mathrm{fc})$, leading to an expected slow-down of convergence.
In this example, a total of $L=3$ ensemble enrichment stages $(b_0,b_1,b_2,b_3) = (50,50,50,50)$ are realized at $t_i>0$ as defined in \eqref{eq:num_t_i_interact} based on diffusion propagation.
While Figure~\ref{fig:homo_concav_enrich} displays the result of one hand crafted design, it may still be far away from optimality.
This motivates the development of fully adaptive schemes that are capable to automatically construct promising switch and enrichment stage designs including time step adaptivity to adjust the switch speed.

For the enrichment schemes, we use forward slicing in the linear homotopy case and diffusion propagation for the concave setup.
In the numerical investigation, no significant difference between the schemes can be observed.

\begin{figure}
    \centering
    \includegraphics[width=\linewidth]{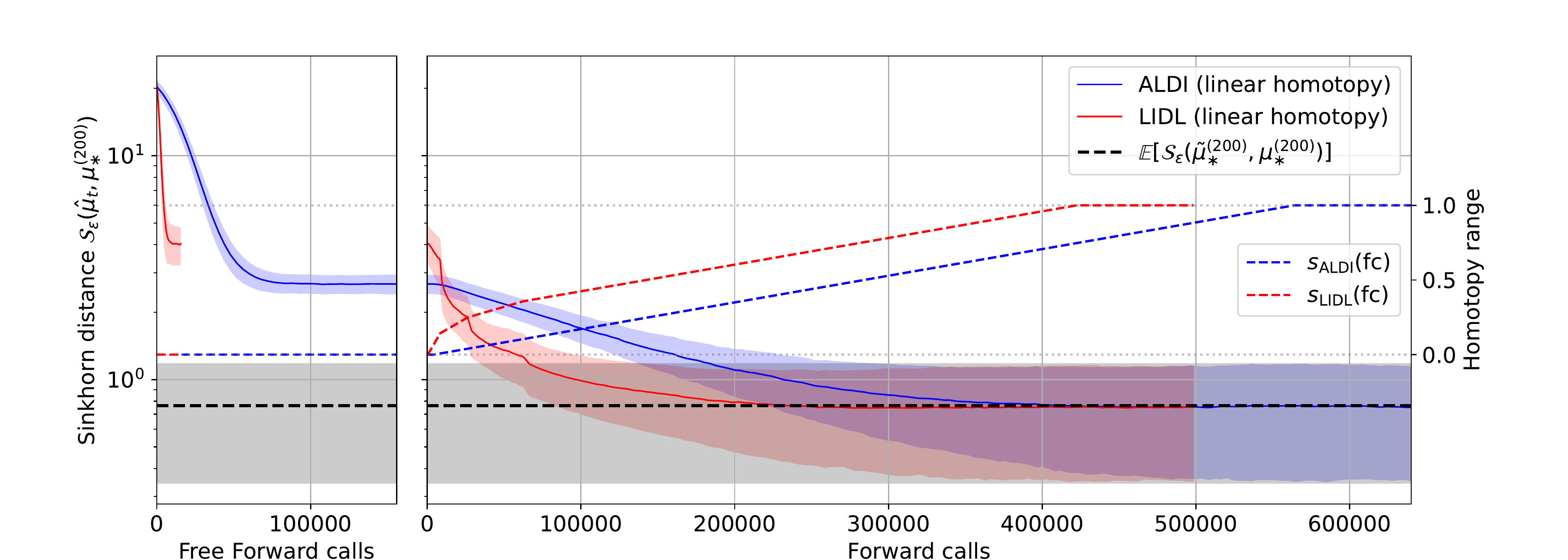}
    \caption{\textit{$\mathrm{LIDL}$ with ensemble enrichment and homotopy approach applied to the potential $\Phi$ from \eqref{eq:num_potential} for $K=4$.
    Both $\mathrm{ALDI}$ and $\mathrm{LIDL}$ use auxiliary potential $\Psi$ corresponding to a centered Gaussian distribution with covariance $8I$. 
    Here, we consider the linear homotopy defined in \eqref{eq:numeric_linear_homotopy}.
    The left subplot shows the convergence when using the potential $\mathcal{H}(s)$ for $s=0$ referred to as free forward calls since no evaluation of $\Phi$ is involved.
    The combination of linear homotopy and sample enrichment leads to faster convergence.    
}}
    \label{fig:homotopy_lin_enrich}
\end{figure}

\begin{figure}
    \centering
    \includegraphics[width=\linewidth]{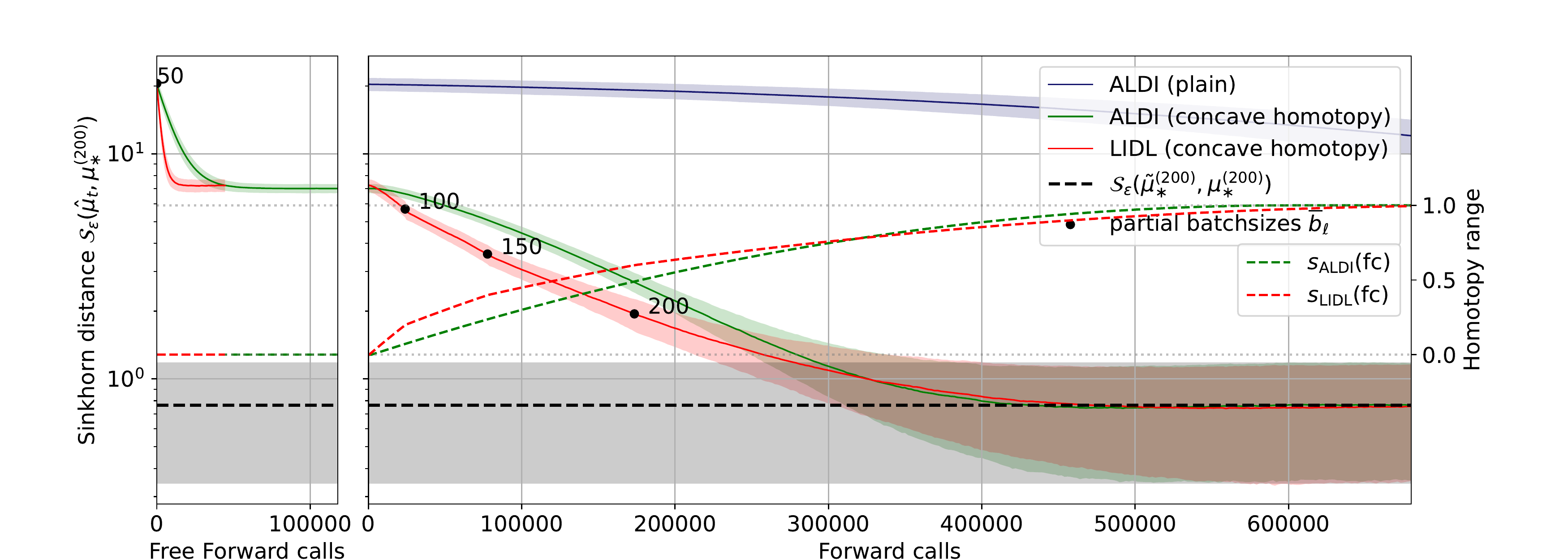}
    \caption{
    \textit{$\mathrm{LIDL}$ with ensemble enrichment and homotopy approach applied to the potential $\Phi$ from \eqref{eq:num_potential} for $K=4$ with concave switch design $s(t)$ defined in \eqref{eq:numeric_concave_homotopy}.
    Both $\mathrm{ALDI}$ and $\mathrm{LIDL}$ use auxiliary potential $\Psi$ corresponding to a centered Gaussian distribution with covariance $8I$.
    The black dots mark the partial batch sizes used in the subsequent propagation within the $\mathrm{LIDL}$ scheme up to $\overline{b}=200$ particles.
    The left subplot displays the convergence when using the potential $\mathcal{H}(s)$ for $s=0$ referred as free forward calls since no evaluation of $\Phi$ is involved.
    This particular design leads to fast initial convergence but is still suboptimal.
    }}
    \label{fig:homo_concav_enrich}
\end{figure}

\subsection{High dimensional example: parametric Darcy equation}
\label{sec:num_darcy}
In this section the one-dimensional parametric Darcy equation as discussed in \cite{garbuno2020affine} is investigated.
Consider the inverse problem of determining the permeability field $a(x)>0$ based on the solution $p$ of an elliptic PDE
\begin{equation*}
    -\partial_x (a(x)\partial_x p(x)) = f(x), \qquad \textnormal{for}\quad x\in I = [0,2\pi). 
\end{equation*}
For $K=10$ we define observation points $\hat{x}_k$ in $I$ by
$$
\hat{x}_k = \dfrac{2\pi(k-1)}{K}, \qquad k=1,\hdots,K,
$$
and consider noisy measurements $\delta_k$ of the solution $p$ at these discrete points defined by
\begin{equation*}
    \delta_k := p(\hat{x}_k) + \eta_k, 
\end{equation*}
where the measurement errors $\eta_j$ are i.i.d. Gaussians with zero mean and variance $\sigma_R = 10^{-4}$.
As in \cite{garbuno2020affine}, this infinite-dimensional problem is made finite-dimensional by introducing a computational grid
\begin{equation*}
    x_i := \dfrac{2\pi i}{D}, \qquad i=0,\hdots,D-1.
\end{equation*}
We choose $D=50$ for the computational grid $\mathcal{M}=\{x_1,\ldots,x_D\}$ and consider the finite-difference formulation
\begin{equation*}
    \dfrac{a_{i+1/2}(p_{i+1}-p_i)-a_{i-1/2}(p_i-p_{i-1})}{h^2} = -f_i,
\end{equation*}
with mesh size $h = \frac{2\pi}{D}$ and $p_i \approx p(x_i)$, $f_i = f(x_i)$ and $a_{i-1/2} = \exp(u_i)$. 
Since $\{f_i\}_{i=0}^{D-1}$ is known and fixed, the forward operator $\mathcal{G}$ maps $\{u_i\}_{i=1}^{D} \in \mathbb{R}^D$ to the restriction of the discrete solution $\{p_i\}_{i=0}^{D-1} \in \mathbb{R}^D$ to the observation grid $\{\hat{x}_j\}_{j=1}^K \in \mathbb{R}^K$.
Note that the observation grid has to be a subset of the computational grid.

In this example, we use the forcing 
\begin{equation*}
    f_i = \exp\left(\dfrac{-(2x_i-2\pi)^2}{40}\right) - \frac{3}{5}.
\end{equation*}
The prior on $u$ is defined to be centered Gaussian and covariance matrix $P_0$ defined by
\begin{equation*}
    P_0^{-1} = 4h \left( \dfrac{\mu}{D}1_D 1_D^{\intercal} - \Delta_h \right)^2,
\end{equation*}
where $\Delta_h$ is the standard finite-difference discretization of the Laplace operator with periodic boundary on $I$ with uniform mesh width $h$,
$1_D=(1,\ldots,1)^{\intercal}\in\mathbb{R}^{D}$ 
and $\mu = 10^2$ leads to a penalization of deviations of the spatial mean of $u$ away from $0$.

Having all these quantities in place, we generate observations according to
\begin{equation*}
    \delta_j = p_\ell + \eta_j, \qquad \ell = \dfrac{D}{K}j = 5j, \qquad \eta_j\sim\mathcal{N}(0,\sigma_R), \qquad j=1,\hdots,K,
\end{equation*}
were the indices $\ell$ are defined such that all observations lie on the observation grid.
The discrete observations $p_i$ on the computational grid are generated with
\begin{equation*}
    a^{\dagger}_{i-1/2} = \exp(u_i^{\dagger}), \qquad u_i^{\dagger} = \dfrac{1}{2}\sin(x_i-h/2),
\end{equation*}
for $i=1,\hdots,D$.


 
 In Figure~\ref{fig:num_darcy_lidl}, different sampling methods are applied to the Darcy problem.
 We consider EKS/ALDI setups with a total number of samples $\overline{b}=240$ as well as EKS/ALDI-based LIDL setups with  $(t_1,t_2,t_3) = (1,1.5,1.75)$ and $b_0=\ldots=b_3=60$, i.e. batch sizes only slightly above the minimum amount $b=52$ of samples necessary to ensure ALDI's ergodicity in the linear case.
 The time step is $\Delta t = 0.01$ and all methods are run up to $T=8$.
 Expectations are approximated by empirical averages over $R=70$ runs each.
 As is expected, due to the rather small batch size EKS and EKS-based LIDL do not converge to the true posterior distribution.
 This can be observed in the Sinkhorn distance plot on the left, where both methods produce an expectation slightly lower than $\mathbb{E}[\mathcal{S}_{\epsilon}(\tilde{\mu}_{\ast}^{(240)},\mu_{\ast}^{(240)})]$.
 In the right plot, we can observe a stagnation in the double Sinkhorn distance \eqref{eq:cont_doublesinkhorn} at around $10^{-5}$ for the EKS-based methods, while ALDI and ALDI-based LIDL seem to fluctuate around $10^{-8}$.
 Comparing all methods, ALDI-based LIDL seems to perform best, achieving a double Sinkhorn distance of $10^{-8}$ in roughly half the amount of forward calls ($\sim 50,000$) as ALDI ($\sim 90,000$).


\begin{figure}
    \centering
    \includegraphics[width=\textwidth]{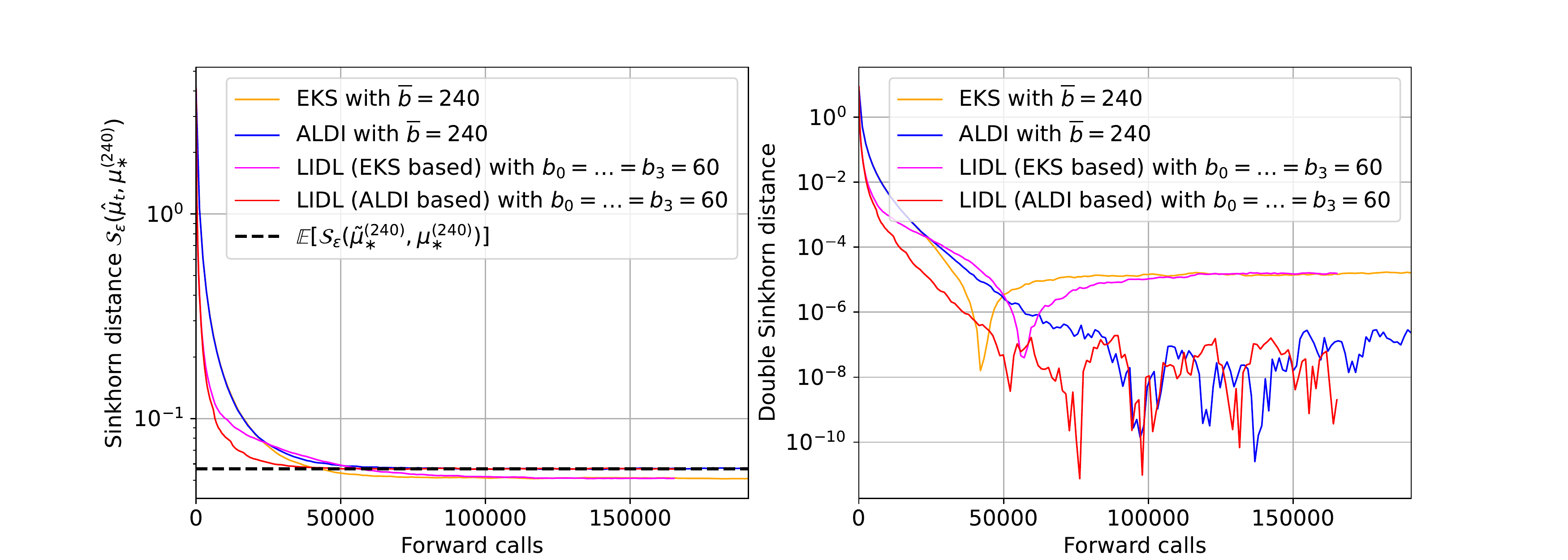}
    \caption{\textit{Performance of different samplers with $\overline{b}=240$ for the Darcy problem measured in expected Sinkhorn divergence $\mathrm{EP}_t$ (left) and double Sinkhorn distance \eqref{eq:cont_doublesinkhorn} (right).
    While $\mathrm{ALDI}$-based $\mathrm{LIDL}$ seems to converge the fastest with a double Sinkhorn distance of $10^{-8}$ after $\sim50000$ function calls, $\mathrm{EKS}$ and $\mathrm{EKS}$-based $\mathrm{LIDL}$ do not converge to the correct posterior distribution.
    Posterior samples are generated with an affine invariant Markov Chain Monte Carlo (MCMC) sampler using the Python package} \texttt{emcee} \cite{emcee}.}
    \label{fig:num_darcy_lidl}
\end{figure}

%% file: 06_conclusion.tex
\section{Conclusion and Outlook}
\label{sec:conclusion}

This work devises a strategy which can significantly reduce the computational effort in state-of-the-art ensemble samplers based on Langevin dynamics to solve Bayesian inference problems. 
A significant improvement when compared to previous techniques is achievement by the notions of intermediate ensemble enrichment and $\mathcal{W}_2$-stable homotopy maps. 

To increase the number of particles during propagation, several ensemble enrichment strategies were introduced.
These are designed with the underlying goal to preserve the distance of the ensemble distribution to the target measure.
The dynamical system is then restarted with the new set of particles as initial condition. 

The concept of homotopy was introduced to enhance (or sometimes to enable at all) the convergence to more involved target distributions.
Moreover, this framework allows to use standard estimators for covariance matrices within the dynamical system.
This is advantageous in comparison to weighted estimators as discussed in \cite{reich2021fokker}, allowing for very cheap computations of a (generalized) square root of the covariance that enters the diffusion part of the underlying SDE. 

A convergence analysis was carried out for both concepts.
In the numerical examples, the efficacy of the method was demonstrated for unimodal and multimodal posterior distributions with Gaussian tails.
Additionally, a high-dimensional and non-linear PDE problem given by the random Darcy equation was investigated.

As an outlook, we suggest two promising directions for future work:
\begin{itemize}
    \item $(t_\ell,b_\ell,\Psi,s(t))$\textbf{-adaptivity of the Sampler.} We have only presented simple heuristics for adaptivity with respect to the enrichment times $t_\ell$ in \eqref{eq:num_diff_approx} and \eqref{eq:slope_approx_average}.
    More sophisticated strategies, involving the enrichment batch sizes $b_{\ell}$ as well as adaptive adjustments of the surrogate potential $\Psi$ and homotopy switch $s(t)$ are desirable.
    While this paper proves the general feasibility of the ensemble enrichment and homotopy approaches, the question of \textit{when} to add \textit{how many} samples to facilitate fast convergence is an open research problem.
    \item \textbf{Hierarchy of approximate potentials.} The homotopy approach is based on the idea to subsequently increase the complexity of the potential, here interpreted as a preconditioning tool.
    In reality, the posterior potential $\Phi$ is not available but instead only some approximation $\Phi_h$ with $\Phi_h\to\Phi$ as $h\to0$ can be used.
    By introducing a monotone decreasing function $h = h(t)$ representing the approximation error over time, we may write the inhomogeneous drift term as 
    \begin{equation*}
        b(t, Y_t) = b(\mathcal{H}(s(t), h(t)), Y_t)), 
    \end{equation*}
    where now $\mathcal{H}(s,h)$ is the interpolation at switch point $s$ between an auxiliary potential and $\Phi_h$.
    A suitable design of $t\to (s(t),h(t))$ then potentially leads to a significantly decreased workload for reaching a desired accuracy threshold. 
\end{itemize}

%% file: 07_acknowledgements.tex
\section*{Acknowledgements}
\ownThanks

%% file: 08_appendix_theory.tex
\newcommand{\e}{\mathbf{e}}
\newcommand{\bu}{\mathbf{u}}
\newcommand{\y}{\mathbf{y}}

\section{Proof of theorem \ref{thm:B-T-convergence_ALDI}}\label{sec:proof_btconv}

The proof follows closely the proof of a comparable result for the Ensemble Kalman Sampler shown in \cite[Theorem 3.1]{DingLi21}. In the linear case, the dynamics of ALDI \eqref{eq:aldi}
become
\ifnum\classstyle=0 
{\footnotesize
   \begin{equation}
    \begin{aligned}\label{eq:lin_aldi}
    \mathrm{d}y^i_t &= -\mathrm{Cov}_{y_t,\mathcal{G}_t}\Gamma^{-1}(\mathcal{G}(y_t^i)-\tilde{\Delta})\mathrm{d}t - \mathrm{Cov}_{y_t}\Gamma_0^{-1}(y_t^i-y_0)\mathrm{d}t + \dfrac{D+1}{B}(y_t^i-\overline{y}_t)\mathrm{d}t + \sqrt{2\mathrm{Cov}_{y_t}}\mathrm{d}W_t^i \\
    &= -\mathrm{Cov}_{y_t}P(y_t^i - y^*)\mathrm{d}t + \dfrac{D+1}{B}(y_t^i-\overline{y}_t)\mathrm{d}t + \sqrt{2\mathrm{Cov}_{y_t}}\mathrm{d}W_t^i, \quad i=1,\ldots,B.
    \end{aligned}
\end{equation} 
}
\fi
\ifnum\classstyle=1 
   \begin{equation}
    \begin{aligned}\label{eq:lin_aldi}
    \mathrm{d}y^i_t &= -\mathrm{Cov}_{y_t,\mathcal{G}_t}\Gamma^{-1}(\mathcal{G}(y_t^i)-\tilde{\Delta})\mathrm{d}t - \mathrm{Cov}_{y_t}\Gamma_0^{-1}(y_t^i-y_0)\mathrm{d}t + \dfrac{D+1}{B}(y_t^i-\overline{y}_t)\mathrm{d}t + \sqrt{2\mathrm{Cov}_{y_t}}\mathrm{d}W_t^i \\
    &= -\mathrm{Cov}_{y_t}P(y_t^i - y^*)\mathrm{d}t + \dfrac{D+1}{B}(y_t^i-\overline{y}_t)\mathrm{d}t + \sqrt{2\mathrm{Cov}_{y_t}}\mathrm{d}W_t^i, \quad i=1,\ldots,B.
    \end{aligned}
\end{equation} 
\fi
\ifnum\classstyle=2 
\begin{small}
   \begin{equation}
    \begin{aligned}\label{eq:lin_aldi}
    \mathrm{d}y^i_t &= -\mathrm{Cov}_{y_t,\mathcal{G}_t}\Gamma^{-1}(\mathcal{G}(y_t^i)-\tilde{\Delta})\mathrm{d}t - \mathrm{Cov}_{y_t}\Gamma_0^{-1}(y_t^i-y_0)\mathrm{d}t + \dfrac{D+1}{B}(y_t^i-\overline{y}_t)\mathrm{d}t + \sqrt{2\mathrm{Cov}_{y_t}}\mathrm{d}W_t^i \\
    &= -\mathrm{Cov}_{y_t}P(y_t^i - y^*)\mathrm{d}t + \dfrac{D+1}{B}(y_t^i-\overline{y}_t)\mathrm{d}t + \sqrt{2\mathrm{Cov}_{y_t}}\mathrm{d}W_t^i, \quad i=1,\ldots,B.
    \end{aligned}
\end{equation} 
\end{small}
\fi
Here, we adopt the notation from \cite{DingLi21} for covariance matrices, i.e.
\begin{equation*}
\begin{aligned}
    \mathrm{Cov}_{x_t,y_t} &= \dfrac{1}{B}\sum_{i=1}^B (x_t^i-\overline{x}_t)(y_t^i-\overline{y}_t)^{\intercal}, \quad \mathrm{Cov}_{y_t} = \mathrm{Cov}_{y_t,y_t} \\
    \mathrm{Cov}_{\rho} &= \mathbb{E}_{x\sim \rho}[(x-\mathbb{E}_{\rho})(x-\mathbb{E}_{\rho})^{\intercal}], \\ \mathrm{Cov}_{y_t,\mathcal{G}_t} &= \mathrm{Cov}_{y_t}A^{\intercal},
\end{aligned}
\end{equation*}
for particle ensembles $\{x_t^i\}^B_{i=1}, \{y_t^i\}^B_{i=1}$ and densities $\rho$. For better readability, we use this form for the rest of the Appendix. We will frequently make use of It\^o's lemma, stating that for any twice differentiable function $f$ of an It\^o diffusion process $X_t$, we have
\ifnum\classstyle=0 
{\footnotesize
   \begin{equation}\label{eq:ito_lemma}
    \mathrm{d}X_t = \mu_t\mathrm{d}t + \sigma_t\mathrm{d}W_t \quad \Longrightarrow\quad \mathrm{d}f(X_t) = \left[ (\nabla f)^{\intercal}\mu_t + \dfrac{1}{2}\mathrm{Tr}[\sigma_t^{\intercal}\mathrm{Hess}_{f}\sigma_t] \right] \mathrm{d}t + (\nabla f)^{\intercal}\sigma_t \mathrm{d}W_t
\end{equation}
}
\fi
\ifnum\classstyle=1 
   \begin{equation}\label{eq:ito_lemma}
    \mathrm{d}X_t = \mu_t\mathrm{d}t + \sigma_t\mathrm{d}W_t \quad \Longrightarrow\quad \mathrm{d}f(X_t) = \left[ (\nabla f)^{\intercal}\mu_t + \dfrac{1}{2}\mathrm{Tr}[\sigma_t^{\intercal}\mathrm{Hess}_{f}\sigma_t] \right] \mathrm{d}t + (\nabla f)^{\intercal}\sigma_t \mathrm{d}W_t
\end{equation}
\fi
\ifnum\classstyle=2 
\begin{small}
   \begin{equation}\label{eq:ito_lemma}
    \mathrm{d}X_t = \mu_t\mathrm{d}t + \sigma_t\mathrm{d}W_t \quad \Longrightarrow\quad \mathrm{d}f(X_t) = \left[ (\nabla f)^{\intercal}\mu_t + \dfrac{1}{2}\mathrm{Tr}[\sigma_t^{\intercal}\mathrm{Hess}_{f}\sigma_t] \right] \mathrm{d}t + (\nabla f)^{\intercal}\sigma_t \mathrm{d}W_t
\end{equation}
\end{small}
\fi

When considering Wasserstein distances between measures, we will often identify a measure with its density, writing, for example, $\mathcal{W}_2(\hat{\mu}_{t}^B,\pi_\ast)$ instead of $\mathcal{W}_2(\hat{\mu}_{t}^B,\mu_\ast)$.

Before getting into details we will briefly sketch the main ingredients of the proof.

\subsection{Roadmap of the Proof}

The central idea is a triangle argument
\begin{equation}\label{eq:triangle}
    \mathbb{E}(\mathcal{W}_2(\pi_\ast,\hat{\mu}_{t}^B)) \leq \underbrace{\mathcal{W}_2(\pi_{\ast},\pi(t))}_{\mathrm{(I)}} + \underbrace{\mathbb{E}(\mathcal{W}_2(\pi(t),\hat{\nu}_{t}^B)}_{\mathrm{(II)}} + \underbrace{\mathbb{E}(\mathcal{W}_2(\hat{\nu}_{t}^B,\hat{\mu}_{t}^B))}_{\mathrm{(III)}},
\end{equation}
where $\pi(t) \equiv \pi(\cdot,t)$ is the solution of the limiting Fokker-Planck equation
\begin{equation}\label{eq:fp}
\begin{aligned}
    &\partial_t \pi = \nabla \cdot (\pi \text{Cov}_{\pi(t)}\nabla\Phi(y)) + \mathrm{Tr}(\text{Cov}_{\pi(t)}D^2\pi), \\
    &\pi(\cdot,0) = \pi_0(\cdot),
\end{aligned}
\end{equation}
and $\hat{\nu}_t^{B}$ is the ensemble distribution induced by the solution $\{ z^i_t \}_{i=1}^B$ of the process
\begin{equation}
\begin{aligned}
    \label{eq:ideal_sde}
    \mathrm{d}z^i_t &= -\mathrm{Cov}_{\pi(t)}\nabla\Phi(z^i_t)\mathrm{d}t + \sqrt{2\mathrm{Cov}_{\pi(t)}} \mathrm{d}W_t^i\\
    &= -\mathrm{Cov}_{\pi(t)}P(z_t^i - y^*)\mathrm{d}t  + \sqrt{2\mathrm{Cov}_{\pi(t)}}\mathrm{d}W_t^i, \qquad i = 1,\ldots, B.
\end{aligned}
\end{equation}
An intuition regarding the three terms can be given as follows. (I) concerns the distance of the posterior to the solution $\pi(t)$ of the limiting Fokker Planck equation. Since \eqref{eq:ideal_sde} produces i.i.d. samples of $\pi(t)$ at time $t$, (II) concerns the distance between $\pi$ and empirical measures of $B$ samples each, drawn from $\pi$. Finally, (III) measures the distance of the ensemble distributions defined by the ALDI right hand side to the ideal (but unavailable) ensembles defined by \eqref{eq:ideal_sde}.  
It remains to be shown that, given $\delta > 0$, we find suitable $T_{\delta}$ and $B_{T_{\delta}}$ such that the right hand side in \eqref{eq:triangle} is bounded by $\delta$. The arguments bounding the first two terms can already be found in \cite{DingLi21} and are repeated here for the sake of completeness. 

\subsection{Bounding (I)}

This term is independent of the particle ensemble $\hat{\mu}_t^{B}$ and hence of the number of particles $B$. Its behaviour is governed by the solution $\pi(t)$ of the underlying Fokker-Planck equation \eqref{eq:fp}. As an immediate consequence of the following result, the $\mathcal{W}_2$-distance between $\pi(t)$ and the posterior decays to $0$ exponentially fast. 

\begin{theorem}[Proposition 3.8 in \cite{Carrillo_2021}]\label{thm:carrillo}
    
     Let $\rho_1$ and $\rho_2$ be two solutions of the nonlinear nonlocal mean field
equation 
\begin{equation}\label{eq:carrillo_fp}
    \partial_t \rho = \nabla\left( \mathrm{Cov}_{\rho(t)} (\nabla \Phi(y)\rho + \sigma\nabla\rho) \right)
\end{equation}
with linear forward model $\mathcal{G}$. Let the covariances and expectations of the initial conditions $\rho_1(0)$ and $\rho_2(0)$ satisfy
\begin{equation}
    \begin{aligned}\label{eq:carrillo_conditions}
    &|\mathrm{Cov}_{\rho_1(0)}|_2, |\mathrm{Cov}_{\rho_2(0)}|_2, |B^{-1}|_2 \leq M, \\ 
    &|\mathrm{Cov}_{\rho_1(0)}^{-1}|_2, |\mathrm{Cov}_{\rho_2(0)}^{-1}|_2, |B|_2 \leq m, \\
    &|\mathbb{E}_{\rho_1(0)}|_2, |\mathbb{E}_{\rho_2(0)}|_2 \leq R
\end{aligned}
\end{equation}
for some constants $M,m,R>0$.
Then it holds
that
\begin{equation}\label{eq:carrillo_decay}
    W_2(\rho_1(t) , \rho_2(t) ) \leq C (1 + m^4M^4 + m^4M^{7/2}R) \dfrac{e^{-\sigma t}}{\sqrt{\alpha(t)}^{1+\lfloor 1 \wedge
    \sigma\rfloor}}W_2(\rho_1(0),\rho_2(0)),
\end{equation}
where $C$ is a constant only dependent on the dimension $d$ and
\begin{equation*}
    \alpha(t)=
\begin{cases}
2t+2, \quad & \sigma=0,\\
\sigma^{-1}(1-e^{-2\sigma t})+ e^{-2\sigma t}, \quad & \sigma >0.
\end{cases}
\end{equation*}
    
\end{theorem}

The Fokker Planck equation \eqref{eq:carrillo_fp} becomes \eqref{eq:fp} for $\sigma=1$. Now, setting $\rho_1(0) = \pi_0$ and $\rho_2(0) = \pi_{\ast}$, \eqref{eq:carrillo_decay} yields (I)$\rightarrow 0$ exponentially fast with $t\rightarrow \infty$.

\subsection{Bounding (II)}

This term is again independent of the dynamics used to sample $\hat{\mu}_t^B$ and depends only on the rate of convergence of the ensembles $\hat{\nu}_t^B$ defined by \eqref{eq:ideal_sde} to their continuous limit \eqref{eq:fp}.  We cite the corresponding result from \cite{DingLi21}.

\begin{theorem}[Proposition 5.1 in \cite{DingLi21}]
    Let $\pi$ solve the Fokker-Planck equation \eqref{eq:fp} and let $\{ z_t^i \}_{i=1}^B$ solve \eqref{eq:ideal_sde} with initial data $\{ z_{t=0}^i \}_{i=1}^B$ drawn i.i.d. from $\pi_0 \in \mathcal{C}^2$ with finite higher moments. Let $\hat{\nu}_t^B$ be the ensemble distribution defined by $\{z_t^i\}$, then for any $t>0$ and $0 < \epsilon < 1/2$, there exists a constant $C$, depending on $t$, the dimension $D$ and $\epsilon$ but not on $B$ such that
    \begin{equation}
        \begin{aligned}
        \mathbb{E}(\mathcal{W}_2(\hat{\nu}_t^B,\pi(t))) \leq C \begin{cases}
    B^{-1/2+\epsilon},& \quad D\leq 4,\\
    B^{-2/D},& \quad D > 4.
\end{cases}
    \end{aligned}
    \end{equation}
\end{theorem}
\begin{proof}
The proof uses Theorem \ref{thm:conv_emp_measure} and the boundedness of the higher moments of $\hat{\nu}_t^B$ (see Proposition 5.3 in \cite{DingLi21}).
\end{proof}

\subsection{Bounding (III)}

This is the main part of the proof, where we show that the results from \cite{DingLi21} for the EKS carry over for the modified dynamics defined by ALDI. For each result, we will cite the corresponding result for the EKS, so that the reader may compare. We start with the main result, yielding the required boundedness.

\begin{theorem}[Compare Proposition 5.2 in \cite{DingLi21}]\label{thm:ensemble_distance}
    Let $\{ y_t^i \}_{i=1}^B$ be the solution of \eqref{eq:lin_aldi} and $\{ z_t^i \}_{i=1}^B$ solve \eqref{eq:ideal_sde}, where $\{z_0^i\}_{i=1}^B = \{y_0^i\}_{i=1}^B$ are drawn i.i.d. from the distribution induced by $\pi_0 \in C^2$ with finite high moments. Let $\hat{\nu}_{t}^{B}$ and $\hat{\mu}_t^{B}$ denote the ensemble distributions defined by $\{z_t^i\}$ and $\{y_t^i\}$ respectively. Suppose \eqref{eq:ev_condition} holds true. Then for any $0 < \epsilon < 1/2$, there exists a constant $C$ depending only on $D, T$ and $\epsilon$ such that
    \begin{equation}
        \mathbb{E}(\mathcal{W}_2(\hat{\nu}_T^B,\hat{\mu}_{T}^B)) \leq \left( \dfrac{1}{B} \sum_{i=1}^B \mathbb{E}|z_T^i - y_T^i|^2 \right)^{1/2} \leq C B^{-1/2+\epsilon}.
    \end{equation}
\end{theorem}

The assertion results from the combination of several lemmas. First, we define 
\begin{equation*}
    x_t^i = y_t^i - z_t^i \qquad \text{and} \qquad p_t^i = x_t^i - \overline{x}_t, \qquad q_t^i = z_t^i - \overline{z}_t
\end{equation*} 
for convenience. We need that the higher moments of $x_t^i$ and $p_t^i$ are bounded for all time.
\begin{lemma}[Compare Lemma 5.3 in \cite{DingLi21}]\label{lem:5.3}
Under the same conditions as in Theorem \ref{thm:ensemble_distance}, for all $2\leq p < \infty$ and $T>0$, there is a constant $C_p$ independent of $B, t$ such that
\begin{equation}\label{eq:5.3_result}
    \mathbb{E}|x_t^i|^p = \mathbb{E}|x^1_t|^p \leq C_p, \qquad \mathbb{E}|p_t^i|^p = \mathbb{E}|p^1_t|^p \leq C_p
\end{equation}
for all $1\leq i \leq B$ and $0 \leq t \leq T$.
\end{lemma}

\begin{lemma}[Compare Lemma 5.4 in \cite{DingLi21}]\label{lem:5.4}
Under the same conditions as in Theorem 5, for any $0 \leq \alpha < 1$ and $T>0$, if there is a constant $\hat{C}$ independent of $B,t$ such that
\begin{equation}\label{eq:5.4_cond}
    \mathbb{E}|x_t^i|^2 \leq \hat{C}B^{-\alpha}
\end{equation}
for all $1\leq i \leq B$ and $0\leq t \leq T$, then for any $0 < \epsilon < 1/2$ and $1 \leq i \leq B$, there exists a constant $C$ independent of $B,t$ such that
\begin{equation}\label{eq:5.4_result}
    \mathbb{E}|p_t^i|^2 = \mathbb{E}\left| x_t^i - \dfrac{1}{B}\sum_{k=1}^B x_t^k \right|^2 \leq \Tilde{C}B^{-1/2-\alpha/2+\epsilon}
\end{equation}
for all $1\leq i \leq B$ and $0\leq t \leq T$.
\end{lemma}

\begin{lemma}[Compare Lemma 5.5 in \cite{DingLi21}]\label{lem:5.5}
Under the same conditions as lemma \ref{lem:5.4}, we have for any $0<\epsilon < 1/2$ and $T>0$ a constant $\Tilde{C}$ independent of $B, t$ such that 
\begin{equation}\label{eq:5.5_result}
    \mathbb{E}|x_t^i|^2 \leq \Tilde{C}B^{-1/2-\alpha/2+\epsilon}
\end{equation}
for all $1\leq i \leq B$ and $0\leq t\leq T$.
\end{lemma}

\begin{proof}[Proof of Theorem \ref{thm:ensemble_distance}]
By lemma \ref{lem:5.3}, \eqref{eq:5.4_cond} is satisfied for $\alpha_0 = 0$, yielding 
\begin{equation*}
    \mathbb{E}|x_t^i|^2 \leq \Tilde{C}B^{-1/2+\epsilon}
\end{equation*}
by lemma \ref{lem:5.5}. Hence, \eqref{eq:5.4_cond} is satisfied for $\alpha_1 = 1/2 - \epsilon$ and recursively for $\alpha_{n+1} = 1/2 + \alpha_n/2 -\epsilon$. The limit of this sequence is $\lim_{n\rightarrow \infty} \alpha_n = 1-2\epsilon$, yielding
\begin{equation*}
    \mathbb{E}|z_t^i-y_t^i|^2 = \dfrac{1}{B}\sum_{i=1}^B \mathbb{E}|z_t^i-y_t^i|^2 \leq \Tilde{C}B^{-1+2\epsilon}.
\end{equation*}
\end{proof}

It remains to show the three lemmas \ref{lem:5.3}-\ref{lem:5.5}, for which we need some more intermediate results.

    \begin{lemma}[Compare Lemma 4.1 from \cite{DingLi21}]\label{lem:4.1}
    Let $\{y_t^i\}_{i=1}^B$ be the solution of \eqref{eq:lin_aldi} with initial conditions $\{y_0^i\}_{i=1}^B$ sampled i.i.d from $\pi_0\in\mathcal{C}^2$. If the initial $p$-th moment is finite, i.e.
    \begin{equation*}
        \left( \mathbb{E}| y_0^i |^p \right)^{1/p} < M, \quad 1\leq i \leq B
    \end{equation*}
    for $p\geq 2$ and some $M>0$ independent of $B$,  then the boundedness also holds true for $\mathbb{E}|e_t^i|$, where $e_t^i=y_t^i-\overline{y}_t$, namely there is a $C>0$ depending only on $p$, so that
    \begin{equation*}
        \left( \mathbb{E}| e_t^i |^p \right)^{1/p}
        < 2(\kappa(P))^{1/2} M \exp(Ct)
    \end{equation*}
    for all $t\geq 0$ and $1\leq i \leq B$, where $\kappa(P) = \|P\|_2\|P^{-1}\|_2$ is the condition number of $P$.
    \end{lemma}

\begin{proof}[Proof of Lemma \ref{lem:4.1}]
Following the proof of lemma 4.1 in \cite{DingLi21}, we will show the claim for $2p$ with $p\geq 1$ and define
\begin{equation*}
    \mathbf{e}_t^i = \sqrt{P}e_t^i, \quad V_p(\mathbf{e}) = \dfrac{1}{B}\sum_{i=1}^B \left\langle \mathbf{e}_t^i,\mathbf{e}_{t}^i \right\rangle^p = \dfrac{1}{B}\sum_{i=1}^B  |\mathbf{e}_{t}^i|^{2p} , \quad h_p(t) = \mathbb{E}V_p = \mathbb{E}|\e^1_t|^{2p}.
\end{equation*}
Since $\lambda_{\textnormal{min}}(P) > 0$, boundedness of $\mathbb{E}| e_t^i |^p$ follows immediatley from boundedness of $h_p(t)$, which we now show.
First, we get
\begin{equation*}
\begin{aligned}
    (h_p(0))^{\frac{1}{2p}} &= \left( \mathbb{E}|\e^1_0|^{2p} \right)^{\frac{1}{2p}} \leq \| P \|^{1/2}_2 \left( \mathbb{E}|e^1_0|^{2p} \right)^{\frac{1}{2p}} \\
    &\leq \| P \|^{1/2}_2 \left( \left( \mathbb{E}|y^1_0|^{2p} \right)^{\frac{1}{2p}} + \left( \mathbb{E}|\overline{y}_0|^{2p} \right)^{\frac{1}{2p}} \right) \\
    &\leq 2\| P \|_2^{1/2} M,
\end{aligned}
\end{equation*}
where we have used the boundedness of the initial moments. Our next intermediate goal is to show
\begin{equation*}
    h_p(t) \leq h_p(0) + C(p,D) \int_{0}^t h_p(s) ds
\end{equation*}
with a constant $C(p,D)$ depending on $p$ and $D$. To that end, note that by \eqref{eq:lin_aldi} we have
\begin{equation}
\begin{aligned}
     \mathrm{d}\overline{y}_t &= -\mathrm{Cov}_{y_t}P(\overline{y}_t-y^\ast)\mathrm{d}t + \dfrac{D+1}{B}\sum_{i=1}^B (y_t^i-\overline{y}_t) + \sqrt{2\mathrm{Cov}_{y_t}}\mathrm{d}\overline{W}_t 
\end{aligned}
\end{equation}
and with $\mathrm{Cov}_{y_t} = \mathrm{Cov}_{e_t} = \sqrt{P}^{-1}\mathrm{Cov}_{\e_t}\sqrt{P}^{-1}$, we find
\ifnum\classstyle=0 
{\footnotesize
   \begin{equation}
    \begin{aligned}\label{eq:de}
    \mathrm{d}\e_{t}^i = \mathrm{d}\sqrt{P}(y_t^i-\overline{y}_t) &=  -\sqrt{P}\mathrm{Cov}_{e_t}Pe_t^i\mathrm{d}t + \dfrac{D+1}{B}(\e_t^i-\overline{\e}_t)\mathrm{d}t + \sqrt{P}\sqrt{2\mathrm{Cov}_{e_t}}\mathrm{d}(W_t^i-\overline{W}_t) \\
    &= -\mathrm{Cov}_{\e_t}\e_t^i\mathrm{d}t + \dfrac{D+1}{B}(\e_t^i-\overline{\e}_t)\mathrm{d}t + \sqrt{P}\sqrt{2\mathrm{Cov}_{e_t}}\mathrm{d}(W_t^i-\overline{W}_t).
\end{aligned}
\end{equation}
}
\fi
\ifnum\classstyle=1 
   \begin{equation}
    \begin{aligned}\label{eq:de}
    \mathrm{d}\e_{t}^i = \mathrm{d}\sqrt{P}(y_t^i-\overline{y}_t) &=  -\sqrt{P}\mathrm{Cov}_{e_t}Pe_t^i\mathrm{d}t + \dfrac{D+1}{B}(\e_t^i-\overline{\e}_t)\mathrm{d}t + \sqrt{P}\sqrt{2\mathrm{Cov}_{e_t}}\mathrm{d}(W_t^i-\overline{W}_t) \\
    &= -\mathrm{Cov}_{\e_t}\e_t^i\mathrm{d}t + \dfrac{D+1}{B}(\e_t^i-\overline{\e}_t)\mathrm{d}t + \sqrt{P}\sqrt{2\mathrm{Cov}_{e_t}}\mathrm{d}(W_t^i-\overline{W}_t).
\end{aligned}
\end{equation}
\fi
\ifnum\classstyle=2 
\begin{small}
   \begin{equation}
    \begin{aligned}\label{eq:de}
    \mathrm{d}\e_{t}^i = \mathrm{d}\sqrt{P}(y_t^i-\overline{y}_t) &=  -\sqrt{P}\mathrm{Cov}_{e_t}Pe_t^i\mathrm{d}t + \dfrac{D+1}{B}(\e_t^i-\overline{\e}_t)\mathrm{d}t + \sqrt{P}\sqrt{2\mathrm{Cov}_{e_t}}\mathrm{d}(W_t^i-\overline{W}_t) \\
    &= -\mathrm{Cov}_{\e_t}\e_t^i\mathrm{d}t + \dfrac{D+1}{B}(\e_t^i-\overline{\e}_t)\mathrm{d}t + \sqrt{P}\sqrt{2\mathrm{Cov}_{e_t}}\mathrm{d}(W_t^i-\overline{W}_t).
\end{aligned}
\end{equation}
\end{small}
\fi
Noting that
   \begin{equation}
    \begin{aligned}
    V_p(\e_t) &= \sum_{i=1}^B f(\e_t^i), 
    \end{aligned}
\end{equation}
with
\begin{equation*}
   f(x) = \dfrac{1}{B}|x|^{2p}, \quad \nabla f(x) = \dfrac{2p}{B}|x|^{2p-2}x, \quad \mathrm{Hess}_f(x) = \dfrac{2p}{B}|x|^{2p-2} I_d + \dfrac{4p(p-1)}{B}|x|^{2p-4}xx^{\intercal}, 
\end{equation*}
we apply It\^o's lemma \eqref{eq:ito_lemma} to $V_p(\e_t)$ to receive
\begin{equation}
\begin{aligned}\label{eq:dVp}
    \mathrm{d}V_p(\e_t) = &-\dfrac{2p}{B}\sum_{i=1}^B \left\langle \e_t^i, \e_t^i \right\rangle^{p-1} \left\langle  \e_t^i,\text{Cov}_{\e_t} \e_t^i \right\rangle \mathrm{d}t\\
    &+ \dfrac{2p}{B}\sum_{i=1}^B  \left\langle \e_t^i, \e_t^i \right\rangle^{p-1} \left\langle \e_t^i,\sqrt{P}\sqrt{2\text{Cov}_{e_t}}\mathrm{d}(W_t^i - \overline{W}_t)  \right\rangle \\
    &+ \dfrac{4(B-1)p(p-1)}{B^2}\sum_{i=1}^B \left\langle \e_t^i, \e_t^i \right\rangle^{p-2} \text{Tr}\left\{ (\e_t^i\otimes \e_t^i)\text{Cov}_{\e_t} \right\} \mathrm{d}t \\
    &+ \dfrac{2(B-1)p}{B^2} \sum_{i=1}^B \left\langle \e_t^i, \e_t^i \right\rangle^{p-1} \text{Tr}\left\{ \text{Cov}_{\e_t} \right\} \mathrm{d}t\\
    &+ \dfrac{2p}{B}\sum_{i=1}^B \left\langle \e_t^i, \e_t^i \right\rangle^{p-1} \left\langle  \e_t^i,\dfrac{D+1}{B}\left[ \e_t^i - \sum_{k=1}^B \e_t^k \right] \right\rangle\mathrm{d}t,
\end{aligned}
\end{equation}
for $p\geq 2$, with the third term vanishing in the case $p=1$. In order to bound these terms in expectation, note that for $l \in \mathbb{N}$ with $l \leq p$ we have by H\"older's inequality
\begin{equation}
    \begin{aligned}\label{eq:hoelder}
    \mathbb{E}\bigg[ \left\langle \e_t^j,\e_t^j \right\rangle^{p-l} \left\langle \e_t^k,\e_t^k \right\rangle^l \bigg] &\leq  \mathbb{E}\bigg[ \left\langle \e_t^j,\e_t^j \right\rangle^{p} \bigg]^{\frac{p-l}{p}}\mathbb{E}\bigg[ \left\langle \e_t^j,\e_t^j \right\rangle^{p} \bigg]^{\frac{l}{p}} \\
    &= \mathbb{E}\bigg[ \left\langle \e_t^j,\e_t^j \right\rangle^{p} \bigg].
\end{aligned}
\end{equation}

Furthermore, we will make use of the inequalities
\begin{equation}\label{eq:young}
    \langle x,y \rangle \leq \dfrac{\langle x,x \rangle + \langle y, y \rangle}{2} \quad \text{and} \quad
     \langle x,y \rangle^2 \leq \langle x,x\rangle \langle y,y\rangle  \leq \dfrac{\langle x,x \rangle^2 + \langle y, y \rangle^2}{2},
\end{equation}
holding for arbitrary elements $x,y$ of a linear space with inner product $\langle.,.\rangle$. 
Taking the expectation in \eqref{eq:dVp} and noting that the first term is always nonnegative, we find
\ifnum\classstyle=0 
{\scriptsize
   \begin{equation*}
\begin{aligned}
    h_p(t) - h_p(0) \leq
    &~\dfrac{4(B-1)p(p-1)}{B^3}\int_{0}^t\sum_{i,k=1}^B \mathbb{E}\left[ \left\langle \e_s^i, \e_s^i \right\rangle^{p-2} \left\langle \e_s^i, \e_s^k \right\rangle^2 \right] \mathrm{d}s \\
    &+ \dfrac{2(B-1)p}{B^3} \int_{0}^t \sum_{i,k=1}^B \mathbb{E}\left[ \left\langle \e_s^i, \e_s^i \right\rangle^{p-1} \left\langle \e_s^k,\e_s^k \right\rangle \right] \mathrm{d}s\\
    &+ \dfrac{2p(D+1)}{B^2}\left\{\int_0^t \sum_{i=1}^B \mathbb{E}\left[ \left\langle \e_s^i, \e_s^i \right\rangle^{p} \right]  \mathrm{d}s - \int_{0}^t\sum_{i,k=1}^B \mathbb{E}\left[ \left\langle \e_s^i, \e_s^i \right\rangle^{p-1} \left\langle \e_s^i, \e_s^k \right\rangle \right] \mathrm{d}s \right\}
    \\
    \leq 
    &~\dfrac{4(B-1)p(p-1)}{B^3}\int_{0}^t\sum_{i,k=1}^B \mathbb{E}\left[ \left\langle \e_s^i, \e_s^i \right\rangle^{p-2} \frac{\left\langle \e_s^i, \e_s^i \right\rangle^2 + \left\langle \e_s^k, \e_s^k \right\rangle^2}{2} \right] \mathrm{d}s \\
    &+ \dfrac{2(B-1)p}{B^3} \int_{0}^t \sum_{i,k=1}^B \mathbb{E}\left[ \left\langle \e_s^i, \e_s^i \right\rangle^{p-1} \left\langle \e_s^k,\e_s^k \right\rangle \right] \mathrm{d}s\\
    &+ \dfrac{2p(D+1)}{B^2}\left\{\int_0^t \sum_{i=1}^B \mathbb{E}\left[ \left\langle \e_s^i, \e_s^i \right\rangle^{p} \right]  \mathrm{d}s + \int_{0}^t\sum_{i,k=1}^B \mathbb{E}\left[ \left\langle \e_s^i, \e_s^i \right\rangle^{p-1} \frac{\left\langle \e_s^i, \e_s^i \right\rangle + \left\langle \e_s^k, \e_s^k \right\rangle}{2} \right] \mathrm{d}s \right\}
    \\
    \leq
    &~ \dfrac{4(B-1)p(p-1)}{B}\int_0^t h_p(s) \mathrm{d}s + \dfrac{2(B-1)p}{B} \int_{0}^t h_p(s) \mathrm{d}s \\
    &+ \dfrac{2p(D+1)}{B^2} \left\{B\int_0^t h_p(s) \mathrm{d}s + B^2\int_0^t h_p(s) \mathrm{d}s\right\}
    \\
    \leq
    &~ \dfrac{4(B-1)p(p-1) + 2(B-1)p + 2p(D+1) + 2Bp(D+1)}{B}\int_0^t h_p(s) \mathrm{d}s \\
    \leq &~(4p(p-1) + 2p + 2p(D+1) + 2p(D+1)) \int_0^t h_p(s) \mathrm{d}s \coloneqq C(p,D)\int_0^t h_p(s) \mathrm{d}s.
    \end{aligned}
\end{equation*}
}
\fi
\ifnum\classstyle=1 
   \begin{equation*}
\begin{aligned}
    h_p(t) - h_p(0) \leq
    &~\dfrac{4(B-1)p(p-1)}{B^3}\int_{0}^t\sum_{i,k=1}^B \mathbb{E}\left[ \left\langle \e_s^i, \e_s^i \right\rangle^{p-2} \left\langle \e_s^i, \e_s^k \right\rangle^2 \right] \mathrm{d}s \\
    &+ \dfrac{2(B-1)p}{B^3} \int_{0}^t \sum_{i,k=1}^B \mathbb{E}\left[ \left\langle \e_s^i, \e_s^i \right\rangle^{p-1} \left\langle \e_s^k,\e_s^k \right\rangle \right] \mathrm{d}s\\
    &+ \dfrac{2p(D+1)}{B^2}\left\{\int_0^t \sum_{i=1}^B \mathbb{E}\left[ \left\langle \e_s^i, \e_s^i \right\rangle^{p} \right]  \mathrm{d}s - \int_{0}^t\sum_{i,k=1}^B \mathbb{E}\left[ \left\langle \e_s^i, \e_s^i \right\rangle^{p-1} \left\langle \e_s^i, \e_s^k \right\rangle \right] \mathrm{d}s \right\}
    \\
    \leq 
    &~\dfrac{4(B-1)p(p-1)}{B^3}\int_{0}^t\sum_{i,k=1}^B \mathbb{E}\left[ \left\langle \e_s^i, \e_s^i \right\rangle^{p-2} \frac{\left\langle \e_s^i, \e_s^i \right\rangle^2 + \left\langle \e_s^k, \e_s^k \right\rangle^2}{2} \right] \mathrm{d}s \\
    &+ \dfrac{2(B-1)p}{B^3} \int_{0}^t \sum_{i,k=1}^B \mathbb{E}\left[ \left\langle \e_s^i, \e_s^i \right\rangle^{p-1} \left\langle \e_s^k,\e_s^k \right\rangle \right] \mathrm{d}s\\
    &+ \dfrac{2p(D+1)}{B^2}\left\{\int_0^t \sum_{i=1}^B \mathbb{E}\left[ \left\langle \e_s^i, \e_s^i \right\rangle^{p} \right]  \mathrm{d}s + \int_{0}^t\sum_{i,k=1}^B \mathbb{E}\left[ \left\langle \e_s^i, \e_s^i \right\rangle^{p-1} \frac{\left\langle \e_s^i, \e_s^i \right\rangle + \left\langle \e_s^k, \e_s^k \right\rangle}{2} \right] \mathrm{d}s \right\}
    \\
    \leq
    &~ \dfrac{4(B-1)p(p-1)}{B}\int_0^t h_p(s) \mathrm{d}s + \dfrac{2(B-1)p}{B} \int_{0}^t h_p(s) \mathrm{d}s \\
    &+ \dfrac{2p(D+1)}{B^2} \left\{B\int_0^t h_p(s) \mathrm{d}s + B^2\int_0^t h_p(s) \mathrm{d}s\right\}
    \\
    \leq
    &~ \dfrac{4(B-1)p(p-1) + 2(B-1)p + 2p(D+1) + 2Bp(D+1)}{B}\int_0^t h_p(s) \mathrm{d}s \\
    \leq &~(4p(p-1) + 2p + 2p(D+1) + 2p(D+1)) \int_0^t h_p(s) \mathrm{d}s \coloneqq C(p,D)\int_0^t h_p(s) \mathrm{d}s.
    \end{aligned}
\end{equation*}
\fi
\ifnum\classstyle=2 
{\footnotesize
\begin{equation*}
\begin{aligned}
    h_p(t) - h_p(0) \leq
    &~\dfrac{4(B-1)p(p-1)}{B^3}\int_{0}^t\sum_{i,k=1}^B \mathbb{E}\left[ \left\langle \e_s^i, \e_s^i \right\rangle^{p-2} \left\langle \e_s^i, \e_s^k \right\rangle^2 \right] \mathrm{d}s \\
    &+ \dfrac{2(B-1)p}{B^3} \int_{0}^t \sum_{i,k=1}^B \mathbb{E}\left[ \left\langle \e_s^i, \e_s^i \right\rangle^{p-1} \left\langle \e_s^k,\e_s^k \right\rangle \right] \mathrm{d}s\\
    &+ \dfrac{2p(D+1)}{B^2}\left\{\int_0^t \sum_{i=1}^B \mathbb{E}\left[ \left\langle \e_s^i, \e_s^i \right\rangle^{p} \right]  \mathrm{d}s - \int_{0}^t\sum_{i,k=1}^B \mathbb{E}\left[ \left\langle \e_s^i, \e_s^i \right\rangle^{p-1} \left\langle \e_s^i, \e_s^k \right\rangle \right] \mathrm{d}s \right\}
    \\
    \leq 
    &~\dfrac{4(B-1)p(p-1)}{B^3}\int_{0}^t\sum_{i,k=1}^B \mathbb{E}\left[ \left\langle \e_s^i, \e_s^i \right\rangle^{p-2} \frac{\left\langle \e_s^i, \e_s^i \right\rangle^2 + \left\langle \e_s^k, \e_s^k \right\rangle^2}{2} \right] \mathrm{d}s \\
    &+ \dfrac{2(B-1)p}{B^3} \int_{0}^t \sum_{i,k=1}^B \mathbb{E}\left[ \left\langle \e_s^i, \e_s^i \right\rangle^{p-1} \left\langle \e_s^k,\e_s^k \right\rangle \right] \mathrm{d}s\\
    &+ \dfrac{2p(D+1)}{B^2}\left\{\int_0^t \sum_{i=1}^B \mathbb{E}\left[ \left\langle \e_s^i, \e_s^i \right\rangle^{p} \right]  \mathrm{d}s + \int_{0}^t\sum_{i,k=1}^B \mathbb{E}\left[ \left\langle \e_s^i, \e_s^i \right\rangle^{p-1} \frac{\left\langle \e_s^i, \e_s^i \right\rangle + \left\langle \e_s^k, \e_s^k \right\rangle}{2} \right] \mathrm{d}s \right\}
    \\
    \leq
    &~ \dfrac{4(B-1)p(p-1)}{B}\int_0^t h_p(s) \mathrm{d}s + \dfrac{2(B-1)p}{B} \int_{0}^t h_p(s) \mathrm{d}s \\
    &+ \dfrac{2p(D+1)}{B^2} \left\{B\int_0^t h_p(s) \mathrm{d}s + B^2\int_0^t h_p(s) \mathrm{d}s\right\}
    \\
    \leq
    &~ \dfrac{4(B-1)p(p-1) + 2(B-1)p + 2p(D+1) + 2Bp(D+1)}{B}\int_0^t h_p(s) \mathrm{d}s \\
    \leq &~(4p(p-1) + 2p + 2p(D+1) + 2p(D+1)) \int_0^t h_p(s) \mathrm{d}s \coloneqq C(p,D)\int_0^t h_p(s) \mathrm{d}s.
    \end{aligned}
\end{equation*}
}
\fi
In total we get by Gronwall's inequality
\begin{equation*}
    h_p(t) \leq (2\| B \|_2^{1/2}M)^{2p} e^{C(p,D)t}
\end{equation*}
and hence
\begin{equation*}
\begin{aligned}
    \left( \mathbb{E}|e^j_t|^{2p} \right)^{1/2p} &= \left( \mathbb{E}|e^1_t|^{2p} \right)^{1/2p} = \left( \mathbb{E}|\sqrt{P}^{-1}\e^1_t|^{2p} \right)^{1/2p} \\
    &\leq  \| P^{-1} \|^{1/2}_2 \left( \mathbb{E}|\e^1_t|^{2p} \right)^{1/2p} \\
    &\leq \| P^{-1} \|^{1/2}_2 \left( (2\| P \|_2^{1/2}M)^{2p} e^{C(p,D)t} \right)^{1/2p}\\
    &= 2M\kappa(P)^{1/2}e^{C(p,D)t}.
\end{aligned}
\end{equation*}
\end{proof}

\begin{lemma}[Proposition 4.1 from \cite{DingLi21}]\label{prop:4.1}
Let the conditions of lemma \ref{lem:4.1} be satisfied for $2p$ (i.e. the $2p$-th initial moments are finite). Then we have
\begin{equation*}
    \left( \mathbb{E}|y_t^i - \overline{y}_t|^p \right)^{1/p} \leq Ce^{Ct} \quad \text{and} \quad (\mathbb{E}\| \mathrm{Cov}_{y_t} \|_2^p)^{1/p} \leq Ce^{Ct}
\end{equation*}
as well as
\begin{equation*}
    \left( \mathbb{E}| y^i_t |^p \right)^{1/p} \leq Ce^{Ce^{Ct}}
\end{equation*}
with $C>0$ being independent of $B$ and $t.$
\end{lemma}

\begin{proof}[Proof of Lemma \ref{prop:4.1}]
The first inequality follows directly from lemma \ref{lem:4.1} and H\"older's inequality. The second inequality follows from
\begin{equation*}
\begin{aligned}
    (\mathbb{E}\| \text{Cov}_{y_t} \|_2^p)^{1/p} &\leq \dfrac{1}{B}\sum_{i=1}^B \mathbb{E}\left( \| (y_t^i - \overline{y}_t) (y_t^i - \overline{y}_t)^{\intercal} \|_2^p \right)^{1/p} \\
    &= \dfrac{1}{B} \sum_{i=1}^B \left( \mathbb{E}|y^i_t - \overline{y}_t|^{2p} \right)^{1/p} \\
    &= \dfrac{1}{B} \sum_{i=1}^B \left(\left( \mathbb{E}|y^i_t - \overline{y}_t|^{2p} \right)^{1/2p}\right)^2 \\
    &\leq Ce^{Ct},
\end{aligned}
\end{equation*}
where we have again used lemma \ref{lem:4.1} in the last inequality. Similar to the proof of lemma \ref{lem:4.1}, we now set
    \begin{equation*}
    \mathbf{y}_t^i = \sqrt{P}y_t^i, \quad K_p(\mathbf{y}) = \dfrac{1}{B}\sum_{i=1}^B \left\langle \mathbf{y}_t^i,\mathbf{y}_{t}^i \right\rangle^p = \dfrac{1}{B}\sum_{i=1}^B  |\mathbf{y}_{t}^i|^{2p}
\end{equation*}
and
\begin{equation*}
    g_p(t) = \mathbb{E}K_p = \mathbb{E}|\y^1_t|^{2p}.
\end{equation*}
We will show the remaining inequality for $2p$ instead of $p$, by which H\"older's inequality yields the assertion. The dynamics of the scaled particles $\y_t^i$ is given by
\begin{equation*}
    \mathrm{d}\y_t^i = -\mathrm{Cov}_{\y_t}(\y_t^i - \y^{*})\mathrm{d}t + \dfrac{D+1}{B}(\y_t^j - \overline{\y}_t)\mathrm{d}t + \sqrt{B}\sqrt{2\mathrm{Cov}_{y_t}}\mathrm{d}W^i_t.
\end{equation*}
The proof now follows similar steps to that of lemma \ref{lem:4.1}. By It\^o's lemma we get the dynamics of $K_p(\y_t)$ and, by taking the expectation, that of $g_p(t)$, which we then bound by Gronwall's lemma. More precisely, It\^o's lemma \eqref{eq:ito_lemma} yields
\begin{equation}
    \begin{aligned}\label{eq:dKp}
        \mathrm{d}K_p(\y_t) = &-\dfrac{2p}{B}\sum_{i=1}^B \left\langle \y_t^i,\y_t^i \right\rangle^{p-1} \left\langle  \y_t^i,\mathrm{Cov}_{\y_t}(\y_t^i - \y^*) \right\rangle \mathrm{d}t \\
        &+ \dfrac{2p}{B}\sum_{i=1}^B \left\langle \y_t^i, \y_t^i \right\rangle^{p-1} \left\langle  \y_t^i , \sqrt{B}\sqrt{2\text{Cov}_{y_t}}dW^i_t \right\rangle \\
        &+ \dfrac{4p(p-1)}{B} \sum_{i=1}^B \left\langle \y_t^i, \y_t^i\right\rangle^{p-2}\text{Tr}\left\{ (\y_t^i \otimes \y_t^i) \text{Cov}_{\y_t} \right\}\mathrm{d}t \\
        &+ \dfrac{2p}{B}\sum_{i=1}^B \left\langle \y_t^i, \y_t^i \right\rangle^{p-1} \text{Tr}\{ \mathrm{Cov}_{\y_t} \} \mathrm{d}t \\
        &+ \dfrac{2p(D+1)}{B^2}\sum_{i=1}^B \left\langle \y_t^i, \y_t^i \right\rangle^{p-1} \left\langle \y_t^i, \y_t^i-\overline{\y}_t \right\rangle\mathrm{d}t,
    \end{aligned}
\end{equation}
with the last term corresponding to the ALDI correction term. Using \eqref{eq:hoelder} and \eqref{eq:young}, the expectation of this term term can be bounded according to
    \begin{equation*}
    \begin{aligned}
        &~\dfrac{2p(D+1)}{B^2}\sum_{i=1}^B \mathbb{E}\left[ \left\langle \y_t^i, \y_t^i \right\rangle^{p-1} \left\langle \y_t^i, \y_t^i-\overline{\y}_t \right\rangle\right] \\
        &\leq  \dfrac{2p(D+1)}{B} g_p(t) + \dfrac{2p(D+1)}{B^3} \sum_{i,k=1}^B \mathbb{E}\left[  \dfrac{ \langle \y_t^i,\y_t^i \rangle^{p} + \langle \y_t^i,\y_t^i \rangle^{p-1}\langle \y_t^k, \y_t^k \rangle}{2} \right] \\
        &\leq \dfrac{2p(D+1)}{B} g_p(t) + \dfrac{p(D+1)}{B} g_p(t) + \dfrac{p(D+1)}{B} g_p(t) \\
        &= C(p,D)g_p(t).
    \end{aligned}
    \end{equation*}
    The second term in \eqref{eq:dKp} vanishes in expectation. The first, third and forth terms can be bounded in the same way as in the proof of Proposition 4.1 in \cite{DingLi21}. The computations are technical but not very insightful and since the interested reader may follow the arguments in that work step by step, we will not repeat them here in detail, but simply state the result. For the first term, we get with H\"older's inequality and lemma \ref{lem:4.1}
    \begin{equation}
        \begin{aligned}
            &-\dfrac{2p}{B}\sum_{i=1}^B \mathbb{E} \left[ \left\langle \y_t^i,\y_t^i \right\rangle^{p-1} \left\langle  \y_t^i,\mathrm{Cov}_{\y_t}(\y_t^i - \y^*) \right\rangle \right]  \\
            &\leq 2p|\y^{\ast}|\left( \dfrac{1}{B} \sum_{i=1}^B \mathbb{E}\langle \y_t^i,\y_t^i \rangle^p \right)^{(p-1/2)/p}\left( \dfrac{1}{B}\sum_{i=1}^B \mathbb{E}\langle \e_t^k,\e_t^k \rangle^{2p} \right)^{1/(2p)} \\
            &\leq 2p|\y^{\ast}| g_p^{(p-1/2)/p}(t) Ce^{C(p,D)t}.
        \end{aligned}
    \end{equation}
    For the third and forth terms we receive
    \begin{equation}
        \begin{aligned}
        &\dfrac{4p(p-1)}{B} \sum_{i=1}^B \mathbb{E} \left[ \left\langle \y_t^i, \y_t^i\right\rangle^{p-2}\text{Tr}\left\{ (\y_t^i \otimes \y_t^i) \text{Cov}_{\y_t} \right\} \right] \\
        &\leq  4p(p-1)\left( \dfrac{1}{B}\sum_{i=1}^B \mathbb{E}\langle \y_t^i,\y_t^i \rangle^p \right)^{(p-1)/p} \left( \dfrac{1}{B}\sum_{k=1}^B \mathbb{E}\langle \e_t^k, \e_t^k \rangle^p \right)^{1/p} \\
        &\leq 4p(p-1)g_p^{(p-1)/p}(t)Ce^{C(p,D)t}
        \end{aligned}
    \end{equation}
    and
    \begin{equation}
         \begin{aligned}
         &\dfrac{2p}{B}\sum_{i=1}^B \mathbb{E}\left[\left\langle \y_t^i, \y_t^i \right\rangle^{p-1} \text{Tr}\{ \mathrm{Cov}_{\y_t} \} \right] \\
         &\leq 2p \left( \dfrac{1}{B}\sum_{i=1}^B \mathbb{E}\langle \y_t^i,\y_t^i \rangle^p \right)^{(p-1)/p} \left( \dfrac{1}{B}\sum_{k=1}^B \mathbb{E}\langle \e_t^k, \e_t^k \rangle^p \right)^{1/p} \\
         &\leq 2p g_p^{(p-1)/p}(t)Ce^{C(p,D)t}.
         \end{aligned}
    \end{equation}
    In total, we arrive at
    \begin{equation*}
        g'_p(t) \leq Cg_p(t) + Ce^{Ct}\left( g^{(p-1)/p}_p(t) + g_p^{(p-1/2)/p}(t) \right).
    \end{equation*}
    The result now follows from lemma \ref{lem:gronwall_general}.
\end{proof}

\begin{proof}[Proof of Lemma \ref{lem:5.3}]
The lemma results directly from boundedness of the higher moments of $y_t^i$ (Lemma \ref{prop:4.1}) and $z_t^i$ (see Proposition 5.3 in \cite{DingLi21}).
\end{proof}

\begin{proof}[Proof of Lemma \ref{lem:5.4}]
First, note that
\ifnum\classstyle=0 
{\footnotesize
\begin{equation}
    \mathbb{E}|p_t^i|^2 = \mathbb{E}|x_t^i-\overline{x}_t|^2 = \mathbb{E}(|x_t^i|^2+|\overline{x}_t|^2-2\langle x_t^i, \overline{x}_t \rangle)   = \mathbb{E}(|x_t^i|^2-|\overline{x}_t|^2) = \mathbb{E}\left(\dfrac{1}{B}\sum_{i=1}^B |x_t^i|^2-|\overline{x}_t|^2\right)
\end{equation}
}
\fi
\ifnum\classstyle=1 
\begin{equation}
    \mathbb{E}|p_t^i|^2 = \mathbb{E}|x_t^i-\overline{x}_t|^2 = \mathbb{E}(|x_t^i|^2+|\overline{x}_t|^2-2\langle x_t^i, \overline{x}_t \rangle)   = \mathbb{E}(|x_t^i|^2-|\overline{x}_t|^2) = \mathbb{E}\left(\dfrac{1}{B}\sum_{i=1}^B |x_t^i|^2-|\overline{x}_t|^2\right)
\end{equation}
\fi
\ifnum\classstyle=2 
\begin{small}
    \begin{equation}
    \mathbb{E}|p_t^i|^2 = \mathbb{E}|x_t^i-\overline{x}_t|^2 = \mathbb{E}(|x_t^i|^2+|\overline{x}_t|^2-2\langle x_t^i, \overline{x}_t \rangle)   = \mathbb{E}(|x_t^i|^2-|\overline{x}_t|^2) = \mathbb{E}\left(\dfrac{1}{B}\sum_{i=1}^B |x_t^i|^2-|\overline{x}_t|^2\right)
\end{equation}
\end{small}
\fi
due to the symmetry between the particles. Using \eqref{eq:lin_aldi} and \eqref{eq:ideal_sde}, we have
\begin{equation}
    \begin{aligned}
    \mathrm{d}x_t^i = &\left( -\mathrm{Cov}_{x_t+z_t}P (x_t^i+z_t^i) + \mathrm{Cov}_{\pi(t)} Pz_t^i \right)\mathrm{d}t + (\mathrm{Cov}_{x_t+z_t} - \mathrm{Cov}_{\pi(t)})Py^* \mathrm{d}t \\
    &+ \dfrac{D+1}{B}\left( y_t^i-\overline{y}_t \right)\mathrm{d}t+ \left( \sqrt{2\mathrm{Cov}_{x_t+z_t}}-\sqrt{2\mathrm{Cov}_{\pi(t)}} \right)\mathrm{d}W_t^i.
    \end{aligned}
\end{equation}
Now, applying It\^o's formula \eqref{eq:ito_lemma} and replacing  $\mathrm{Cov}_{\pi(t)}$ in the second and third terms with $\mathrm{Cov}_{z_t}$, we find
\ifnum\classstyle=0 
{\footnotesize
\begin{equation}
    \begin{aligned}\label{eq:d|x|mean}
    \mathrm{d}|x_t^i|^2 = &-2\left\langle x_t^i,\mathrm{Cov}_{x_t+z_t}Px_t^i \right\rangle\mathrm{d}t - 2\left\langle x_t^i, (\mathrm{Cov}_{x_t+z_t}-\mathrm{Cov}_{z_t})P z_t^i \right\rangle\mathrm{d}t \\
    &+2 \left\langle x_t^i, (\mathrm{Cov}_{x_t+z_t}-\mathrm{Cov}_{z_t}Py^{\ast})\right\rangle\mathrm{d}t + \mathrm{Tr}\left[ \left( \sqrt{\mathrm{Cov}_{x_t+z_t}}-\sqrt{\mathrm{Cov}_{\pi(t)}} \right)^2\right] \mathrm{d}t \\
    &+ 2\left\langle x_t^i, \dfrac{D+1}{B}(y_t^i-\overline{y}_t) \right\rangle\mathrm{d}t + 2\left\langle x_t^i, \left( \sqrt{2\mathrm{Cov}_{\pi(t)}} - \sqrt{2\mathrm{Cov}_{\pi(t)}} \right)\mathrm{d}W_t^i \right\rangle + R_t^i \mathrm{d}t
    \end{aligned}
\end{equation}
}
\fi
\ifnum\classstyle=1 
\begin{equation}
    \begin{aligned}\label{eq:d|x|mean}
    \mathrm{d}|x_t^i|^2 = &-2\left\langle x_t^i,\mathrm{Cov}_{x_t+z_t}Px_t^i \right\rangle\mathrm{d}t - 2\left\langle x_t^i, (\mathrm{Cov}_{x_t+z_t}-\mathrm{Cov}_{z_t})P z_t^i \right\rangle\mathrm{d}t \\
    &+2 \left\langle x_t^i, (\mathrm{Cov}_{x_t+z_t}-\mathrm{Cov}_{z_t}Py^{\ast})\right\rangle\mathrm{d}t + \mathrm{Tr}\left[ \left( \sqrt{\mathrm{Cov}_{x_t+z_t}}-\sqrt{\mathrm{Cov}_{\pi(t)}} \right)^2\right] \mathrm{d}t \\
    &+ 2\left\langle x_t^i, \dfrac{D+1}{B}(y_t^i-\overline{y}_t) \right\rangle\mathrm{d}t + 2\left\langle x_t^i, \left( \sqrt{2\mathrm{Cov}_{\pi(t)}} - \sqrt{2\mathrm{Cov}_{\pi(t)}} \right)\mathrm{d}W_t^i \right\rangle + R_t^i \mathrm{d}t
    \end{aligned}
\end{equation}
\fi
\ifnum\classstyle=2 
\begin{small}
\begin{equation}
    \begin{aligned}\label{eq:d|x|mean}
    \mathrm{d}|x_t^i|^2 = &-2\left\langle x_t^i,\mathrm{Cov}_{x_t+z_t}Px_t^i \right\rangle\mathrm{d}t - 2\left\langle x_t^i, (\mathrm{Cov}_{x_t+z_t}-\mathrm{Cov}_{z_t})P z_t^i \right\rangle\mathrm{d}t \\
    &+2 \left\langle x_t^i, (\mathrm{Cov}_{x_t+z_t}-\mathrm{Cov}_{z_t}Py^{\ast})\right\rangle\mathrm{d}t + \mathrm{Tr}\left[ \left( \sqrt{\mathrm{Cov}_{x_t+z_t}}-\sqrt{\mathrm{Cov}_{\pi(t)}} \right)^2\right] \mathrm{d}t \\
    &+ 2\left\langle x_t^i, \dfrac{D+1}{B}(y_t^i-\overline{y}_t) \right\rangle\mathrm{d}t + 2\left\langle x_t^i, \left( \sqrt{2\mathrm{Cov}_{\pi(t)}} - \sqrt{2\mathrm{Cov}_{\pi(t)}} \right)\mathrm{d}W_t^i \right\rangle + R_t^i \mathrm{d}t
    \end{aligned}
\end{equation}
\end{small}
\fi
and
\ifnum\classstyle=0 
{\footnotesize
\begin{equation}
    \begin{aligned}\label{eq:d|x|}
    \mathrm{d}|\overline{x}_t|^2 = &-2\left\langle \overline{x}_t,\mathrm{Cov}_{x_t+z_t}P\overline{x}_t \right\rangle\mathrm{d}t - 2\left\langle \overline{x}_t, (\mathrm{Cov}_{x_t+z_t}-\mathrm{Cov}_{z_t})P \overline{z}_t \right\rangle\mathrm{d}t \\
    &+2 \left\langle \overline{x}_t, (\mathrm{Cov}_{x_t+z_t}-\mathrm{Cov}_{z_t}Py^{\ast})\right\rangle\mathrm{d}t + \mathrm{Tr}\left[ \left( \sqrt{\mathrm{Cov}_{x_t+z_t}}-\sqrt{\mathrm{Cov}_{\pi(t)}} \right)^2\right] \mathrm{d}t \\
    &+ 2\sum_{i=1}^B \left\langle \overline{x}_t, \dfrac{D+1}{B^2}(y_t^i-\overline{y}_t) \right\rangle\mathrm{d}t + 2\left\langle \overline{x}_t, \left( \sqrt{2\mathrm{Cov}_{\pi(t)}} - \sqrt{2\mathrm{Cov}_{\pi(t)}} \right)\mathrm{d}\overline{W}_t \right\rangle + \overline{R}_t \mathrm{d}t,
    \end{aligned}
\end{equation}
}
\fi
\ifnum\classstyle=1 
\begin{equation}
    \begin{aligned}\label{eq:d|x|}
    \mathrm{d}|\overline{x}_t|^2 = &-2\left\langle \overline{x}_t,\mathrm{Cov}_{x_t+z_t}P\overline{x}_t \right\rangle\mathrm{d}t - 2\left\langle \overline{x}_t, (\mathrm{Cov}_{x_t+z_t}-\mathrm{Cov}_{z_t})P \overline{z}_t \right\rangle\mathrm{d}t \\
    &+2 \left\langle \overline{x}_t, (\mathrm{Cov}_{x_t+z_t}-\mathrm{Cov}_{z_t}Py^{\ast})\right\rangle\mathrm{d}t + \mathrm{Tr}\left[ \left( \sqrt{\mathrm{Cov}_{x_t+z_t}}-\sqrt{\mathrm{Cov}_{\pi(t)}} \right)^2\right] \mathrm{d}t \\
    &+ 2\sum_{i=1}^B \left\langle \overline{x}_t, \dfrac{D+1}{B^2}(y_t^i-\overline{y}_t) \right\rangle\mathrm{d}t + 2\left\langle \overline{x}_t, \left( \sqrt{2\mathrm{Cov}_{\pi(t)}} - \sqrt{2\mathrm{Cov}_{\pi(t)}} \right)\mathrm{d}\overline{W}_t \right\rangle + \overline{R}_t \mathrm{d}t,
    \end{aligned}
\end{equation}
\fi
\ifnum\classstyle=2 
\begin{small}
\begin{equation}
    \begin{aligned}\label{eq:d|x|}
    \mathrm{d}|\overline{x}_t|^2 = &-2\left\langle \overline{x}_t,\mathrm{Cov}_{x_t+z_t}P\overline{x}_t \right\rangle\mathrm{d}t - 2\left\langle \overline{x}_t, (\mathrm{Cov}_{x_t+z_t}-\mathrm{Cov}_{z_t})P \overline{z}_t \right\rangle\mathrm{d}t \\
    &+2 \left\langle \overline{x}_t, (\mathrm{Cov}_{x_t+z_t}-\mathrm{Cov}_{z_t}Py^{\ast})\right\rangle\mathrm{d}t + \mathrm{Tr}\left[ \left( \sqrt{\mathrm{Cov}_{x_t+z_t}}-\sqrt{\mathrm{Cov}_{\pi(t)}} \right)^2\right] \mathrm{d}t \\
    &+ 2\sum_{i=1}^B \left\langle \overline{x}_t, \dfrac{D+1}{B^2}(y_t^i-\overline{y}_t) \right\rangle\mathrm{d}t + 2\left\langle \overline{x}_t, \left( \sqrt{2\mathrm{Cov}_{\pi(t)}} - \sqrt{2\mathrm{Cov}_{\pi(t)}} \right)\mathrm{d}\overline{W}_t \right\rangle + \overline{R}_t \mathrm{d}t,
    \end{aligned}
\end{equation}
\end{small}
\fi
with 
\begin{equation}
\begin{aligned}
    R_t^j &= 2\left\langle x_t^i, (\mathrm{Cov}_{\pi(t)}-\mathrm{Cov}_{z_t})Pz_t^i \right\rangle - 2\left\langle x_t^i, (\mathrm{Cov}_{\pi(t)}-\mathrm{Cov}_{z_t})P y^* \right\rangle, \\
    \overline{R}_t &= 2\left\langle \overline{x}_t, (\mathrm{Cov}_{\pi(t)}-\mathrm{Cov}_{z_t})P\overline{z}_t \right\rangle - 2\left\langle \overline{x}_t, (\mathrm{Cov}_{\pi(t)}-\mathrm{Cov}_{z_t})P y^* \right\rangle.
\end{aligned}
\end{equation}
Combining \eqref{eq:d|x|} and \eqref{eq:d|x|mean}, we finally arrive at
\ifnum\classstyle=0 
{\footnotesize
\begin{equation}
    \begin{aligned}\label{eq:d|x|-d|x|mean}
    \mathrm{d}\left( \dfrac{1}{B} \sum_{i=1}^B |x_t^i|^2 -  |\overline{x}_t|^2 \right) = &-\dfrac{2}{B}\sum_{i=1}^B \left\langle p_t^i,\mathrm{Cov}_{p_t+q_t}Pp_t^i \right\rangle\mathrm{d}t - \dfrac{2}{B}\sum_{i=1}^B \left\langle p_t^i, (\mathrm{Cov}_{p_t+q_t}-\mathrm{Cov}_{q_t})Pq^i_t \right\rangle \mathrm{d}t \\
    &+ 2\left( 1-\dfrac{1}{B} \right)  \mathrm{Tr}\left[\left( \sqrt{\mathrm{Cov}_{x_t+z_t}}-\sqrt{\mathrm{Cov}_{\pi(t)}} \right)^2\right]\mathrm{d}t + \left( \dfrac{1}{B}\sum_{i=1}^B R_t^i - \overline{R}_t \right)\mathrm{d}t\\
    &+ \frac{D+1}{B^2}\sum_{i=1}^B\left[
2\left\langle  x_t^i, \left( y_t^i-\overline{y}_t \right) \right\rangle -
    2\left\langle  \overline{x}_t, \frac{1}{B}\sum_{i=1}^B\left( y_t^i-\overline{y}_t \right) \right\rangle \right]\mathrm{d}t \\
    &+ \dfrac{2}{B}\sum_{i=1}^B \left\langle \left( x_t^i-\overline{x}_t \right), \left( \sqrt{2\mathrm{Cov}_{x_t+z_t}}-\sqrt{2\mathrm{Cov}_{\pi(t)}} \right) \mathrm{d}\left( W_t^i - \overline{W}_t \right)\right\rangle.
    \end{aligned}
\end{equation}
}
\fi
\ifnum\classstyle=1 
\begin{equation}
    \begin{aligned}\label{eq:d|x|-d|x|mean}
    \mathrm{d}\left( \dfrac{1}{B} \sum_{i=1}^B |x_t^i|^2 -  |\overline{x}_t|^2 \right) = &-\dfrac{2}{B}\sum_{i=1}^B \left\langle p_t^i,\mathrm{Cov}_{p_t+q_t}Pp_t^i \right\rangle\mathrm{d}t - \dfrac{2}{B}\sum_{i=1}^B \left\langle p_t^i, (\mathrm{Cov}_{p_t+q_t}-\mathrm{Cov}_{q_t})Pq^i_t \right\rangle \mathrm{d}t \\
    &+ 2\left( 1-\dfrac{1}{B} \right)  \mathrm{Tr}\left[\left( \sqrt{\mathrm{Cov}_{x_t+z_t}}-\sqrt{\mathrm{Cov}_{\pi(t)}} \right)^2\right]\mathrm{d}t + \left( \dfrac{1}{B}\sum_{i=1}^B R_t^i - \overline{R}_t \right)\mathrm{d}t\\
    &+ \frac{D+1}{B^2}\sum_{i=1}^B\left[
2\left\langle  x_t^i, \left( y_t^i-\overline{y}_t \right) \right\rangle -
    2\left\langle  \overline{x}_t, \frac{1}{B}\sum_{i=1}^B\left( y_t^i-\overline{y}_t \right) \right\rangle \right]\mathrm{d}t \\
    &+ \dfrac{2}{B}\sum_{i=1}^B \left\langle \left( x_t^i-\overline{x}_t \right), \left( \sqrt{2\mathrm{Cov}_{x_t+z_t}}-\sqrt{2\mathrm{Cov}_{\pi(t)}} \right) \mathrm{d}\left( W_t^i - \overline{W}_t \right)\right\rangle.
    \end{aligned}
\end{equation}
\fi
\ifnum\classstyle=2 
{\footnotesize
\begin{equation}
    \begin{aligned}\label{eq:d|x|-d|x|mean}
    \mathrm{d}\left( \dfrac{1}{B} \sum_{i=1}^B |x_t^i|^2 -  |\overline{x}_t|^2 \right) = &-\dfrac{2}{B}\sum_{i=1}^B \left\langle p_t^i,\mathrm{Cov}_{p_t+q_t}Pp_t^i \right\rangle\mathrm{d}t - \dfrac{2}{B}\sum_{i=1}^B \left\langle p_t^i, (\mathrm{Cov}_{p_t+q_t}-\mathrm{Cov}_{q_t})Pq^i_t \right\rangle \mathrm{d}t \\
    &+ 2\left( 1-\dfrac{1}{B} \right)  \mathrm{Tr}\left[\left( \sqrt{\mathrm{Cov}_{x_t+z_t}}-\sqrt{\mathrm{Cov}_{\pi(t)}} \right)^2\right]\mathrm{d}t + \left( \dfrac{1}{B}\sum_{i=1}^B R_t^i - \overline{R}_t \right)\mathrm{d}t\\
    &+ \frac{D+1}{B^2}\sum_{i=1}^B\left[
2\left\langle  x_t^i, \left( y_t^i-\overline{y}_t \right) \right\rangle -
    2\left\langle  \overline{x}_t, \frac{1}{B}\sum_{i=1}^B\left( y_t^i-\overline{y}_t \right) \right\rangle \right]\mathrm{d}t \\
    &+ \dfrac{2}{B}\sum_{i=1}^B \left\langle \left( x_t^i-\overline{x}_t \right), \left( \sqrt{2\mathrm{Cov}_{x_t+z_t}}-\sqrt{2\mathrm{Cov}_{\pi(t)}} \right) \mathrm{d}\left( W_t^i - \overline{W}_t \right)\right\rangle.
    \end{aligned}
\end{equation}
}
\fi
The expectation of the last term vanishes by the properties of Brownian motion. The first four terms have the same form as in equation (5.22) in \cite{DingLi21} and can be treated in a similar way. Note that the terms are not identical to the terms in that work, since the underlying process $y_t$ and hence $x_t, p_t, q_t$ are different. They can however be treated in a completely analogous way, tracing the arguments one by one. Hence, we will only state the resulting bounds here, and refer to equations (5.22)-(5.27) in \cite{DingLi21} for details. The first two terms can be bounded in expectation according to
\begin{equation}
    \begin{aligned}
    &-\dfrac{2}{B}\mathbb{E}\left[ \sum_{i=1}^B \left\langle p_t^i,\mathrm{Cov}_{p_t+q_t}Pp_t^i \right\rangle\mathrm{d}t + \sum_{i=1}^B \left\langle p_t^i, (\mathrm{Cov}_{p_t+q_t}-\mathrm{Cov}_{q_t})Pq^i_t \right\rangle \mathrm{d}t \right] \\
    \leq &-\lambda_{\mathrm{min}}(P) \mathbb{E}\left( \| \mathrm{Cov}_{x_t+z_t} - \mathrm{Cov}_{z_t}\|^2_F \right) + 3\| P\|_F \mathrm{Var}(\pi(t))\mathbb{E}|p_t^1|^2 + CB^{-1/2-\alpha(1-\epsilon)}.
\end{aligned}
\end{equation}
For the third term, we get
\ifnum\classstyle=0 
{\footnotesize
\begin{equation}
    \begin{aligned}
    \mathbb{E}\left[ \mathrm{Tr}\left[\left( \sqrt{\mathrm{Cov}_{x_t+z_t}}-\sqrt{\mathrm{Cov}_{\pi(t)}} \right)^2\right] \right] \leq \lambda_{0}(t)^{-1} \mathbb{E}\| \mathrm{Cov}_{p_t+q_t}-\mathrm{Cov}_{q_t} \|_F^2 + C_{\epsilon}B^{-1/2-\alpha/2+\epsilon\alpha/4}
    \end{aligned}
\end{equation}
}
\fi
\ifnum\classstyle=1 
\begin{equation}
    \begin{aligned}
    \mathbb{E}\left[ \mathrm{Tr}\left[\left( \sqrt{\mathrm{Cov}_{x_t+z_t}}-\sqrt{\mathrm{Cov}_{\pi(t)}} \right)^2\right] \right] \leq \lambda_{0}(t)^{-1} \mathbb{E}\| \mathrm{Cov}_{p_t+q_t}-\mathrm{Cov}_{q_t} \|_F^2 + C_{\epsilon}B^{-1/2-\alpha/2+\epsilon\alpha/4}
    \end{aligned}
\end{equation}
\fi
\ifnum\classstyle=2 
{\small
\begin{equation}
    \begin{aligned}
    \mathbb{E}\left[ \mathrm{Tr}\left[\left( \sqrt{\mathrm{Cov}_{x_t+z_t}}-\sqrt{\mathrm{Cov}_{\pi(t)}} \right)^2\right] \right] \leq \lambda_{0}(t)^{-1} \mathbb{E}\| \mathrm{Cov}_{p_t+q_t}-\mathrm{Cov}_{q_t} \|_F^2 + C_{\epsilon}B^{-1/2-\alpha/2+\epsilon\alpha/4}
    \end{aligned}
\end{equation}
}
\fi
with a constant $C_{\epsilon}$ depending only on $\epsilon$ and $\lambda_0(t) = (\lambda_{\min}(\mathrm{Cov}_{x_t+z_t})^{1/2} + \lambda_{\min}(\mathrm{Cov}_{\pi(t)})^{1/2})^{2}$ coming from the Ando-Hemmen inequality (see Theorem 6.2 in \cite{higham2008}).






Finally, for the forth term, we get
\begin{equation}
    \begin{aligned}
    \mathbb{E}\left[ \dfrac{1}{B}\sum_{i=1}^B R_t^i - \overline{R}_t  \right] \leq C_{\varepsilon}B^{-1/2-\alpha/2+\epsilon\alpha/4},
    \end{aligned}
\end{equation}
again with a constant $C_{\epsilon}$ depending on $\epsilon$.
The fifth term comes from the ALDI correction term. To bound it, we note that
\begin{equation*}
\begin{aligned}
    \left\langle x_t^i, \left(  y_t^i-\overline{y}_t\right) \right\rangle &\leq \mathbb{E}|x_t^i|\left( \mathbb{E}|y_t^i| + \mathbb{E}|\overline{y}_t| \right) \\
    &\leq 2 (\mathbb{E}|x_t^i|^2)^{1/2} \mathbb{E}(|y_t^i|^2)^{1/2} \\
    &\leq C_1
\end{aligned}
\end{equation*}
with a constant $C_1$ independent of $B,t$, where we have used H\"older's inequality in the second inequality and Lemmas \ref{prop:4.1} and \ref{lem:5.3} in the last inequality. Similarly, we get 
\begin{equation*}
    \left\langle  \overline{x}_t, \frac{1}{B}\sum_{i=1}^B\left( y_t^j-\overline{y}_t \right) \right\rangle  \leq \frac{2}{B} \mathbb{E}|x_t^i| \sum_{i=1}^B \mathbb{E}|y_t^i| = 2 \mathbb{E}|x_t^i| \mathbb{E}|y_t^i| \leq C_1,
\end{equation*}
to arrive at
\ifnum\classstyle=0 
{\footnotesize
\begin{equation*}
\begin{aligned}
    &\mathbb{E}\left[ \frac{D+1}{B^2}\sum_{i=1}^B\left[
2\left\langle  x_t^i, \left( y_t^i-\overline{y}_t \right) \right\rangle -
    2\left\langle  \overline{x}_t, \frac{1}{B}\sum_{i=1}^B\left( y_t^i-\overline{y}_t \right) \right\rangle \right] \right] \leq \dfrac{2(D+1)}{B^2} B C_1 = 2(D+1)C_1 B^{-1}.
\end{aligned}
\end{equation*}
}
\fi
\ifnum\classstyle=1 
\begin{equation*}
\begin{aligned}
    &\mathbb{E}\left[ \frac{D+1}{B^2}\sum_{i=1}^B\left[
2\left\langle  x_t^i, \left( y_t^i-\overline{y}_t \right) \right\rangle -
    2\left\langle  \overline{x}_t, \frac{1}{B}\sum_{i=1}^B\left( y_t^i-\overline{y}_t \right) \right\rangle \right] \right] \leq \dfrac{2(D+1)}{B^2} B C_1 = 2(D+1)C_1 B^{-1}.
\end{aligned}
\end{equation*}
\fi
\ifnum\classstyle=2 
{\small
\begin{equation*}
\begin{aligned}
    &\mathbb{E}\left[ \frac{D+1}{B^2}\sum_{i=1}^B\left[
2\left\langle  x_t^i, \left( y_t^i-\overline{y}_t \right) \right\rangle -
    2\left\langle  \overline{x}_t, \frac{1}{B}\sum_{i=1}^B\left( y_t^i-\overline{y}_t \right) \right\rangle \right] \right] \leq \dfrac{2(D+1)}{B^2} B C_1 = 2(D+1)C_1 B^{-1}.
\end{aligned}
\end{equation*}
}
\fi
Hence, taking the expectation in \eqref{eq:d|x|-d|x|mean}, we find
\begin{equation}
    \begin{aligned}\label{eq:dEpdt}
    \dfrac{\mathrm{d}\mathbb{E}|p_t^1|^2}{\mathrm{d}t} \leq &3\| P \|_2\mathrm{Var}(\pi(t))\mathbb{E}|p_t^1|^2 - \left(  \lambda_{\min}(B)-2\lambda_0^{-1} \right) \mathbb{E}\| \mathrm{Cov}_{p_t+q_t}-\mathrm{Cov}_{q_t} \|_F^2 \\
    &+ CB^{-1/2-\alpha(1-\varepsilon)} + C_{\epsilon}B^{-1/2-\alpha/2+\epsilon\alpha/4} + + 2(D+1)C_1B^{-1}.
    \end{aligned}
\end{equation}
The second term is always negative by assumption \eqref{eq:ev_condition}. Furthermore we note that since $0\leq \alpha < 1$ and $0<\epsilon< 1/2$ all exponents of $B$ appearing in \eqref{eq:dEpdt} are smaller than $-1/2-\alpha/2 +\epsilon$, leading to
\begin{equation}
    \begin{aligned}\label{eq:dEp}
    \dfrac{\mathrm{d}\mathbb{E}|p_t^1|^2}{\mathrm{d}t} \leq &3\| P \|_2\mathrm{Var}(\pi(t))\mathbb{E}|p_t^1|^2 + CB^{-1/2-\alpha/2+\epsilon}.
    \end{aligned}
    \end{equation}
Now, by equation (2.2) in \cite{Carrillo_2021} we have
\begin{equation}
    \begin{aligned}
    \mathrm{Cov}_{\pi(t)} = (1-e^{-2t})P^{-1} + e^{-2t}\mathrm{Cov}_{\pi_0}^{-1}
    \end{aligned}
\end{equation}
and hence $\mathrm{Var}(\pi(t)) \leq M$ for some $M>0$. Integrating \eqref{eq:dEp} leads to
\begin{equation}
    \begin{aligned}
        \mathbb{E}|p_t^1|^2 \leq C\int_0^t \mathbb{E}|p_s^1|^2 ds + CtB^{-1/2-\alpha/2+\epsilon}
    \end{aligned}
\end{equation}
and applying lemma \ref{lem:gronwall_standard} yields
\begin{equation}
    \begin{aligned}
     \mathbb{E}|p_t^1|^2 \leq CtB^{-1/2-\alpha/2+\epsilon} + \int_0^t CsB^{-1/2-\alpha/2+\epsilon}e^{C(t-s)}\mathrm{d}s = C(t) B^{-1/2-\alpha/2+\epsilon},
    \end{aligned}
\end{equation}
finishing the proof.

\end{proof}

\begin{proof}[Proof of Lemma \ref{lem:5.5}]
The proof is the same as that of Lemma 5.5 in \cite{DingLi21}.
\end{proof}

\newpage

\section{Gronwall inequalities}

\begin{lemma}[Gronwall \cite{gronwall1919note}]\label{lem:gronwall_standard}
Let $u(t)$ be a nonnegative function satisfying
\begin{equation*}
    u(t) \leq \beta(t) + \int_0^t c(t)u(t) \mathrm{d}t \qquad \mathrm{for~all} \quad t\geq 0,
\end{equation*}
where $\beta$ and $c$ are continuous nonnegative functions for $t\geq 0$. Then
\begin{equation*}
    u(t) \leq \beta(t)+\int_0^t \beta(s)c(s)e^{(\int_s^t c(r)\mathrm{d}r)}\mathrm{d}s \quad \mathrm{for~ all} \quad t \geq 0. 
\end{equation*}
\end{lemma}
\begin{proof}
See e.g. Theorem 1 in \cite{dragomir2002some}.
\end{proof}

\begin{lemma}[Perov, \cite{perov1959k} or Theorem 21 in \cite{dragomir2002some}]\label{lem:gronwall_bernoulli}
Let $u(t)$ be a nonnegative function satisfying 
\begin{equation*}
    u'(t) \leq c(t)u(t) + a(t)u(t)^{\alpha},
\end{equation*}
where $0\leq \alpha < 1$ and $c$ and $a$ are countinous nonnegative functions for $t>0$. Then we have
\begin{small}
  \begin{equation*}
     u(t) \leq \left[u^{1-\alpha}(0)e^{(1-\alpha)\int_0^t c(s)ds} + e^{(1-\alpha)\int_0^t c(s)ds}\int_0^t (1-\alpha)a(s) e^{(\alpha-1)\int_0^s a(r) dr} ds\right]^{\frac{1}{1-\alpha}}.
 \end{equation*} 
 \end{small}

\end{lemma}
\begin{proof}
Let $v(t)$ solve 
\begin{equation*}
    v'(t) = c(t)v(t) + a(t)v(t)^{\alpha}, \quad v(0) = u(0).
\end{equation*}
This is a classic Bernoulli-type ODE, which can be solved by noting that the ODE becomes linear in $v^{(1-\alpha)}$:
\begin{equation*}
    (v^{1-\alpha})'(t) = (1-\alpha)c(t)v^{1-\alpha}(t) + (1-\alpha)a(t).
\end{equation*}
 Applying the variation of parameters to $v^{1-\alpha}$ we get
 \begin{small}
  \begin{equation*}
     v(t) = \left[u^{1-\alpha}(0)e^{(1-\alpha)\int_0^t c(s)ds} + \int_0^t (1-\alpha)a(s) e^{(1-\alpha)\int_s^t c(r) dr} ds\right]^{\frac{1}{1-\alpha}}.
 \end{equation*} 
 \end{small}
 The assertion now immediately follows, since $u(t) \leq v(t)$ by construction.
\end{proof}

\begin{lemma}\label{lem:gronwall_general}
Let $u(t)$ be a nonnegative function satisfying
\begin{equation*}
    u'(t) \leq c(t) u(t)  + a(t) u^{\alpha}(t) + b(t) u^{\beta}(t),
\end{equation*}
where $0<\alpha<\beta\leq 1$ and $c, a$ and $b$ are continuous and nonnegative functions for $t\geq 0$. Then $u(t)$ is bounded for all $t\geq 0$ by
\begin{small}
\begin{equation}
    u(t) \leq \left[u^{1-\alpha}(0)e^{(1-\alpha)\int_0^t c(s)+b(s)ds} + \int_0^t (1-\alpha)(a(s)+b(s)) e^{(1-\alpha)\int_s^t c(r)+b(r) dr} ds\right]^{\frac{1}{1-\alpha}}.
\end{equation}
\end{small}
\end{lemma}
\begin{proof}
For $u(t) < 1$ we have $u^{\beta}(t) \leq u^{\alpha}(t)$ and hence
    \begin{equation*}
    \begin{aligned}
        u'(t) &\leq c(t) u(t)  + (a(t)+b(t)) u^{\alpha}(t) \\
        &\leq (c(t)+b(t)) u(t)  + (a(t)+b(t)) u^{\alpha}(t).
    \end{aligned}
    \end{equation*}
    For $u(t) \geq 1$, we have $u^{\beta}(t) \leq u(t)$ and hence
        \begin{equation*}
        \begin{aligned}
            u'(t) &\leq (c(t)+b(t)) u(t)  + a(t) u^{\alpha}(t)\\
        &\leq (c(t)+b(t)) u(t)  + (a(t)+b(t)) u^{\alpha}(t).
        \end{aligned}
    \end{equation*}
    In total, we arrive at $u'(t) \leq (c(t)+b(t)) u(t)  + (a(t)+b(t)) u^{\alpha}(t)$ for all $t\geq 0$ and can use lemma \ref{lem:gronwall_bernoulli}.
\end{proof}

\newpage
\section{Existence and Uniqueness of a solution of ALDI}
    
Existence and uniqueness of a strong solution by means of construction of a Lyapunov function of the system has already been proven in the original work on ALDI \cite{garbuno2020affine}. We provide an alternative proof, using the same Lyapunov function that is used in \cite{DingLi21} for existence and uniqueness of the EKS solution, merely for the interest of the reader.

\begin{theorem}[Compare Thm 3.2 in \cite{DingLi21}]\label{thm:aldi_existence}
    Suppose $G$ is linear and $(y^i_0)_{i=1}^B$ are i.i.d. Then for all $t\geq 0$ there exists a unique strong solution $(y^i_t)_{i=1}^B$ (up to $\mathbb{P}$-indistinguishability) of the set of coupled SDEs defined by \eqref{eq:aldi}. 
\end{theorem}
    
\begin{proof}{Proof of Theorem \ref{thm:aldi_existence}}
We consider the stacked SDE
\begin{equation*}
    \mathrm{d}Y_t = F(Y_t)\mathrm{d}t + G(Y_t)\mathrm{d}W_t,
\end{equation*}
where $Y_t = (y_t^i)_{i=1}^B \in \mathbb{R}^{DB}$, $W_t = \left( W_t^i \right)_{i=1}^B$ and
\begin{equation*}
\begin{aligned}
    F(Y_t) &= \left( -\text{Cov}_{y_t}P(y_t^i-u^*) + \dfrac{D+1}{B}(y_t^j-\overline{y}_t)\right)_{i=1}^B \in \mathbb{R}^{DB}, \\
    G(Y_t) &= \text{diag}\left( \sqrt{2\text{Cov}_{y_t}} \right)_{i=1}^B \in \mathbb{R}^{DB\times DB}, \\
\end{aligned}
\end{equation*}
where $\text{diag}(D_i)_{i=1}^B$ is a block diagonal matrix with entries $(D_i)_i^B$ on the diagonal.

    We prove the assertion by showing existence of a Lyapunov function of the system, i.e. a function $V\in\mathcal{C}^2(\mathbb{R}^{DB},\mathbb{R})$ such that 
\begin{itemize}
    \item[(1)] there exists a $c > 0$ so that
    \begin{equation*}
        \mathcal{L}V(Y) \coloneqq \nabla V(Y)\cdot F(Y) + \frac{1}{2}\text{tr}\left[ G(Y)^{\intercal} \text{Hess}_{V}(Y) G(Y) \right] \leq c V(Y)
    \end{equation*}
    for all $Y\in\mathbb{R}^{DB}$.
    \item[(2)] we have $$\inf_{|Y| > R} V(Y) \longrightarrow \infty $$
    as $R \longrightarrow \infty$.
\end{itemize}
The function we choose is the same that is used in \cite{DingLi21} for the EKS:
\begin{equation*}
    V(Y) = V_1(Y) + V_2(Y) = \dfrac{1}{B}\sum_{i=1}^B |y^j-\overline{y}|^2 + |\overline{y}-y^*|_P^2.
\end{equation*}
Towards (1), note that
    \begin{equation*}
    \begin{aligned}
        \nabla V_1(Y) &= \left( \frac{2}{B}(y^j - \overline{y}) \right)_{i=1}^B, \\
        \nabla V_2(Y) &= \left( 2P(\overline{y}-y^*) \right)_{i=1}^B
    \end{aligned}
    \end{equation*}
    and that the diagonal blocks of $\text{Hess}_{V_2}(Y)$ and $\text{Hess}_{V_1}(Y)$ are given by $$\frac{2}{B}\text{diag}(P)_{i=1}^B$$ 
    and 
    $$\frac{2}{B}\left( 1-\frac{1}{B} \right)\text{diag}\left(  I_d\right)_{i=1}^B$$ respectively.

    This yields
    \begin{equation*}
    \begin{aligned}
        \nabla V_1(Y)\cdot F(Y) &= -\frac{2}{B}\sum_{i=1}^B \left\langle y^i - \overline{y}, \text{Cov}_y B(y^i-y^*)\right\rangle + \frac{2(D+1)}{B^2}\sum_{i=1}^B \left\langle y^i - \overline{y}, y^i - \overline{y}\right\rangle \\
        &= -\frac{2}{B}\sum_{i=1}^B \left\langle y^i - \overline{y}, \text{Cov}_y B(y^i-\overline{y})\right\rangle + \frac{2(D+1)}{B}V_1(Y) \leq  \frac{2(D+1)}{B}V_1(Y),
    \end{aligned}
    \end{equation*}
    where we have used that both $\text{Cov}_y$ and $B$ are positive semi-definite, as well as
    \begin{equation*}
    \begin{aligned}
        \nabla V_2(Y)\cdot F(Y) &= -2\sum_{i=1}^B \left\langle P(\overline{y}-y^*), \text{Cov}_y P(y^j-y^*)\right\rangle + \frac{2(D+1)}{B}\sum_{i=1}^B \left\langle P(\overline{y} - y^*), y^i - \overline{y}\right\rangle,  \\
        &= -2B \left\langle P(\overline{y}-y^*), \text{Cov}_y P(\overline{y}-y^*)\right\rangle + 2(D+1) \left\langle P(\overline{y} - y^*), \overline{y} - \overline{y}\right\rangle \leq 0.
    \end{aligned}
    \end{equation*}
    Furthermore, we have
    \begin{equation*}
    \begin{aligned}
        \frac{1}{2}\text{Tr}\left[G(Y)^{\intercal}\text{Hess}_{V_1}(Y)G(Y)\right] &= \frac{1}{2}\text{Tr}\left[\text{Hess}_{V_1}(Y)G(Y)G(Y)^{\intercal}\right] = \text{Tr}\left[\text{Hess}_{V_1}(Y)\text{diag}(\text{Cov}_y)_{i=1}^B\right] \\
        &= \frac{2}{B}\sum_{i=1}^B \left( 1-\frac{1}{B} \right) \text{Tr}\left[\text{Cov}_y\right] = 2 \left( 1-\frac{1}{B} \right) \text{Tr}\left[\text{Cov}_y\right] \\
        &= 2 \left( 1-\frac{1}{B} \right)\dfrac{1}{B} \sum_{i=1}^B (y^i - \overline{y})^{\intercal} (y^i - \overline{y}) = 2\left( 1-\dfrac{1}{B} \right) V_1(Y)
    \end{aligned}
    \end{equation*}
  and
  \begin{equation*}
  \begin{aligned}
      \frac{1}{2}\text{Tr}\left[G(Y)^{\intercal}\text{Hess}_{V_2}(Y)G(Y)\right] &= \text{Tr}\left[\text{Hess}_{V_2}(Y)\text{diag}(\text{Cov}_y)_{i=1}^B\right] 
      = 2 \text{Tr}[P\text{Cov}_y] \\
      &= \dfrac{2}{B}\sum_{i=1}^B (y^i-\overline{y})^{\intercal} P (y^i-\overline{y}) \leq 2 \|P\| V_1(Y).
  \end{aligned}
  \end{equation*}
In total, we get
\begin{equation*}
    \mathcal{L}V(Y) \leq \left( \dfrac{2(D+1)}{B} + 2\left(1-\dfrac{1}{B}\right) + 2\| P \| \right)V(Y)
\end{equation*}
for all $Y$, proving (1). Now, towards (2) assume there is a sequence $(Y_n)_{n\in\mathbb{N}}$ in $\mathbb{R}^{DB}$ with  $V(Y_n)<M$ for some $M>0$ and all $n$. By the construction of $V$, we have
\begin{equation*}
    |y_n^i-\overline{y}_n| < \sqrt{MB}, \quad |\overline{y}_n - y^*| < \sqrt{M}
\end{equation*}
and hence
\begin{equation*}
    |Y_n|^2 = \sum_{i=1}^B |y_n^i|^2 = \sum_{i=1}^B |y_n^i - \overline{y}_n + \overline{y}_n - y^* + y^*|^2 < \sum_{i=1}^B \left( \sqrt{M}(\sqrt{B}+1) + |y^*| \right)^2
\end{equation*}
for all $n$, meaning the sequence $(|Y_n|)_{n\in\mathbb{N}}$ is also bounded. 
\end{proof}

\newpage
\section{Wasserstein spaces and convergence}
For $p\geq 1$ let $\mathcal{D}_p(X)$ denote the set of probability measures over a Banach space $X$ with norm $\|.\|$  that have finite $p$-th moment. On $\mathcal{D}_p(X)$ we define the $p$-Wasserstein metric as
$$
 \mathcal{W}_p(\mu,\nu)^p = \inf\limits_{\pi\in \Pi_p(\mu,\nu)} \int\limits_{X\times X}\| x-y\|^p\mathrm{d}\pi(x,y), 
$$
with $\Pi_p(\mu,\nu)=\{\pi\in\mathcal{D}_p(X\times X)\,|\, \pi(\cdot\times X)=\mu, \pi(X\times \cdot)=\nu\}$.
Furthermore for $X=\mathbb{R}^D$, we denote with $F_p^k(X)$ with $k\in\mathbb{N}\cup\{\infty\}$ the set of regular measures with $p$-th moment and a Lebesque density in $\mathcal{C}^k(X)$, in particular
$$
F_p^k=\{\mu\in \mathcal{D}_p(\mathbb{R}^D) | \exists f:=\mathrm{d}\mu/\mathrm{d}x \text{ and } f \in\mathcal{C}^k(\mathbb{R}^D)\}.
$$

\begin{lemma} \label{lemma:dense_F_p_k_W_p}
Let $p\geq 1$. Then, for any $k\in\mathbb{N}\cup\{\infty\}$ the subset $F_p^k\subset\mathcal{D}_p$ is dense in $\mathcal{D}_p$ with respect to the $\mathcal{W}_p$ metric.
\end{lemma}
\begin{proof}
Consider a non-negative convolution kernel $\rho\in\mathcal{C}_{c}^\infty(\mathbb{R}^D)$ with $\|\rho\|_{1}=1$, where $\|\cdot\|_1$ denotes the $L^1$-norm over $\mathbb{R}^D$. Then define $\rho_\epsilon(x)=\epsilon^{-D}\rho(\epsilon^{-1}x)$ and 
$$
    f_\epsilon(x) = (\rho_\epsilon * \mu)(x)=\int\limits_{\mathbb{R}^D}\rho_\epsilon(x-y)\mathrm{d}\mu(y).
$$
Then $f_\epsilon \in \mathcal{C}^\infty(\mathbb{R}^d)$ is a Lebesque density defining a measure $\mu_\epsilon$ on $\mathbb{R}^D$.
Now consider the transport plan or coupling
$$
\pi_\epsilon(\mathrm{d}x,\mathrm{d}y) = \rho_\epsilon(x-y)\mathrm{d}\mu(y)\mathrm{d}x.
$$
Then $\pi_\epsilon\in \mathcal{W}_2(\mathbb{R}^D\times\mathbb{R}^D)$ with marginal projections $\mu_\epsilon(\cdot\times \mathbb{R}^D)=\mu_\epsilon$ and $\pi_\epsilon(\mathbb{R}^D\times \cdot)=\mu$. Then it holds
\begin{equation*}
\begin{aligned}
    \mathcal{W}_p(\mu,\mu_\epsilon)^p&\leq \int\limits_{\mathbb{R}^D\times\mathbb{R}^D}|x-y|^p\mathrm{d}\pi_\epsilon \\
&= \int\limits_{\mathbb{R}^D}\int\limits_{\mathbb{R}^D} |x-y|^p\rho_\epsilon(x-y)\mathrm{d}\mu(y)\mathrm{d}x \\
&=
\int\limits_{\mathbb{R}^D}\int\limits_{\mathbb{R}^D}\epsilon^p|z|^p\rho(z)\mathrm{d}\mu(y)\mathrm{d}z \\
&= \epsilon^p \int\limits_{\mathbb{R}^D} |z|^p\rho(z)\mathrm{d}z,
\end{aligned}
\end{equation*}
where we used the substitution $z=\epsilon^{-1}(x-y)$. Now, clearly, $\mathcal{W}_p(\mu,\mu_\epsilon) \to 0$ as $\epsilon\to 0$.
\end{proof}

\begin{theorem}[Theorem 1 in \cite{fournier2015rate}]\label{thm:conv_emp_measure}
Let $\mu$ be a probability measure on $\mathbb{R}^{D}$ and let $\mu^{(B)} = \frac{1}{B}\sum_{i=1}^B \delta_{y^i}$ be an empirical measure of i.i.d. samples $\{y^i\}_{i=1}^B$ drawn from $\mu$. Let $p > 0$ and assume that 
\begin{equation}
    M_q(\mu) \coloneqq \int_{\mathbb{R}^{D}} |y|^{q} \mathrm{d}\mu(y) < \infty \qquad \textrm{for some} \quad q > p.
\end{equation}
Then, there exists a constant $C$ depending only on $p, d, q$ such that, for all $N \geq 1$,
\ifnum\classstyle=0 
{\footnotesize
\begin{equation}
\begin{aligned}
      \mathbb{E}[(\mathcal{W}_p(\hat{\mu}^{(B)} , \mu))] \leq 
C M^{p/q}_q (\mu) \cdot \begin{cases}
B^{-1/2} + B^{-(q- p)/q} &\quad \textrm{if} \quad p > d/2 \quad\textrm{and}\quad q = 2p,\\
B^{-1/2} \log(1 + B) + B^{-(q- p)/q} &\quad \textrm{if} \quad p = d/2 \quad\textrm{and}\quad q = 2p,\\
B^{-p/d} + B^{-(q- p)/q} &\quad \textrm{if}\quad p \in (0, d/2) \quad\textrm{and}\quad q = d/(d - p). 
\end{cases}
\end{aligned}
\end{equation}
}
\fi
\ifnum\classstyle=1 
\begin{equation}
\begin{aligned}
      \mathbb{E}[(\mathcal{W}_p(\mu^{(B)} , \mu))] \leq 
C M^{p/q}_q (\mu) \cdot \begin{cases}
B^{-1/2} + B^{-(q- p)/q} &\quad \textrm{if} \quad p > d/2 \quad\textrm{and}\quad q = 2p,\\
B^{-1/2} \log(1 + B) + B^{-(q- p)/q} &\quad \textrm{if} \quad p = d/2 \quad\textrm{and}\quad q = 2p,\\
B^{-p/d} + B^{-(q- p)/q} &\quad \textrm{if}\quad p \in (0, d/2) \quad\textrm{and}\quad q = d/(d - p). 
\end{cases}
\end{aligned}
\end{equation}
\fi
\ifnum\classstyle=2 
{\footnotesize
\begin{equation}
\begin{aligned}
      \mathbb{E}[(\mathcal{W}_p(\mu^{(B)} , \mu))] \leq 
C M^{p/q}_q (\mu) \cdot \begin{cases}
B^{-1/2} + B^{-(q- p)/q} &\quad \textrm{if} \quad p > d/2 \quad\textrm{and}\quad q = 2p,\\
B^{-1/2} \log(1 + B) + B^{-(q- p)/q} &\quad \textrm{if} \quad p = d/2 \quad\textrm{and}\quad q = 2p,\\
B^{-p/d} + B^{-(q- p)/q} &\quad \textrm{if}\quad p \in (0, d/2) \quad\textrm{and}\quad q = d/(d - p). 
\end{cases}
\end{aligned}
\end{equation}
}
\fi

\end{theorem}